\documentclass[a4paper,10pt]{amsart}

\usepackage{amsmath, amssymb, amsthm, xcolor, mathtools}
\usepackage[shortlabels]{enumitem}
\usepackage{verbatim}
\usepackage[hidelinks]{hyperref}
\usepackage{graphicx}
\usepackage[normalem]{ulem}
\usepackage{tikz-cd} 
\usetikzlibrary{matrix}

\usepackage{cancel}
\usepackage{enumitem}

\usepackage[american]{babel}

\let\oldtocsubsection=\tocsubsection
\renewcommand{\tocsubsection}[2]{\hspace{.75cm}\oldtocsubsection{#1}{#2}}

\DeclareRobustCommand{\SkipTocEntry}[5]{}

\newtheorem{theorem}{Theorem}[section]
\newtheorem{lemma}[theorem]{Lemma}
\newtheorem{proposition}[theorem]{Proposition}
\newtheorem{corollary}[theorem]{Corollary}

\newtheorem*{claim}{Claim}
\newtheorem*{theorem*}{Theorem}

\theoremstyle{definition}
\newtheorem{definition}[theorem]{Definition}

\newtheorem{remark}[theorem]{Remark}

\newtheorem*{remark*}{Remark}

\newenvironment{claimproof}[1][\proofname]{\begin{proof}[#1]}{\end{proof}}

\newcommand{\Z}{\mathcal{Z}}
\newcommand{\R}{\mathbb{R}}
\newcommand{\M}{\mathcal{M}}
\newcommand{\K}{\mathcal{K}}

\newcommand{\Q}{\mathbb{Q}}
\newcommand{\C}{\mathbb{C}}
\newcommand{\N}{\mathbb{N}}

\renewcommand{\phi}{\varphi}

\newcommand{\e}{\epsilon}

\newcommand{\Cu}{\operatorname{{Cu}}}

\newcommand{\Aff}{\operatorname{{Aff}}}
\newcommand{\tr}{\operatorname{{tr}}}
\newcommand{\eps}{\epsilon}
\newcommand{\completion}[2]{\overline{#1}{}^{#2}}

\newcommand{\cstar}{\texorpdfstring{$C^*$}{C*}}
\newcommand{\wstar}{\texorpdfstring{$W^*$}{W*}}
\newcommand{\IIone}{\texorpdfstring{II$_1$}{II-1}}

\newcommand{\safeK}{\texorpdfstring{$K$}{K}}
\newcommand{\JiangSu}{\texorpdfstring{$\mathcal{Z}$}{Z}}

\makeatletter
\def\namedlabel#1#2{\begingroup
    #2%
    \def\@currentlabel{#2}%
    \label{#1}\endgroup
}
\makeatother

\usepackage{perpage}
\newcounter{mparcnt}
\MakePerPage{mparcnt}

\usepackage{xcolor}

\numberwithin{equation}{section}

\title{Regular simple nuclear $C^*$-algebras}
\author[]{Samuel Evington}
\address{Samuel Evington, School of Mathematics and Statistics, University College Dublin, Belfield, Dublin 4, D04 V1W8, Ireland}
\email{sam.evington@ucd.ie}

\author[]{Shanshan Hua}
\address{Shanshan Hua, Mathematical Institute, University of M\"unster, Einsteinstr.\ 62, 48149 M\"unster, Germany}
\email{shua@uni-muenster.de}

\author[]{Christopher Schafhauser}
\address{Christopher Schafhauser, Department of  Mathematics, University of Ne-\linebreak braska-Lincoln, Lincoln, Nebraska 68508, USA}
\email{cschafhauser2@unl.edu}

\author[]{Apurva Seth}
\address{Apurva~Seth, Mathematical Institute, University of Oxford, Oxford OX2 6GG, United Kingdom.}
\email{apurva.seth@maths.ox.ac.uk}

\author[]{Aaron Tikuisis}
\address{Aaron Tikuisis, Department of Mathematics and Statistics, University of
  Ottawa, 585 King Edward, Ottawa, ON, K1N 6N5, Canada}
\email{aaron.tikuisis@uottawa.ca}

\author[]{Stuart White}
\address{Stuart White, Mathematical Institute, University of Oxford,
  Oxford, OX2 6GG, United Kingdom}
\email{stuart.white@maths.ox.ac.uk}

\thanks{This work was funded by: the Deutsche Forschungsgemeinschaft (DFG, German Research Foundation) under Germany’s Excellence Strategy EXC 2044-390685587, Mathematics Münster: Dynamics–Geometry–Structure, by the SFB 1442 of the DFG (SE, SH); the Engineering and Physical Sciences Research Council [EP/X026647/1] (AS, SW); NSF grant DMS-2400178 (CS), an NSERC Discovery Grant (AT). SW thanks OpenAI for providing access to ChatGPT through their ChatGPT for Academic Researchers programme.  For the purpose of open access, the authors have applied a CC-BY license to any author accepted manuscript arising from this submission. } 

\begin{document}

\begin{abstract}
    We prove $\mathcal Z$-stability for all simple, separable, pure, nuclear $C^*$-algebras. Combining this with very recent work of Thiel completes the proof of the Toms--Winter conjecture. 

    We do this through the framework of II$_1$ factorial tracially complete $C^*$-algebras. We characterise regularity for these algebras in various terms. In particular they are regular if and only if they have real rank zero and the Murray--von Neumann equivalence classes of projections correspond to continuous affine functions from the  designated traces to $[0,1]$. As an application of our analysis of ultrapowers of tracial completions, we show simple pure $C^*$-algebras whose quasitraces are traces are $K$-stable.
    
    In the amenable setting we give a local quantisation result, \`a la Popa, for II$_1$ factorial amenable tracially complete $C^*$-algebras, which obtains hyperfiniteness from regularity conditions. Our $\Z$-stability results follow from this together with a theorem of Matui and Sato.
    
    An initial version of an essential lemma in the Bauer trace simplex case was obtained from ChatGPT Astra.
\end{abstract}

\maketitle

\tableofcontents
\renewcommand{\thetheorem}{\Alph{theorem}}

\section*{Introduction} 

\subsection*{The Toms--Winter conjecture} In operator algebras, as with many areas of mathematics, structure and classification go hand in glove.  In their foundational work on what we now call von Neumann algebras, Murray and von Neumann obtained a powerful uniqueness theorem in the 1940s for II$_1$ factors with a very strong internal finite-dimensional approximation structure known as \emph{hyperfiniteness} (\cite{MvN.4}).  While this theorem is extremely elegant, the required internal structure is hard to recognise explicitly in examples, and the power of the theorem wasn't fully unleashed until Alain Connes characterised hyperfiniteness abstractly in terms of \emph{injectivity} over 30 years later (\cite{Connes:Annals}).
Crucially, injectivity is readily verifiable in examples: a discrete group $G$ has an injective von Neumann algebra if and only if $G$ is amenable.  Through the marriage of Connes' structure theorem and Murray and von Neumann's uniqueness result, the von Neumann algebras of all countable discrete amenable groups with infinite non-trivial conjugacy classes are isomorphic.
In \cite{Haagerup:InjectiveHyperfinite,Popa:InjectiveHyperfinite}, Haagerup and Popa respectively provided insightful new proofs of Connes' result which bypass many intermediate steps.

The Elliott classification programme for $C^*$-algebras aims for structure and classification results for simple separable nuclear $C^*$-algebras which are analogous to the Connes--Murray--von Neumann uniqueness theorem, and the Connes--Haagerup classification of separably-acting injective type III factors (\cite{Connes:Annals,Haagerup:Acta}).  Early results for $C^*$-algebras with internal finite-dimensional approximations go back to \cite{Glimm:UHF,Bratteli:AF,Elliott.AF}, but the programme only started in earnest with Elliott's results for stably finite algebras in the early 90s (\cite{Elliott:Crelle}) and his ICM talk (\cite{Elliott:ICM}), together with the Kirchberg--Phillips theorem on the purely infinite side (\cite{Kirchberg:preprint,Phillips:Doc}; see also \cite{Rordam:Book,Gabe:MAMS,BouwenGabe}). 
A definitive classification theorem is now available also for stably finite $C^*$-algebras, with final steps combining \cite{EGLN,GLN1,GLN2,TWW,Winter:AJM,Winter:Crelle,CETWW}. An alternative approach containing a detailed history is given in \cite{CGSTW}.
 
While the analogies between von Neumann algebraic and $C^*$-algebraic classification run deep, there are two fundamental ways in which the $C^*$-algebra case is different.  
Firstly -- though not the subject of this paper -- one needs Rosenberg and Schochet's universal coefficient theorem (UCT) from \cite{RS:UCT} to control $KK$-theory in terms of $K$-theory. Indeed, it remains a major challenge to determine whether all separable nuclear $C^*$-algebras satisfy the UCT. Secondly -- and very much the point of this paper -- work of Villadsen, R{\o}rdam, and Toms (\cite{villadsen1998simple,Villadsen:JAMS,rordam_acta,Andrew_classification}) demonstrated that simple nuclear $C^*$-algebras can exhibit highly exotic behaviour. Therefore, an additional \emph{regularity} hypothesis is needed for classification.  

The Toms--Winter conjecture stems from Toms and Winter's detailed analysis in \cite{TomsWinter:JFA} of Villadsen's first family of examples in 2007, containing both classifiable and non-classifiable algebras. For these algebras, they found that conditions of very different natures -- finite nuclear dimension, $\mathcal Z$-stability, and having almost unperforated Cuntz semigroup -- are equivalent, and they conjectured that this equivalence holds for all simple separable nuclear non-elementary $C^*$-algebras.  That is, there should be one notion of regularity giving rise to classification, which should be accessible from different directions.  

The Toms--Winter conjecture became a central objective for the field in the 2010s, with the equivalence of \ref{tw1}, \ref{tw1prime} and \ref{tw2} in Theorem \ref{Intro:Structure} below being established across \cite{Winter:I1,winter2012nuclear,MS:Duke,SWW,BBSTWW,CETWW,CE:APDE}.  The implications \ref{tw2}$\Rightarrow$\ref{tw3prime}$\Rightarrow$\ref{tw3} really go back to R\o{}rdam in \cite{Rordam:IJM} (although condition \ref{tw3prime} was not identified explicitly until much later). Earlier this week, Hannes Thiel gave a proof that \ref{tw3}$\Rightarrow$\ref{tw3prime} (\cite{Thiel:arXiv}). Here, we provide the final step in the puzzle by proving \ref{tw3prime}$\Rightarrow$\ref{tw2}, completing the Toms--Winter conjecture which is now a unified structural theorem for simple nuclear $C^*$-algebras.

\begin{theorem}[The structure theorem for simple separable nuclear $C^*$-algebras]\label{Intro:Structure}
Let $A$ be a simple separable  nuclear non-elementary $C^*$-algebra. Then the following are equivalent:
\begin{enumerate}
\item\label{tw1} $A$ has finite nuclear dimension;
\item[\namedlabel{tw1prime}{(i')}] $A$ has nuclear dimension at most one;
\item\label{tw2} $A$ is $\mathcal Z$-stable;
\item\label{tw3} $A$ has almost unperforated Cuntz 
semigroup;
\item[\namedlabel{tw3prime}{(iii')}] $A$ is pure.
\end{enumerate}
\end{theorem}

Although we are centring our discussion on the parallels with Connes' II$_1$ factor work, let us briefly point out that Theorem \ref{Intro:Structure} covers both finite and infinite $C^*$-algebras. In the infinite setting -- when there are no non-zero densely-defined lower-semicontinuous traces -- every 
 simple $C^*$-algebra satisfying \ref{tw3} or \ref{tw3prime}  is in fact purely infinite in the sense of Cuntz, and this case of Theorem \ref{Intro:Structure}\ref{tw2} is Kirchberg's celebrated $\mathcal O_\infty$-stability theorem (see \cite{K:ICM,KP:Crelle}),\footnote{Here $\mathcal O_\infty$ is the Cuntz algebra on infinitely many generators, which is already $\mathcal Z$-stable.} which has been the standard route to access the Kirchberg--Phillips classification theorem in examples.  Our task is to prove Theorem \ref{Intro:Structure}\ref{tw3prime}$\Rightarrow$\ref{tw2} in the tracial case.

\subsection*{\JiangSu-stability and pureness}

In hindsight, the conditions and implications of the Toms--Winter conjecture have strong parallels to Connes' work -- a viewpoint brought to the fore by Nate Brown in talks around 2010.  As an important intermediate step in his proof that an injective separably acting II$_1$ factor $\mathcal M$ is hyperfinite, Connes shows that $\mathcal M$ absorbs the hyperfinite II$_1$ factor $\mathcal R$ tensorially, i.e.\ $\mathcal M\cong\mathcal M\,\overline{\otimes}\,\mathcal R$ (a condition known as the McDuff property, due to the central sequence characterisation in \cite{McDuff}).  The $C^*$-algebraic counterpart of $\mathcal R$ from the perspective of tensorial absorption is the Jiang--Su algebra $\mathcal Z$ discovered in \cite{JiangSu}.\footnote{Classification shows that Elliott's earlier construction of an algebra with the same invariant in his work on the range of the Elliott invariant (\cite{Elliott:Invariant}), is isomorphic to $\mathcal Z$, though no particular attention appears to have been given to this algebra at the time.} It should be viewed as a non-commutative analogue of the complex numbers: $\mathcal Z$ is $KK$-equivalent to $\mathbb C$ and has unique trace. Moreover, $A$ and $A\otimes\mathcal Z$ have the same $K$-theory and traces for all $C^*$-algebras $A$.  This essentially makes \emph{$\mathcal Z$-stability} -- tensorial absorption of $\mathcal Z$ -- necessary for classification. In fact, it is the final regularity hypothesis needed for classification: the unital classification theorem classifies simple separable unital nuclear $C^*$-algebras which are $\mathcal Z$-stable and satisfy the UCT in terms of $K$-theory and traces. Not only does $\Z$-stability parallel the McDuff property, but, in a very precise sense, the passage from nuclearity and $\mathcal Z$-stability to finite nuclear dimension\footnote{Winter and Zacharias' nuclear dimension (\cite{WZ:Adv}) is a non-commutative topological covering dimension, and can be thought of as a higher-dimensional version of hyperfiniteness. Indeed $C^*$-algebras of nuclear dimension zero, are approximately finitely-dimensional.} mirrors Connes' passage from injectivity and the McDuff property to hyperfiniteness. 

Pureness is a much more basic property, which is akin to Murray and von Neumann's foundational comparison theory for projections on II$_1$ factors. Let $\M$ be a II$_1$ factor with trace $\tau$. Then $\M$ has:
\begin{enumerate}[label=\arabic*.]
    \item \emph{comparison of projections}: a projection $p$ is subequivalent to a projection $q$ if and only if $\tau(p)\leq \tau(q)$; \label{it:IIone-comp}
    \item \emph{existence of projections}: for each $t\in [0,1]$, there exists a projection $p\in \mathcal M$ with $\tau(p)=t$.\label{it:IIone-exist}
\end{enumerate}
The appropriate $C^*$-algebraic analogue of these properties uses positive elements rather than projections (projections may not tell much of the story, as, for example, there may be no non-trivial projections). In \cite{Cuntz:Semigroup}, Cuntz initiated a comparison theory for positive elements in a $C^*$-algebra $A$, which was further developed in \cite{BH82,Rordam92,Perera97,KirchbergRordam02,BPT}, before reaching its modern form as the Cuntz semigroup $\Cu(A)$ of equivalence classes of positive elements in the stabilisation $A\otimes\mathcal K$ (see \cite{CEI,APT:MAMS} and the surveys \cite{APT,GP_Cuntz_semi}). In this setup, positive elements $a,b\in A\otimes\mathcal K$ are \emph{Cuntz subequivalent} (written $a\precsim b$) if one can approximately compress $b$ onto $a$.  Subsequently, the Cuntz semigroup has been heavily studied both in its own right -- highlights include \cite{ERS:AJM,Thiel:CMP,APRT:Duke} -- and as an instrumental tool in work on regularity and classification (\cite{Andrew_classification,BPT,OPR:IMRN,Robert:Adv:NCCW}).

Unlike II$_1$ factors, simple stably finite $C^*$-algebras can have complex collections of traces. In fact, any metrisable Choquet simplex can arise as the tracial state space of a simple unital approximately finite-dimensional (AF) $C^*$-algebra (\cite{Blackadar:AF}). Therefore, the appropriate analogues of Murray and von Neumann's conditions must be formulated in a way that involves all (quasi)traces.\footnote{To get all functionals on the Cuntz semigroup one needs dimension functions from quasitraces rather than just traces. It is a huge open problem which arguably goes back to the foundations of the field whether quasitraces are traces.  Haagerup's breakthrough positive answer for exact $C^*$-algebras (\cite{Haagerup:Traces}) covers the case of nuclear $C^*$-algebras of interest in this paper.} The relevant tracial data for a positive element $a$ is the tracial rank (or ``dimension function'') $d_\tau(a)\coloneqq \lim_{n\to\infty}\tau(a^{1/n})$, which is just the trace of the support projection of $\pi_\tau(a)$ in the von Neumann algebra $\pi_\tau(A)''$. Note that, for fixed $a \in A$, the map $\tau\mapsto d_\tau(a)$ is a lower-semicontinuous affine function of $\tau$. With this setup, the following two conditions provide the analogues of Murray and von Neumann's conditions for simple $C^*$-algebras.
\begin{enumerate}
    \item[\namedlabel{almost-unperf}{(1)}] \emph{Almost unperforation}\footnote{As will be familiar to experts, the condition that follows is known as `strict comparison'. However there are various versions of strict comparison in the literature, and in this paper, we extensively use the subtly different `strict comparison with respect to $X$ (a specified set of bounded traces)'. To avoid confusion, we use the terminology `almost unperforation' here.} of the Cuntz semigroup, or equivalently (in the simple case), that for non-zero positive elements $a,b\in A\otimes\mathcal K$ such that $d_\tau(a)<d_\tau(b)$ for all lower-semicontinuous (quasi)traces $\tau:A_+\to[0,\infty]$ we have that $a\precsim b$.
    \item[\namedlabel{ranks}{(2)}] \emph{All ranks occur}: in the unital case,\footnote{In general, one must work with suitable $[0, \infty]$-valued functions on the cone of lower-semicontinuous (quasi)traces.} this says every strictly positive lower-semicontinuous affine function on the (quasi)tracial state space of $A$ arises as the rank function of a positive element in $A\otimes\mathcal K$.
\end{enumerate}
The need for a strict inequality in \ref{almost-unperf} in the $C^*$-algebraic setting -- in contrast to the von Neumann condition -- can be seen already with inequivalent projections in AF algebras that agree on all traces.\footnote{Take, for example, an AF algebra $A$ with $K_0(A)\cong \Q^2$ and $K_0(A)_+=\{(0,0)\} \cup (\Q_{>0} \times \Q$).}
This condition goes back to Blackadar's work on the fundamental comparison property of projections (\cite{Blackadar:Comparison}). For (quasi)traceless simple $C^*$-algebras, the second condition above is automatic so that almost unperforation is equivalent to pureness (and pure infiniteness, see \cite{Cuntz81}).  In contrast to the analogous von Neumann algebraic property, strict comparison is not automatic: it fails in all of the exotic $C^*$-algebras from \cite{villadsen1998simple,Villadsen:JAMS,rordam_acta,Andrew_classification}.  

Whether all ranks occur in every non-elementary simple stably finite $C^*$-algebra is an open problem, which we believe is due to Nate Brown (see \cite[Section 7 and Problem XXIII]{99} for further discussion). A major result of Thiel shows all ranks occur for all simple $C^*$-algebras of stable rank one (\cite{Thiel:CMP}),\footnote{The simplicity hypothesis was subsequently removed in \cite{APRT:Duke}.} and his recent preprint \cite{Thiel:arXiv} shows that all continuous ranks \emph{almost occur} on a simple unital non-elementary $C^*$-algebra $A$, i.e. the collection of rank functions is uniformly dense in the continuous affine functions on the trace space of $A$.

In the presence of almost unperforation, many different conditions have been long known to equivalent to \ref{ranks} for a simple $C^*$-algebra (see \cite[Proposition 5.2]{CETW:IMRN}, for example). In particular, almost unperforation closes the gap between all continuous ranks almost occuring, and all ranks occuring: by Thiel's \cite{Thiel:arXiv} for simple unital non-elementary $C^*$-algebras with almost unperforated Cuntz semigroup, condition \ref{ranks} above is automatic.

Another important concept which is equivalent to condition \ref{ranks} in the presence of almost unperforation -- and hence now automatic -- is the following notion of divisibility. Introduced by Winter in \cite{winter2012nuclear}, this represents the idea that a Cuntz class $x$ can almost be divided in the Cuntz semigroup.\footnote{There are various versions of this definition in the literature. Winter's original definition (\cite[Definition 3.5(i)]{winter2012nuclear}), omits the $\epsilon$ cut down from \ref{div} -- we have included it here, since this is the modern form of the definition used in the Cuntz semigroup literature. In the presence of almost unperforation \ref{div} and \cite[Definition 3.5(i)]{winter2012nuclear} are equivalent.  Note too that it is not reasonable to ask for an exact division, as there could be $K$-theoretic obstructions to dividing classes of projections. }
\begin{enumerate} 
    \item[\namedlabel{div}{(2$'$)}] \emph{Almost divisibility}: given a positive element $x\in A\otimes\mathcal K$, $\epsilon>0$, and $n\in\mathbb N$, there exists a positive element $y\in A\otimes \mathcal K$ such that $y^{\oplus n}$ is Cuntz subequivalent to $x$, and $(x-\epsilon)_+$ is Cuntz subequivalent to $y^{\oplus(n+1)}$.
\end{enumerate}

Simple $C^*$-algebras with almost unperforation and almost divisibility were coined \emph{pure} $C^*$-algebras by Winter in \cite{winter2012nuclear}, and both comparison and almost divisibility (and their higher-dimensional analogues) are crucial ingredients in Winter's second $\mathcal Z$-stability theorem (and its extension to the non-unital setting in \cite{Tikuisis:MathAnn}). This captures $\Z$-stability at the level of the Cuntz semigroup: $A$ is pure if and only if $\Cu(A)\cong\Cu(A\otimes\mathcal Z)$  
(see \cite[Proposition 23]{99}).   Since \cite{winter2012nuclear}, there has been extensive research into the structure of pure $C^*$-algebras, with new developments outside the simple setting (\cite{robert2017nuclear,antoine2024pure, perera2025extensions,PTV26}) and breakthroughs giving new examples of simple pure non-nuclear $C^*$-algebras (\cite{Robert:Adv:Selfless,Selfless,Ozawa:Selfless}) especially concerning natural $C^*$-simple groups.

The last piece of the Toms--Winter conjecture we provide in this paper is the $C^*$-algebraic analogue of the first part of Connes' work: squeezing out non-trivial central sequences (known as Murray and von Neumann's property $\Gamma$) from semidiscreteness (\cite[Sections 1 and 2]{Connes:Annals}) en route to the McDuff property by means of his highly delicate spectral gap characterisation of full II$_1$ factors. Subsequently Haagerup, and then Popa, gave direct routes to hyperfinitess from semidiscreteness and injectivity respectively (\cite{Haagerup:InjectiveHyperfinite,Popa:InjectiveHyperfinite}). Our main result parallels this part of Connes' theorem by obtaining $\mathcal Z$-stability -- a central sequence condition -- for simple $C^*$-algebras from the combination of nuclearity and fundamental information about the positive elements of $A$ -- pureness.

\begin{theorem}\label{Intro:pure-implies-Zstable}
Let $A$ be a simple separable nuclear pure $C^*$-algebra.  Then $A$ is $\mathcal Z$-stable.
\end{theorem}

As we will explain in Section \ref{S:Outline}, the strategy behind Theorem \ref{Intro:pure-implies-Zstable} is directly connected to the von Neumann algebraic setting, and in particular to Popa's proof of injectivity implies hyperfiniteness, through his \emph{local quantisation technique} (\cite{Popa:InjectiveHyperfinite}). Popa's strategy already has a history of major impact on the structure and classification of nuclear $C^*$-algebras, especially with his $C^*$-algebraic version of local quantisation in \cite{Popa:PJM} in the presense of real rank zero, which in turn inspired Lin's notion of tracially approximately finite-dimensional $C^*$-algebras (\cite{Lin2001, Lin2001classification}). Here we develop these local quantisation ideas in the framework of tracially complete $C^*$-algebras. 

\subsection*{The \JiangSu-stability theorems of Winter and Matui--Sato} While Theorems \ref{Intro:Structure} and \ref{Intro:pure-implies-Zstable} give complete characteristions of $\Z$-stability for simple separable nuclear $C^*$-algebras in terms of positive elements, there are a number of important precursors. Firstly Winter's tour-de-force from 2010 not only found a route from finite nuclear dimension to $\mathcal Z$-stability for non-elementary simple separable unital\footnote{This hypothesis was removed by AT in \cite{Tikuisis:MathAnn},} nuclear $C^*$-algebras but also isolated the importance of the combination of almost unperforation and almost divisibility.

\begin{theorem*}[{Winter's second $\mathcal Z$-stability theorem; \cite{winter2012nuclear}}]
Let $A$ be a simple separable unital  nuclear $C^*$-algebra. Suppose that either
\begin{enumerate}
    \item\label{it:winter1}  $A$ is non-elementary and of finite nuclear dimension, or
    \item\label{it:winter2} $A$ is pure and of locally finite nuclear dimension.
\end{enumerate}
Then $A$ is $\mathcal Z$-stable.
\end{theorem*}

In part \ref{it:winter2} of Winter's second $\Z$-stability theorem, the locally finite nuclear dimension hypothesis is an internal approximation condition designed to include all approximately subhomogeneous (ASH) $C^*$-algebras. It is an open question whether every simple stably finite nuclear $C^*$-algebra is ASH. In fact, it is not even known whether they have locally finite nuclear dimension.  At present, there seems to be no reasonable approach to tackling either problem. Instead, for examples of interest (such as those from group dynamics or groupoids) the systematic route to ASH structure runs through classification, and so requires $\mathcal Z$-stability as an input. So, while on the one hand one can view Theorem \ref{Intro:pure-implies-Zstable} as just removing the locally finite nuclear dimension hypothesis from Winter's theorem, on the other hand it provides a fully abstract approach to accessing classification from positive element data alone. Following this result, Winter (and others) started to speculate that, if almost unperforation alone was not enough information to give $\Z$-stability for non-elementary simple separable unital nuclear $C^*$-algebras, then surely pureness would suffice.

Theorem \ref{Intro:pure-implies-Zstable} also gives a faster route to part \ref{it:winter1} of Winter's theorem. It is not so difficult to show that simple non-elementary $C^*$-algebra of nuclear dimension at most $n$ satisfies higher-dimensional versions of conditions \ref{almost-unperf} and \ref{ranks} (\cite{Robert:MJM} and \cite[Section 4]{winter2012nuclear}). In \cite{APRT:JFA,antoine2024pure}, dimension-reduction techniques were developed in the Cuntz semigroup, showing that the combination of weaker comparison and divisibility conditions are equivalent to pureness (\cite[Theorem 10.5]{APRT:JFA} for the case of simple $C^*$-algebras needed here, and \cite[Theorem C]{antoine2024pure} for a much more general result extending to the non-simple setting). While tricky, these arguments are reasonably short and performing the dimension-reduction in the Cuntz semigroup is considerably less involved than the highly demanding (and impressive) central divisibility and dimension-reduction Winter performs in the central sequence algebra.  Once pureness is reached, Theorem \ref{Intro:pure-implies-Zstable} kicks in to give $\Z$-stability.

There have been a number of conditional abstract results getting from almost unperforation to $\Z$-stability, particularly with conditions on the trace space. Of particular importance to us is the following breakthrough result of Matui and Sato.  
\begin{theorem*}[Matui--Sato's $\Z$-stability theorem {(\cite{MS:Acta})}]
Let $A$ be a non-elementary simple separable unital  nuclear $C^*$-algebra with finitely many extremal traces.  Suppose that the Cuntz semigroup of $A$ is almost unperforated. Then $A$ is $\Z$-stable.
\end{theorem*}
The main technical ingredient for this is an essential component of Theorem \ref{Intro:pure-implies-Zstable}. Matui and Sato's headline result was extended to more general trace simplices in \cite{Sato12,TWW:IMRN,KR14}: those where the extremal traces are compact and of finite covering dimension.  Ozawa introduced the theory of $W^*$-bundles and tracial completions in \cite{Ozawa} as a framework to explain this last set of results, and this is the starting point for the line of work \cite{BBSTWW,CETWW,CETW:IMRN,carrion2023tracially,E:JLMS,ET:RR0} on tracially complete $C^*$-algebras, which scaffolds our proof of Theorem~\ref{Intro:Structure}.

We give a detailed description of the tools and techniques behind the proof of Theorems~\ref{Intro:Structure} and~\ref{Intro:pure-implies-Zstable} in Section \ref{S:Outline}.

\subsection*{\safeK-stability} The methods behind Theorems \ref{Intro:Structure} and \ref{Intro:pure-implies-Zstable} involve a structural analysis of the uniform tracial ultrapower as we set out in Section \ref{S:Outline}. We expect these tools to have broader applications outside the regularity programme. One such application is to $K$-stability of pure $C^*$-algebras.

The notion of \emph{$K$-stability} was introduced by Thomsen (\cite{thomsen}) in the context of non-stable $K$-theory. Operator $K$-theory is defined considering all matrix amplifications: for the $K_1$-group of a unital $C^*$-algebra $A$ is defined as
\begin{equation}
    K_1(A)=\pi_0(\bigcup_{n=1}^\infty U(M_n(A))),
\end{equation}
where $u\in U(M_n(A))$ is included as $\mathrm{diag}(u,1)\in U(M_{n+1}(A))$. In many situations, it is important to know how $K_1(A)$ -- which can often be computed using powerful six term exact sequences -- relates to unitaries in the algebra. To this end, $A$ is said to be $K_1$-injective ($K_1$-surjective) if the canonical map $U(A)/U_0(A)\to K_1(A)$ is injective (or surjective).  Cuntz's work shows that all simple purely infinite $C^*$-algebras are $K_1$-bijective,\footnote{The definition of $K_1$-bijectivity will not come as a surprise.} and Rieffel's stable rank one provides a strong from of stable finiteness which gives rise to $K_1$-bijectivity (\cite{Rieffel83}; see \cite[V.3.1.26]{blackadar2006operator}). As a consequence of Lin's recent stable rank one theorem (\cite{LinStableRank}), all simple pure $C^*$-algebras are $K_1$-bijective. In general, simple $C^*$-algebras need not be $K_1$-bijective: Villadsen provided an example showing the failure of $K_1$-surjectivity (\cite{MR1916650}), and Toms has very recently announced an example showing failure of $K_1$-injectivity (\cite{toms2026simple}).

For a unital $C^*$-algebra $A$, Thomsen's non-stable $K$-theory studies the homology theory given by the higher homotopy groups of the unitary group of $A$,\footnote{We discuss how to define $K$-stability of non-unital $C^*$-algebras in Section \ref{sec:KStable}.} and $A$ is said to be \emph{$K$-stable} when the natural maps $U_n(A)\to U_{n+1}(A)$ induce isomorphisms $\pi_k(U_n(A))\cong \pi_k(U_{n+1}(A))$ for all $n$ and all $k$. This is precisely the condition that $C(S^k, M_{n}(A))$ is $K_1$-bijective for all $k\geq 0$ and $n\geq 1$.  For a $K$-stable $C^*$-algebra, the non-stable $K$-groups agree with the ordinary $K$-groups.  

Certainly not all $C^*$-algebras are $K$-stable -- indeed, no commutative or finite-dimensional $C^*$-algebra can be $K$-stable (\cite{thomsen}) -- but it turns out that quite often one can get $K$-stability from regularity conditions.  Essentially by definition, stable $C^*$-algebras are $K$-stable, and in \cite{thomsen}, Thomsen obtained $K$-stability in the presence of various tensor product decompositions. Further examples include simple purely infinite $C^*$-algebras (\cite{zhang2001}), non-elementary simple $C^*$-algebras with stable rank one and real rank zero (\cite{zhang}), and $\mathcal{Z}$-stable $C^*$-algebras (\cite{jiang1997}; see also \cite{hua2025k}). For us, Rieffel gave an important upgrade from $K_1$-bijectivity of $C^*$-algebras with stable rank one to $K$-stability under an additional divisibility assumption (\cite[Theorem 4.13]{rieffel87}).\footnote{A $C^*$-algebra $B$ is divisible in Rieffel's sense, if there are arbitrarily large $n\in\mathbb N$ such that $B\cong M_n(C_n)$ for some $C^*$-algebras $C_n$.} 
 
Moreover, $K$-stability is not restricted to regular $C^*$-algebras: more recently, it was shown that non-$\mathcal{Z}$-stable Villadsen algebras of the first kind (\cite{villadsen1998simple}) are $K$-stable (\cite{seth2026tensorial}), and outside the simple case there are results characterising $K$-stability in classes of inductive limits (\cite{seth2023, seth2026tensorial}).

Our first application is to show that for simple $C^*$-algebras, pureness is a sufficient regularity condition for $K$-stability (at least when quasitraces are traces).  Modulo the quasitrace problem, this extends the $K_1$-bijectivity of all simple pure $C^*$-algebras from Lin's stable rank one theorem to $K$-stability.

\begin{theorem}\label{Intro:KStable1}
Let $A$ be a unital simple pure $C^*$-algebra whose quasitraces are traces.  Then $A$ is $K$-stable.
\end{theorem}

More generally, this works under the hypotheses of strict comparison and almost divisible unit; simplicity is not required.

\begin{theorem}\label{Intro:KStable2}
Let $A$ be a unital $C^*$-algebra which has strict comparison with respect to bounded traces and almost divisible unit. Then $A$ is $K$-stable.
\end{theorem}

These results address the problem, raised by Thiel, and recorded as \cite[Problem LVIII]{99}, as to whether all pure $C^*$-algebras might be $K$-stable. In the purely infinite setting, Kirchberg examined this problem recording progress in his unfinished manuscript \cite{Kirchberg:preprint}; see his squeezing property and \cite[Proposition 4.2.15]{Kirchberg:preprint}.

The proof strategy is to split the problem using the trace-kernel extension. This goes back to Matui and Sato's work (\cite{MS:Acta}) and has become a go-to tool in analysing the structure of simple $C^*$-algebras; indeed it underpins Theorems~\ref{Intro:Structure} and~\ref{Intro:pure-implies-Zstable} through the Matui--Sato theorem.  Related to these $K$-stability questions, Lin's stable rank one theorem from \cite{LinStableRank} also runs through the trace--kernel extension; very crudely, the idea is to obtain a strong form of approximations giving rise to stable rank one in the trace-kernel quotient, and then use strict comparison to lift this back to obtain a result for $A$. 

For a free ultrafilter $\omega\in\beta\mathbb N\setminus\mathbb N$ on the natural numbers, the norm-ultrapower $A_\omega$ of a tracial $C^*$-algebra fits into the short exact sequence
\begin{equation}
0\longrightarrow J_A\longrightarrow A_\omega\longrightarrow A^\omega\longrightarrow 0,
\end{equation}
where $J_A$ and $A^\omega$ are the trace-kernel ideal and quotient respectively (see Section~\ref{subsec:tracially-complete}).  While $J_A$ is not stable, it goes back to \cite{Schafhauser:Crelle} using \cite{HR:JFA} that it is \emph{separably stable} (separable subalgebras are contained in stable separable subalgebras) when $A$ behaves nicely enough, and that will give $K$-stability of $J_A$. In particular, comparison is enough for the separable stability of $J_A$ (this is essentially in \cite{CGSTW}).  Our work in Section~\ref{sec:KStable} will establish both stable rank one and Rieffel divisibility of $A^\omega$ from the assumptions in Theorem~\ref{Intro:KStable2}. We then get $K$-stability of $A_\omega$ by Thomsen's result on $K$-stability of extensions. This passes back to $K$-stability of $A$, in roughly the same way that $K_1$-bijectivity of $A_\omega$ gives $K_1$-bijectivity of $A$.

\subsection*{AI statement, author contribution statement and description of our route to the results} On 8 September 2026, and after discussion with some (but not all) authors on this paper, SW asked ChatGPT 6 Astra using codex whether a $\|\cdot\|_{2,u}$-separable $W^*$-bundle $\mathcal M$ over a compact metrisable space with classification of projections (Definition \ref{def:classification-of-projections} below, but specialised to the $W^*$-bundle framework) and real rank zero such that all fibres of $\mathcal M$ are isomorphic to the hyperfinite II$_1$ factor, must be trivial.\footnote{Initially, the separability condition was omitted from the prompt; this was swiftly added when ChatGPT started to produce non-separable counterexamples.}  After 27 minutes of computation, and no further human input, it produced a short 4 page proof of this statement, which will be available in the supplementary materials section of \url{https://sites.google.com/view/stuartwhite/research/publications} for transparency. A follow up prompt the same day asked Astra to ``Please carefully check your work, be particularly careful with the maximality argument in the setting of $W^*$-bundles." Its response to this is summarised in the pdf. 

Beyond the two prompts on 8 September described above, the authors have made no further active use of generative AI to produce this version of the paper (for developing the mathematical content, proofreading, or any other purpose). Instead we chose to investigate the consequences of these ideas ourselves, leading to this paper. As this is our first experiment with using generative AI as a research tool, some of the authors intend to investigate ChatGPT 6 Astra's capabilities on the research questions tackled in this paper during the time between submitting this paper to the arXiv and its public availability.  We will report on our findings in a subsequent revision.

The authorship of this paper was established between 9 and 13 September. The authors fall into 2 subsets with non-empty intersection:\footnote{We will stick to mathematical traditions and not identify these subsets.} 
\begin{itemize}
    \item a subset of the authors have been heavily involved in developing the theory of tracially complete $C^*$-algebras over the last 5 years for the purpose of tackling problems of this nature, by means of regularity conditions at the tracially complete level. Some of these authors had also considered various potential statements in the spirit of the `small projections lemma', inspired by reading of Connes' \cite[Theorem 2.1]{Connes:Annals}, at different times over many years but had never been able to make a result of this nature work.
    \item a subset of the authors were already finalising a paper including the $K$-stability results given here as Theorems~\ref{Intro:KStable1} and~\ref{Intro:KStable2}, which were obtained after the INI programme at Queen's University Belfast in July 2025, in the window August 2025--May 2026. Technical ingredients from the $K$-stability project regarding the structure of $A^\omega$ have been repurposed for this paper, as they play a significant role in our route to Theorem~\ref{Intro:pure-implies-Zstable}. For this reason, the two projects were merged.  
\end{itemize} 

At the time of making the 8 September prompt, some of the authors thought that a positive answer would give Theorem \ref{Intro:pure-implies-Zstable} in the case when $T(A)$ is a Bauer simplex. They believed they had a proof that the tracial completion $(\M,X)$ of a pure $C^*$-algebra $A$ is real zero, in a batch of unpublished work build up from ongoing attempts at Theorem~\ref{Intro:pure-implies-Zstable} using tracially complete $C^*$-algebras since at least 2021 and the $K$-stability project. In fact, as we discovered in one of a few quite stressful moments during this project, that `proof' was based on a misunderstanding of a key point in \cite{robert2017nuclear} and flawed.

Following the 8 September prompt, we quickly identified that the weak CPoU condition should be the right tool to extend the $W^*$-bundle lemma to the general tracially complete framework. A lot of trial and error was involved, before a strategy for this was obtained on 19 September. Simultaneously, we had quickly realised that we could get away with the weaker condition that $\M^\omega$ has real rank zero, which we found much easier to prove than $\M$ having real rank zero.  At this point we were trying to put together an argument for McDuffness of $\M$ using that $\M$ has classification of projections together with $\M^\omega$ having real rank zero, and had found a way of making the maximality argument work, but kept running into what felt like local problems in implementing a tracially complete version of the small matrix lemma. This came together on 2 October, via the lifting arguments making up Section \ref{Sec:CentralDivisible}.

Right up until 3 October, the main tool was our tracially complete version of the small matrix lemma; it was only as Section \ref{S:SmallMatricesOutline} was drafted did we realise the potential to prove hyperfiniteness directly in the spirit of Popa. The relevant sections were completely rewritten between 3--7 October. Finally, the introduction to the paper was completely rewritten on 6--7 October, following Thiel's \cite{Thiel:arXiv}, which we had no prior sight of, though some of us verified on 6 October. 

\subsection*{Acknowledgements} SE, AT and SW thank Wilhelm Winter for the extensive and inspiring conversations we have all had over years on the subject of the Toms--Winter conjecture. SE, CS, AT, and SW thank José Carrión, Jorge Castillejos, and Jamie Gabe for all the insights shared during the large scale collaboration on tracially complete $C^*$-algebras \cite{carrion2023tracially} which provides the framework for this paper.

CS would like to thank Ben Hayes for helpful conversations about tracially complete $C^*$-algebras during CS's visit to University of Virginia in April 2026.

SW would like to thank Andrew Toms for discussions on the Toms--Winter conjecture during his visit to Oxford in the second half of 2025, and David Jekel and Maxwell Ryder for the collaborative work \cite{JRW:inPrep} on selfless $W^*$-bundles (and apologise to them for delay in finalising \cite{JRW:inPrep} while this paper was completed). It was through the work on \cite{JRW:inPrep} that SW became convinced that the combination of real rank zero and classification of projections should be a fundamental regularity property for tracially complete $C^*$-algebras.  SW would also like to thank Shuoxing Zhou for helping him set up ChatGPT codex. The prompt mentioned above was the first test prompt he sent to the model following this set up.

\section{Discussion of methods}\label{S:Outline}

\renewcommand{\thetheorem}{\arabic{theorem}}
\numberwithin{theorem}{section}

\subsection{Tracially complete \cstar-algebras and Matui--Sato's theorem}\label{S2.1} 
Our route to Theorems \ref{Intro:Structure} and \ref{Intro:pure-implies-Zstable} crosses the tracially complete bridge from von Neumann algebras to $C^*$-algebras.  Consider a simple separable nuclear unital and non-elementary $C^*$-algebra $A$. When $A$ has a unique trace $\tau$, the GNS representation $\pi_{\tau}(A)$ of $A$ generates the hyperfinite II$_1$ factor as a von Neumann algebra by Connes' theorem.  In particular $\pi_{\tau}(A)''$ is McDuff (again by Connes).  In their breakthrough work \cite{MS:Acta}, Matui and Sato obtain a condition -- \emph{property (SI)} -- as a consequence of strict comparison. This is designed to lift the McDuff property from $\pi_\tau(A)''$ back to $\Z$-stability of $A$, and allows them to prove $\Z$-stability for simple separable unital and nuclear $C^*$-algebras with strict comparison and a unique trace.   See \cite[Section 6]{99} for a discussion of this central sequence trace-kernel lifting argument.

When $A$ has many traces, the McDuff property of each $\pi_\tau(A)''$ individually (as one gets from Connes' theorem)\footnote{This works just as well at non-extremal traces when $\pi_\tau(A)''$ is not a factor.} is not enough to run Matui and Sato's argument. Instead one needs the McDuff property to hold in a uniform fashion across all $\pi_\tau(A)''$ as $\tau$ varies (called \emph{uniformly McDuff} in \cite[Definition 4.2]{CETW:IMRN}).  This is a condition that lives at the level of the tracial completion.
Just as the enveloping von Neumann algebra $\pi_\tau(A)''$ may be obtained by completing the unit ball of $A$ in the $2$-norm associated to $\tau$, for many traces an analogous completion using the \emph{uniform $2$-norm} $\|x\|_{2,T(A)}\coloneqq \sup_{\tau\in T(A)}\|x\|_{2,\tau}$ gives the \emph{tracial completion} $\completion{A}{T(A)}$, introduced by Ozawa in \cite{Ozawa}. The level of structure $\completion{A}{T(A)}$ carries depends on the complexity of the Choquet simplex $T(A)$ of traces on $A$. 
\begin{itemize}
    \item When $T(A)$ is finite-dimensional (i.e.\ $A$ has finitely many extremal traces), $\completion{A}{T(A)}$ is a finite von Neumann algebra. In fact, it is $\oplus_{\tau\in\partial_eT(A)}\pi_\tau(A)''$. 
    \item When $T(A)$ is a \emph{Bauer} simplex, Ozawa showed that $\completion{A}{T(A)}$ naturally carries the structure of a $W^*$-bundle with factorial fibres, i.e.\ it is the section algebra of suitably continuous bundle of II$_1$ factors (the algebras $\pi_\tau(A)''$, for $\tau\in\partial_eT(A)$) over the base space $\partial_eT(A)$ (see \cite{Ozawa,Evington:PhD,BBSTWW} for the theory of these $W^*$-bundles).
    \item With no restrictions on $T(A)$, four of the authors together with Carrión, Castillejos, and Gabe worked out an abstract framework for these tracial completions, which we imaginatively called \emph{factorial tracially complete $C^*$-algebras} (\cite{carrion2023tracially}). These consist of pairs $(\mathcal M,X)$, where $\mathcal M$ is a unital $C^*$-algebra whose unit ball is closed in the norm $\|x\|_{2,X}\coloneqq\sup_{\tau\in X}\|x\|_{2,\tau}$ for a closed face $X\subseteq T(\mathcal M)$. Again for $\tau \in \partial_e X$, $\pi_\tau(\M)''$ is a factor. Using $X\coloneqq T(A)$, a tracial completion $\completion{A}{T(A)}$ is a factorial tracially complete $C^*$-algebra. 
\end{itemize}

The raison d'être of tracially complete $C^*$-algebras is to provide an intermediate step in the transfer of results from the von Neumann algebraic framework back to the $C^*$-level (see the discussion in \cite[Section 1.6]{carrion2023tracially}).  This was their role (a little in disguise) in both the passage from $\mathcal Z$-stability to finite nuclear dimension --- indeed, this happened in 3 distinct steps \cite{MS:Duke,SWW}, \cite{BBSTWW} and \cite{CETWW} separated according to the distinct levels of complexity of the tracial completion  --- and in the abstract approach to the unital classification theorem in \cite{CGSTW}. The performance is reprised here for the structure theorem, and this time tracial completions take centre stage. This is possible through the following modern framing of Matui and Sato's $\mathcal Z$-stability theorem.

\begin{theorem*}[{Matui and Sato's $\mathcal Z$-stability theorem \cite{MS:Acta}}]
Let $A$ be a unital simple separable nuclear $C^*$-algebra with strict comparison.  Suppose that $A$  is uniformly McDuff, or equivalently, its tracial completion $\completion{A}{T(A)}$ is McDuff.  Then $A$ is $\mathcal Z$-stable.
\end{theorem*}

The uniform McDuff property for $A$ means that $\completion{A}{T(A)}$ absorbs $\mathcal R$ tensorially, or equivalently that $M_2$ embeds into the uniform tracial central sequence algebra $A^\omega\cap A'$. We will prove the following, from which Theorems \ref{Intro:Structure} and \ref{Intro:pure-implies-Zstable} follow using the Matui--Sato $\mathcal Z$-stability theorem.

\begin{theorem}\label{Thm:McDuff1}
    Let $A$ be a separable and nuclear $C^*$-algebra such that $T(A)$ is compact and non-empty.  Suppose that $A$ has strict comparison with respect to $T(A)$ and an almost divisible strictly positive element.  Then its tracial completion $(\completion{A}{T(A)},T(A))$ is McDuff.
\end{theorem}

\subsection{Regular amenable tracially complete \cstar-algebras} In our opinion, there are two major advantages of working in the framework of tracially complete $C^*$-algebras and tracial completions, rather than trying to keep everything at the $C^*$-algebraic level. Firstly doing so avoids the need for quite elaborate functional calculus arguments which have the effect of replicating approximate versions of behaviour in the tracial completion in the original $C^*$-algebra. Arguably more important though is how this allows arguments to be conceptually organised using properties of the tracial completion.  Indeed, there are complete classification theorems for separable (in the uniform trace norm) factorial amenable\footnote{This can either be described by asking for each GNS representation associated to a designated trace to generate an injective von Neumann algebra, or by means of the completely positive approximation property in the $\|\cdot\|_{2,X}$-norm; see \cite[Section 4]{carrion2023tracially}.} McDuff tracially complete $C^*$-algebras. In the $W^*$-bundle framework, this is Ozawa's trivialisation theorem (\cite[Theorem 15]{Ozawa}), and in full generality it is \cite[Theorem C]{carrion2023tracially}.  These are complemented by structural results: the equivalence of property $\Gamma$, the McDuff property, and hyperfiniteness in both cases (again covered by Ozawa's trivialisation theorem for $W^*$-bundles, and \cite[Theorem B]{carrion2023tracially} in general). Until very recently it was unknown whether these conditions were automatic. However, Toms' recent announcement (\cite{Toms:Gamma}) of a non-elementary simple separable unital nuclear $C^*$-algebra which fails to have uniform property $\Gamma$ (i.e.\ its uniform tracial completion does not have $\Gamma$) shows that, in the tracially complete setting, just as with $C^*$-algebra classification, this regularity hypothesis is needed (though as Toms' example has a complex trace space it is still open whether regularity is genuinely needed in the $W^*$-bundle framework).  Accordingly, the remaining challenge (both for the structure theorem for simple nuclear $C^*$-algebras, and for its tracially complete counterpart) is to find a characterisation of regularity for factorial amenable tracially complete $C^*$-algebras in much more elementary terms. 

Known results for tracial completions of $\mathcal Z$-stable $C^*$-algebras guided our way (e.g.\ from \cite{carrion2023tracially,E:JLMS,ET:RR0}), by pointing to useful properties of these tracial completions that could be aimed for as intermediary lemmas.
Firstly, regular factorial amenable tracially complete $C^*$-algebras have well-behaved classes of projections.\footnote{It is open whether every II$_1$ factorial $W^*$-bundle has a non-trivial projection.} Indeed in a tracially complete $C^*$-algebra $(\mathcal M,X)$ with property $\Gamma$,\footnote{The real condition underpinning this result is called complemented partitions of unity (CPoU), which are obtained from property $\Gamma$; see \cite[Theorem 1.4]{carrion2023tracially}.} projections are \emph{classified by traces} in the sense that the following tracially complete analogues of Conditions \ref{it:IIone-comp} and \ref{it:IIone-exist} for projections in II$_1$ factors hold (see \cite[Theorems 7.17 and 7.18]{carrion2023tracially}).
\begin{enumerate}
    \item[1.] \emph{Comparison of projections}: if $p,q\in \mathcal M$ are projections, then $p\precsim q$ (Murray--von Neumann subequivalence) if and only if $\tau(p)\leq\tau(q)$ for all $\tau\in X$.
    \item[2.] \emph{Existence of projections}: given a continuous affine function $f\colon X\to [0,1]$, there exists a projection $p\in\mathcal M$ with $\tau(p)=f(\tau)$ for all $\tau\in X$.
\end{enumerate}

These conditions are very much akin to pureness, but to justify focusing on projections, rather than considering all positive elements through a Cuntz semigroup-like object, we also need to know there are enough projections.  This was recently achieved by SE and AT who showed that tracially complete $C^*$-algebras with property $\Gamma$ (again, the real condition underlying this is CPoU) have real rank zero (\cite[Theorem A]{ET:RR0}).  Forthcoming work of Jekel, Ryder and SW will obtain the combination of real rank zero and classification of projections for many $W^*$-bundles without property $\Gamma$ (\cite{JRW:inPrep}) and explores the consequences of these two conditions in the $W^*$-bundle setting.  All results of \cite{JRW:inPrep} were completed before the results of this paper. 

These developments collectively led to the belief that the combination of real rank zero and the classification of projections should provide an appropriate elementary form of regularity for tracially complete $C^*$-algebras.  We justify this viewpoint through the following two theorems, which are the main technical results of the paper. Firstly we prove the equivalence of the various potential forms of regularity for factorial tracially complete $C^*$-algebras $(\M,X)$ and their ultrapowers $(\M^\omega,X^\omega)$.

We first introduce some terminology, which can be thought of as a tracially complete version of Cuntz-subequivalence, strict comparison, rank realisation for tracially complete $C^*$-algebras.

\begin{definition}\label{def:traciallyPure}
    Let $(\M,X)$ be a tracially complete $C^*$-algebra. 
    \begin{enumerate}
        \item \label{tt:traciallypure-1} Let $a,b\in \M_+$. Say that $a$ is \emph{$(2,X)$-Cuntz subequivalent to $b$}, written $a\precsim_X b$, if there exists a sequence $(v_m)\in\M$ with $\|v_m^*bv_m-a\|_{2,X}\to 0$.  
        \item \label{tt:traciallypure-2} Say that $(\M,X)$ has \emph{strict $(2,X)$-comparison} if for all $n\in\mathbb N$ and $a,b\in M_n(\M)_+$ with $d_\tau(a)<d_\tau(b)$ for all $\tau\in X$,\footnote{Each $\tau \in X$ is viewed as a trace on $M_n(\M)$ with trace $n$.} we have $a\precsim_{X}b$ in $M_n(\M)$.
        \item \label{tt:traciallypure-3}Say that $(\M,X)$ has \emph{approximately constant rank functions} if for each $\lambda>0$ and $\epsilon>0$ there exists $n\in\mathbb N$ and a positive element $a\in M_n(\M)_+$ such that $\sup_{\tau\in X}|d_\tau(a)-\lambda|<\epsilon$.
    \end{enumerate}
\end{definition}

The combination of \ref{tt:traciallypure-2} and \ref{tt:traciallypure-3} can be thought of as a version of pureness of tracially complete $C^*$-algebras. We show below that it implies much stronger notions of regularity, including that $\mathcal M$ is pure as a $C^*$-algebra.

\begin{theorem}[The regularity theorem for factorial tracially complete $C^*$-algebras]\label{PurenessThm} Let $(\M,X)$ be a II$_1$ factorial tracially complete $C^*$-algebra.  Then the following are equivalent:
\begin{enumerate}
    \item \label{it:Pureness.1} $(\M,X)$ has strict $(2, X)$-comparison and has approximately constant rank functions.
    \item \label{it:Pureness.3}  $\mathcal M$ has strict comparison with respect to $X$ and $\mathcal M$ has divisible unit;
    \item \label{it:Pureness.4}  $(\M,X)$ has classification of projections and $\mathcal M$ has real rank zero.
\end{enumerate}
In this case, $T(\mathcal M) = X$ and $\mathcal M$ is pure and has stable rank one.
\end{theorem}

We call a tracially complete $C^*$-algebra \emph{regular} if the conditions of Theorem~\ref{PurenessThm} hold.  Along the way to proving the regularity theorem, we will show that a II$_1$ factorial tracially complete $C^*$-algebra $(\mathcal M, X)$ is regular if and only if its ultrapower $(\mathcal M^\omega, X^\omega)$ is regular (Corollary~\ref{cor:regular}).
We will describe these conditions, and the route through this theorem in the next subsection, but henceforth we call all such factorial tracially complete $C^*$-algebras \emph{regular}.

Importantly but unsurprisingly tracial completions of simple unital pure $C^*$-algebras (whose quasitraces are traces) are regular, and in fact from  looking at (iii) one only needs comparison and almost divisibility of the unit.

\begin{proposition}\label{Pure:TracialCompletion}
    Let $A$ be a $C^*$-algebra such that $T(A)$ is compact and non-empty. Suppose that $A$ has strict comparison by $T(A)$ and the class of a strictly positive element of $A$ is almost divisible (as happens if $A$ is simple unital pure, and all quasitraces are traces). Then the tracial completion $(\completion{A}{T(A)},T(A))$ satisfies Theorem \ref{PurenessThm}(iii) and so is regular.  If $A$ is nuclear, then $(\completion{A}{T(A)},T(A))$ is amenable.
\end{proposition}

Throughout the body of the paper, our main regularity assumption will be that $(\mathcal M^\omega, X^\omega)$ satisfies \ref{it:Pureness.3} of the regularity theorem, as this is most readily obtained for tracial completions of simple unital pure $C^*$-algebras (Proposition~\ref{prop:sc_divisible_tracial_ultra}) and leads to the properties needed for Theorem~\ref{Thm:McDuff1}.

As a counterpart to the regularity theorem, we upgrade the classification theorem for amenable tracially complete $C^*$-algebras in \cite[Theorem C]{carrion2023tracially} to a structure theorem.
The equivalence of (i), (ii) and (ii$'$) is \cite[Theorem B]{carrion2023tracially}, and it is straightforward to get from these properties to (iii). The novel implication here is (iii)$\Rightarrow$(ii), which we outline in Section \ref{S:Outline}, showing how it relates to Connes' work.

\begin{theorem}[Structure theorem for factorial amenable II$_1$ tracially complete $C^*$-algebras]\label{TC:Structure}
    Let $(\mathcal M,X)$ be type II$_1$ factorial amenable tracially complete $C^*$-algebra. The following are equivalent:
    \begin{enumerate}
        \item $(\mathcal M,X)$ is hyperfinite;
        \item $(\mathcal M,X)$ is McDuff;
        \item[(ii$'$)] $(\mathcal M,X)$ has property $\Gamma$;
        \item $(\mathcal M,X)$ is regular.
    \end{enumerate}
\end{theorem}

Once we have all these results in place, Theorem \ref{Thm:McDuff1} is an immediate consequence of Theorem \ref{TC:Structure} and Proposition \ref{Pure:TracialCompletion}.

\subsection{The regularity theorem for factorial tracially complete \cstar-algebras}

In this section we describe some aspects of the proof of Theorem \ref{PurenessThm}.

\subsubsection{Real rank zero of the tracial ultrapower and tracial approximate oscillation} Real rank zero of the ultrapower $(\M^\omega,X^\omega)$ plays an essential role in our argument.  In the last few years, the importance of this condition has become evident through work of both Fu and Lin (\cite{fu2025tracial,LinStableRank,fu2025stablerankrealrank,fu2026divisibilityrealrankzero}) on $C^*$-algebras with tracial approximate oscillation zero, Elliott and Niu (\cite{ElliottNiu.SmallBoundary.1,ElliottNiu.SmallBoundary.2}) on $C^*$-algebraic characterisations of the small boundary property, and Vaccaro's recent work \cite{Vaccaro:arXiv}.

One of the fundamental difficulties in working with positive elements $a$ in a $C^*$-algebra with tracial state space $T(A)$ is that the rank function $\tau\to d_\tau(a)$ is only lower-semicontinuous; the support projection of a positive element will lie in the tracial completion $\M\coloneqq \completion{A}{T(A)}$ when its rank function is continuous. Fu and Lin introduced the notion of tracial approximate oscillation zero in \cite{fu2025tracial}, roughly speaking to mean that positive elements in $A$ can be tracially approximated (inside the hereditary subalgebra they generate) by elements whose rank function is almost continuous; we will not get further in to the specifics of this definition and its various modifications -- they are quite technical. This concept played an important role in Lin's stable rank one theorem (\cite{LinStableRank}). 

 In \cite[Section 6]{fu2025tracial}, Fu and Lin set out that for the class of $C^*$-algebras $A$ they call compact,\footnote{This means that $A$ is equal to its Pedersen ideal (see \cite[Definition 2.8]{fu2025tracial}).} tracial approximate oscillation zero implies that $A^\omega$ is real rank zero (see \cite[Corollary 6.5]{fu2025tracial}). We believe that this could be applied to $A\coloneqq\M$ (under the assumption that $(\M,X)$ has strict comparison with respect to $X$) to get real rank zero for $\M^\omega$. However, Fu and Lin's work does not to our knowledge classify projections in $\M^\omega$ -- a result we need anyway -- and once we have such a classification in place, we are able to obtain real rank zero for $\M^\omega$ more directly. Moreover,  for simple $C^*$-algebras, Fu shows tracial approximate oscillation zero is equivalent to real rank zero of $A^\omega$ (see \cite{fu2026divisibilityrealrankzero}, which also gives other characterisations).
 
 From the viewpoint of emphasising the tracial completion and ultrapower that we described in Section \ref{S2.1}, we feel that the route we have taken from Cuntz semigroup conditions to real rank zero of the tracial ultrapower in this section has the advantage of being conceptual, as well as being able to go beyond the simple case, to our framework of comparison by bounded traces and almost divisible unit.

\subsubsection{Lifting unitaries}\label{SS1.3.3} It is a standard theme in $C^*$-algebra classification that existence and uniqueness results for maps into ultrapowers give rise to classification results for maps into the underlying algebra by means of intertwining arguments.  A subtle point here is that a unitary $u$ in an ultrapower $A_\omega$ of a unital $C^*$-algebra $A$ always lifts to a sequence of unitaries in $A$, and likewise a unitary in an ultrapower $\M^\omega$ of a II$_1$ factor $\M$ lifts to a sequence of unitaries in $\M$. In the framework of tracially complete $C^*$-algebras the lifting corresponding result is not known (as we do not know of a $\|\cdot\|_{2,X}$-stability result for unitaries). Accordingly the corresponding intertwining argument from \cite[Theorem 5.11]{carrion2023tracially} at the tracially complete level has a hypothesis that unitaries lift from reduced products to sequences of unitaries in the underlying algebra.\footnote{For technical reasons that intertwining argument is set up using the sequence algebra $(\M^\infty,X^\infty)$ rather than the ultrapower.}  

For similar reasons, we need certain unitaries to lift in our argument, though we were unable to prove that the do a priori.\footnote{As a consequence of Theorem~\ref{TC:Structure} and a CPoU argument, it turns out that all unitaries lift. Results in \cite{JRW:inPrep} also give techniques for lifting unitaries in ultrapowers of trivial $W^*$-bundles.}
We circumvent this difficulty by correcting the $K$-theory of the unitaries we need to lift by adjusting them on tracially small corners, crucially using K$_1$-injectivity from stable rank one of $\M^\omega$ by Rieffel (\cite{rieffel87}). This part of the argument is reminiscent of how Lin's classification results into tracially AF algebras (\cite{Lin2001,Lin2001classification}), as well as arguments using more general forms of tracial approximations, involve adjusting the $K$-theory in tracially small corners.

\subsubsection{Organisation of the argument}\label{subsec:organisation}

Let $A$ be a $C^*$-algebra whose trace space $T(A)$ is compact and $A$ has strict comparison with respect to $T(A)$ and divisible unit.  Let $(\mathcal M, X) = (\completion{A}{T(A)}, T(A))$ denote the tracial completion of $A$.  Our starting point is to show Theorem~\ref{PurenessThm}\ref{it:Pureness.3} holds with $(\mathcal M^\omega, X^\omega) = (A^\omega, \overline{T_\omega(A)})$ in place of $(\mathcal M, X)$.  Under the conditions on $A$, $(\mathcal M^\omega, X^\omega)$ satisfies \ref{it:Pureness.3} of Theorem~\ref{PurenessThm} by Proposition~\ref{prop:sc_divisible_tracial_ultra}.  For this reason we view condition \ref{it:Pureness.3} for the ultrapower of the main starting point of the work in the regularity theorem for tracially complete $C^*$-algebras.

Let $(\M,X)$ be a factorial II$_1$ tracially complete $C^*$-algebra whose ultrapower satisfies Theorem~\ref{PurenessThm}\ref{it:Pureness.3}. As a consequence of comparison $T(\M^\omega)=X^\omega$ (Proposition~\ref{prop:computing-traces}).  We describe in the steps below, the route first to \ref{it:Pureness.4} for $(\mathcal M^\omega, X^\omega)$ and then to \ref{it:Pureness.3} and \ref{it:Pureness.4} for $(\mathcal M, X)$ itself, as well as the additional ingredients we need for the structure theorem.
\begin{enumerate}[(1)]
    \item Comparison of projections for $\M^\omega$ is almost an immediate consequence of strict comparison (with small constant trace projections coming from the divisible unit being added to create a small gap).\label{Step.1}
    \item Using the projectionisation ideas from \cite{CETWW}, use the divisible unit of $\M^\omega$ show that $(\M^\omega,X^\omega)$ has existence of projections. Note that only the divisible unit is required to realise all continuous ranks on $X^\omega$ by projections in $\M^\omega$.\label{Step.2}
    \item Real rank zero of $\M^\omega$ is obtained by combining the existence of projections with comparison techniques to build increasing sequences of projections inside a hereditary subalgebra generated by a positive contraction $a$ which converge to $a$ in $\|\cdot\|_{2,X^\omega}$. This completes the proof of \ref{it:Pureness.4} for the ultrapower.\label{Step.3}
    \item Since we can apply the previous steps at all matrix levels, the Murray--von Neumann semigroup of $\M^\omega$ is identified with the positive continuous affine functions on $X$, and real rank zero then shows that $\M^\omega$ is pure. Stable rank one also follows from real rank zero and the classification of projections, using a result of Rieffel (\cite{rieffel87}).  This proves the final statement in Theorem~\ref{PurenessThm} for the ultrapower.\label{Step.4}
    \item We obtain  existence of projections in $(\M,X)$ by intertwining the existence and uniqueness of projections in $(\M^\omega,X^\omega)$. The liftability of the unitaries we need here comes from stable rank one of $\M^\omega$, via a theorem of Larry Brown.  \label{Step.5}
    \item Uniqueness of projections in $(\M,X)$ is also obtained by intertwining the results from $(\M^\omega,X^\omega)$ to build a $\|\cdot\|_{2,X}$-Cauchy sequence of partial isometries converging to the required partial isometry. Here we correct the $K$-theory in small corners of the unitaries arising from uniqueness in $\M^\omega$ to ensure liftability as described in Subsection \ref{SS1.3.3}. \label{Step.6}
    \item As an technical ingredient in the structure theorem for II$_1$ factorial amenable tracially complete $C^*$-algebras, we also need to lift certain $^*$-homomorphisms from finite-dimensional $C^*$-algebras into $\M^\omega$ to a sequence of maps into $\M$. Again this is powered by stable rank one as described in Subsection \ref{SS1.3.3}. \label{Step.7}
    \item 
    At this point we can prove the structure theorem (Theorem \ref{TC:Structure}) and obtain all its consequences for $\Z$-stability of simple separable nuclear $C^*$-algebras set out in the introduction. \label{Step.7b}
    \item Given a unital embedding $\iota\colon \mathcal R\to\mathcal M^\omega$, we show the relative commutant $\M^\omega\cap\mathcal \iota(R)'$ has strict comparison with respect to $X^\omega$, and hence $T(\M^\omega\cap\iota(\mathcal R)')=X^\omega$.  As a consequence it is both factorial and McDuff. Accordingly, while in general $\M^\omega$ may not have complemented partitions of unity, its relative commutants $\M^\omega\cap\iota(\mathcal R)'$ do.  In this way, we can perform local-to-global arguments in such relative commutants.  \label{Step.8}
    \item To show $\M$ has real rank zero and reach \ref{it:Pureness.4} (and hence \ref{it:Pureness.3} by a small Riesz decomposition argument), we follow the recent work of SE and AT showing that II$_1$ factorial tracially complete $C^*$-algebras with CPoU have real rank zero. A careful analysis of their work shows that it is sufficient to have CPoU inside  certain self-adjoint elements of $\mathcal M^\omega$. By the previous step, this can be achieved by building a copy of $\mathcal R$ in the commutant of self-adjoint elements in $\mathcal M^\omega$ via real rank zero.\label{Step.9}
    \item Finally,  Section~\ref{Sect8} examines strict $(2,X)$-comparison using standard functional calculus arguments.  In particular, we show that these conditions coincide for $\mathcal M$ and $\mathcal M^\omega$ and moreover agree with operator-norm strict comparision for $\mathcal M^\omega$.  This is used to show the equivalence with the other condition \ref{it:Pureness.1} and the other conditions, both for $\mathcal M$ and $\mathcal M^\omega$.\label{Step.10} 
\end{enumerate}
Steps \ref{Step.1}--\ref{Step.4} are proved in Section~\ref{sec:tracial-ultrapowers}, followed by the $K$-stability theorems in Section \ref{sec:KStable}.  We then come back to Steps \ref{Step.5}--\ref{Step.7} in Section \ref{Sec:CentralDivisible}. Section \ref{S:LocalHyperfinite} proves the structure theorem for II$_1$ factorial amenable tracially complete $C^*$-algebras (Step \ref{Step.7b}), which is followed in Section \ref{S:StructureNuclear} by the $C^*$-algebraic counterparts (Theorems \ref{Intro:Structure} and \ref{Intro:pure-implies-Zstable}).  We conclude the proof of the regularity theorem with Steps \ref{Step.8}, and \ref{Step.9} in Section \ref{S:RegTCAlg}, followed by Step~\ref{Step.10} in Section~\ref{Sect8}.

\subsection{The structure theorem for factorial amenable tracially complete \cstar-algebras}\label{S:SmallMatricesOutline}

Finally, to complete these expository sections, we come back to Connes' proof that injective II$_1$ factors have property $\Gamma$ and how it relates to the implication (iii)$\Rightarrow$(ii) of Theorem \ref{TC:Structure}. We do this in the increasing levels of trace simplex complexity.

\subsubsection{\IIone\ factors: Connes and Popa}

Recall that a \IIone\ factor $\M$ has property $\Gamma$ if for any finite collection $u_1,\dots,u_n$ of unitaries in $\M$ and $\epsilon>0$, there exists a projection $p\in \M$ with $\tau(p)=1/2$ and $\|[u_i,p]\|_2<\epsilon$ for $i=1,\dots,n$ (this characterisation is due to Dixmier in \cite{Dixmier69}).  In his spectral gap characterisation of fullness (\cite[Theorem 2.1]{Connes:Annals}), Connes gives a route to property $\Gamma$ by means of a maximality argument: considering approximately central projections $p$, where the tolerance in the commutation estimates is controlled by the size of $p$. Precisely, let $\M$ be a \IIone\ factor with the property that the $C^*$-algebra generated by $\M$ and $\M'$ on $L^2(\M)$ does not meet the compacts (this is a straightforward consequence of semidiscreteness). In \cite[Theorem 2.1(d)$\Rightarrow$(b) and Lemma 2.4]{Connes:Annals}, Connes shows that given unitaries $u_1,\dots,u_n$ and $\epsilon>0$, there is a non-zero projection $p\in\M$ with
\begin{equation}\label{Connes:SmallProjection}
    \|[p,u_i]\|_2\leq\epsilon\|p\|_2,\quad i=1,\dots,n.
\end{equation}
Moreover, this can also be achieved in all corners of $\M$ (\cite[Lemma 2.5]{Connes:Annals}).  Via a maximality argument, Connes then concludes that there is a trace $1/2$ projection satisfying \eqref{Connes:SmallProjection}, giving property $\Gamma$.

Popa's local quantisation argument for injectivity implies hyperfiniteness (\cite{Popa:InjectiveHyperfinite}) uses a very similar maximality argument at the very last step. In his framework, for a fixed finite subset $\mathcal F$ of an injective \IIone\ factor $\M$ and $\epsilon>0$, Popa shows (\cite[Proposition 2.2]{Popa:InjectiveHyperfinite}) that one can find a matrix subalgebra $F$ of $\M$ with unit $1_F$ such that
\begin{equation}\label{E:Popa:MatrixApprox}
    \|E_F(1_Fx1_F)-(x-(1_{\M}-1_F)x(1_{\M}-1_F))\|_{2}<\epsilon\|1_F\|_2,\quad x\in\mathcal F.
\end{equation}
Again, a maximality argument gives hyperfinitess (see \cite[Section 3]{Popa:InjectiveHyperfinite}): one considers a maximal family $(F_i)_i$ of matrix subalgebras with pairwise othogonal unit, such that \eqref{E:Popa:MatrixApprox} holds for $F\coloneqq\bigoplus_iF_i$. By applying his lemma on a corner, it follows that $1_F=1$, giving rise to the local approximations needed to witness hyperfiniteness.

In both cases, a crucial point is Namioka's trick for replacing an approximately central positive element by an approximately central projection, given by a suitable spectral projection of the given element.

\subsubsection{Small approximately central constant trace projections in \wstar-bundles}\label{RobotSection}

Now let us consider Connes' idea in the framework of $W^*$-bundles. Recall that property $\Gamma$ for $W^*$-bundles is defined by asking for a globally approximately central projection of trace $1/2$. Precisely, a $\|\cdot\|_{2,u}$-separable $W^*$-bundle $\mathcal M$ over the compact metrisable space $K$ has property $\Gamma$ if there exists a projection $p\in\M^\omega\cap \M'$ with $\tau(p)=1/2$ for all $\tau\in K$ (see \cite[Definition 4.3.1]{Evington:PhD}). 

In contrast to the II$_1$ factor setting, in the framework of $W^*$-bundles, increasing sequences (or nets) of projections need not converge in $\|\cdot\|_{2,u}$.  Therefore in order to run a maximality argument, one needs to force convergence. This is cleanly done by working with constant trace projections, as any increasing family of constant trace projections will be $\|\cdot\|_{2,u}$-Cauchy.  

Like for Connes, the main challenge is to produce the small approximately central constant trace projections.  This is achieved in the following result -- which we reiterate was provided by ChatGPT 6 Astra.

\begin{lemma}\label{RobotLemma}
    Let $\M$ be a $\|\cdot\|_{2,u}$-separable $W^*$-bundle over $K$ whose fibres are all isomorphic to copies of the  hyperfinite II$_1$ factor. Suppose $\M$ has classification of projections and real rank zero. Then $\M$ has the property that for all finite subsets $\mathcal F\subset\mathcal M$ and $\epsilon>0$, and $0<\gamma_0$, there exists a projection $p\in\M$ with constant trace $\gamma_p$ for some $0<\gamma_p<\gamma_0$ such that
 \begin{equation}\label{2.2:E1}
     \sup_{\tau\in K}\|[x,p]\|_{2,\tau}^2\leq \epsilon^2\gamma_p,\quad x\in\mathcal F. 
 \end{equation}
\end{lemma}

This result will not be used in the paper, but we sketch the proof ideas below, giving ourself the extremely generous hypothesis that self-adjoints admit spectral projections for intervals.\footnote{This will not be the case, but using real rank zero instead just amounts to keeping track of more $\delta$'s.}

\begin{proof}[Sketch proof]
Fix a finite set $\mathcal F \subseteq \mathcal M$ consisting of positive contractions and $\delta > 0$.  For each $\tau \in K$, $\pi_\tau(\mathcal M) \cong \mathcal R$, so classification of projections in $\mathcal M$ allows us to find a finite-dimensional matrix algebra $F \subseteq \mathcal M$ which $(2,\tau)$-approximately contains $\mathcal F$ in the sense that for all $x \in \mathcal F$, there is $y \in F$ such that $\|x - y\|_{2, \tau} < \epsilon$.
A standard compactness argument provides a finite open cover $K = \bigcup_{i=1}^n U_i$ and finite-dimensional matrix algebras $F_1, \ldots, F_n \subseteq \mathcal M$ such that for each $i = 1, \ldots, n$ and $\tau \in U_i$, $\mathcal F$ is $\|\cdot\|_{2,\tau}$-approximately contained in $F_i$.  Fix a partition of unity $f_1, \ldots, f_n \in C(K)_+$ subordinate to $(U_i)_{i=1}^n$, i.e. $\sum_{i=1}^n f_i = 1$ and $\mathrm{supp}(f_i) \subseteq U_i$.  

The goal is to produce a small constant $\gamma>0$ and mutually orthogonal projections $p_1, \ldots, p_n \in \mathcal M$ such that $p_i$ commutes with $F_i$ and $\tau(p_i) = \gamma f_i(\tau)$ for all $\tau \in K$.  Once this has been done, $p \coloneqq \sum_{i=1}^n p_i \in \mathcal M$ is a projection with $\tau(p) = \gamma$ for all $\tau \in K$.  Further, since $p_i$ commutes with $F_i$ and $\mathcal F$ is $(2, \tau)$-approximately contained in $F_i$ whenever $\tau(p_i) = \gamma f_i(\tau) \neq 0$, one expects $p$ to approximately commute with $\mathcal F$.  A computation shows that if $x \in \mathcal F$ and $y_i \in F_i$ for $i = 1, \ldots, n$, then
\begin{equation}
    \|[x, p]\|_{2, \tau}^2 = \sum_{i=1}^n \big( \|(x - y_i) p_i\|_{2, \tau}^2 + \|p_i (x - y_i)\|_{2, \tau}^2 \big), \quad \tau \in K. 
\end{equation}
\emph{If} for each $i=1,\dots,n$, we could obtain the bound 
\begin{equation}\label{E1.5}
    \|(x - y_i)p_i\|_{2, \tau}^2 \leq \|x - y_i\|_{2, \tau}^2 \tau(p_i) = \gamma f_i(\tau) \|(x - y_i)\|_{2, \tau}^2,\quad \tau\in K 
\end{equation}
then it would follow that $p$ is approximately central.

The crucial ingredient is choosing $p_i$ (and the $y_i$) so that $x_i - y_i$ has small operator norm on the corner under $p_i$ allowing us to achieve an estimate of the form \eqref{E1.5}.  To do this, we refine the approach above slightly.  For each $x \in \mathcal F$, let $y_{x, i} \in \mathcal F_i$ be positive contractions chosen to be $(2, \tau)$-close to $x$ whenever $\tau \in U_i$ and set
\begin{equation}
    a_i = \sum_{x \in \mathcal F} (x - y_{x, i})^2,
\end{equation}
so that $\tau(a_i)$ is small if and only if $\mathcal F$ is $(2, \tau)$-approximately contained in $F_i$. Adjusting the initial estimates as necessary, we can arrange
\begin{equation}
    \sup_{\tau \in X} \min_{1 \leq i \leq n} \tau(a_i) < \delta/32.
\end{equation}

For a sufficiently small $\gamma > 0$, the projections $p_1, \ldots, p_n$ are constructed inductively as follows.  Let $E_1 \colon \mathcal M \rightarrow \mathcal M \cap F_1'$ denote the canonical conditional expectation.  Since $\mathcal M \cap F_1'$ is isomorphic to a corner of $\mathcal M$, our generous assumption allows us to take the spectral projection $q_1 \in \mathcal M \cap F_1'$ of $E_1(a_1)$ on the interval $[0, \delta/2]$.  Then one computes $\tau(q_1) \geq 1/4$ for all $\tau \in U_i$, so the classification of projections in $\mathcal M \cap F_1'$ allows one to construct a projection $p_1 \in \mathcal M \cap F_1'$ such that $\tau(p_1) = \gamma f_1(\tau)$ for all $\tau \in K$ (assuming $\gamma < 1/4)$.  The key estimate that will allow us to show $p_1$ (and eventually $p$) is approximately central is $p_1 a_1 p_1 \leq \delta p_1$.

Assuming $p_1, \ldots, p_{i-1}$ have been constructed, one repeats the construction above replacing $E_1(a_1)$ with 
\begin{equation}
    E_i(a_i) + \dim(F_i) E_i(p_1 + \ldots p_{i -1}).
\end{equation}
So long as $\gamma$ is sufficiently small, the resulting projection $\tilde p_i\in\mathcal M\cap F_i$ will be (operator norm) approximately orthogonal to $p_1, \ldots, p_{i-1}$, and a standard polar decomposition argument produces a partial isometry $v_i$ with $v_i^*v_i=\tilde{p}_i$ and $v_i\approx \tilde{p_i}$ and a projection $p_i\coloneqq v_iv_i^* \in \mathcal M$ close to $\tilde p_i$ and approximately orthogonal to $p$.  Note that $\gamma$ depends on the integer $n \in \mathbb N$ coming from the compactness argument and the dimension of the approximately finite-dimensional subalgebras $F_i$.
\end{proof}

Connes' maximality argument then shows that amenable II$_1$ factorial $W^*$-bundles with real zero and classification of projections have property $\Gamma$. In the $(2,u)$-separable case, Ozawa's trivialisation theorem \cite[Theorem 15]{Ozawa} gives McDuffness and hyperfiniteness. Using real rank zero from Theorem~\ref{PurenessThm}, this therefore proves Theorem \ref{TC:Structure} under the additional assumption that $X$ is Bauer.

Let us make two brief remarks on this argument:
\begin{itemize}
    \item Firstly, the choice of $p_i$ under a small spectral projection in the proof of Lemma \ref{RobotLemma} to get operator norm control of $x_i-y_i$ on the corner under $p_i$ is reminiscent of the Namioka trick used by both Connes and Popa.  Of course, we had the advantage of their deep work to produce local finite-dimensional approximations as the starting point, so we can  averaging the appropriate positive element into the commutant of $F$ and taking the spectral projection there. In this way we obtain the analogous approximately central projection which behaves approximately like a spectral projection for the given positive element.
\item Secondly, staring at this proof more carefully reveals that at the very last step, instead of conjugating $\tilde{p}_i$ to $p_i$, we could instead have moved the finite-dimensional subalgebra $F_i\tilde{p}_i$ underneath $1-(p_1+\dots p_{i-1})$. Taking care with the estimates, this would produce a finite-dimensional subalgebra $F$ of $\M$ whose unit has constant trace $\gamma$, and $(2,u)$-approximately contains $\mathcal F$. Through Popa's maximality argument we can directly obtain hyperfiniteness, and avoid using the trivialisation theorem.
\end{itemize}

\subsubsection{Small approximately central matrices in tracially complete \cstar-algebras}

In order to prove the regularity theorem in full, we will need to modify the argument above to handle the general framework of tracially complete $C^*$-algebras. The overarching strategy is the same, but there are a number of technical challenges which must be overcome. It would not be as straightforward to use property $\Gamma$ directly as in the $W^*$-bundle case because the definition of property $\Gamma$ for a factorial II$_1$ tracially complete $C^*$-algebra $(\mathcal M,X)$ is subtly different compared to the $W^*$-bundle framework.\footnote{It is not a priori enough to have a projection $p\in\mathcal M^\omega\cap \M'$ which has $\tau(p)=1/2$ for all $\tau \in X^\omega$; instead $p$ must divide the trace of all elements and satisfy $\tau(pa)=\tau(p)=\tau(a)=\tau(a)/2$ for all $a\in A$, and $\tau\in X^\omega$ (see \cite[Section 3]{CETW:IMRN}).}
It would be more sensible to head for the McDuff property (which has the advantage of being exactly the ingredient needed by Matui--Sato's theorem for our $C^*$-algebraic consequences), or even better, to implement Popa's strategy and obtain hyperfiniteness directly.

It is much more obvious how to perform partition-of-unity arguments in $W^*$-bundles compared with tracially complete $C^*$-algebras.  Indeed, in the previous section, we could embed the functions $f_1,\dots,f_n\in C(K)_+$ centrally into $\M$, and the fact that $f_i$ is supported on $U_i$ enters into the computation that $\tau(q_i)\geq 1/4$. However a factorial tracially complete $C^*$-algebra might have trivial centre, and we can only expect to perform partition-of-unity arguments using positive elements when we have \emph{convex open covers} of $X$ using open sets of the form $\{\tau\in X:f_n(\tau)<\alpha\}$ for some $f_1,\ldots,f_N \in\Aff(X)$ and $\alpha\in\mathbb R$. Assuming the $f_n$ are strictly positive, and using the Cuntz--Pedersen realisation result (Proposition \ref{prop:CP}), one can describe a finite convex open cover in the form $U_n\coloneqq \{\tau\in T(A):\tau(a_n)<\delta\}$ for positive elements $a_1,\dots,a_N\in\M$ and $\delta>0$. The fact that this is a cover is expressed in the condition
\begin{equation}
\sup_{\tau\in T(A)}\min_{1\leq n\leq N}\tau(a_n)<\delta.
\end{equation}

A similar problem was faced in the proof that $\Z$-stable simple separable nuclear $C^*$-algebras have finite nuclear dimension, where the Bauer simplex case (\cite{BBSTWW}) was handled years before the general case (\cite{CETWW}) due to the difficulty of performing partition-of-unity arguments. The tool resolving this issue in \cite{CETWW} -- complemented partition of unities (CPoU) -- is too strong; in our context these will only arise once we have proved property $\Gamma$. But, a weaker version always exists, called weak CPoU. Up to now it has only been a technical ingredient on the route from property $\Gamma$ to CPoU (it is used in similar ways in equivariant and relative settings in \cite{SW:APDE,KLTV:TAMS}).  The partition-of-unity intuition plays out by using weak CPoU to perform an affine partition-of-unity argument without first obtaining property $\Gamma$ (or CPoU) in the tracially complete setting. Modulo considerable technicalities, this enables us to prove a `small local approximation result' (Theorem~\ref{SmallLocalApproximations}), which is exactly what is needed to run Popa's maximality argument.

While our sketch in the previous section, and indeed our initial prompt, assumes $\M$ has real rank zero, in fact it is very little extra trouble to work with the a priori weaker hypothesis that $\M^\omega$ has real rank zero. We set up Theorem~\ref{SmallLocalApproximations} in this way.

\section{Preliminaries}

\subsection{The Cuntz semigroup and traces} 
Let $A$ be a $C^*$-algebra and let $a,b \in (A \otimes \K)_+$. By definition, $a$ is \emph{Cuntz subequivalent} to $b$ (notation $a \precsim b$) if there exists a sequence $(r_n)_{n=1}^\infty$ in $A \otimes \K$ such that $a = \lim_{n\to\infty} r_nbr_n^*$; however, we shall freely make use of the equivalent formulations of Cuntz subequivalent given by R{\o}rdam's Lemma (\cite[Proposition 2.4]{Rordam92}; see the account in the survey \cite[Theorem 2.7]{GP_Cuntz_semi}). 
By definition, $a$ is Cuntz equivalent to $b$ (notation $a \sim b$) if both $a \precsim b$ and $b \precsim a$. We recall that projections are Cuntz subequivalent if and only if they are Murray von Neumann subequivalent.
When $A$ is stably finite, Cuntz equivalence is the same as Murray--von Neumann equivalence for projections. 

The \emph{Cuntz semigroup} is the set of Cuntz equivalence class $\Cu(A) = (A \otimes \K)_+ / \sim$ equipped with the natural semigroup operation defined by direct sum (using that $M_2(\K) \cong \K$).  We write $[a]$ for the Cuntz equivalence class of $a \in (A \otimes \K)_+$.
The Cuntz subequivalence relation descends to an order on the semigroup $\Cu(A)$.
Every increasing sequence $(x_n)_{n=1}^\infty$ in $\Cu(A)$ has a supremum $\sup_n x_n \in \Cu(A)$ by \cite{CEI}. For $x, y \in \Cu(A)$, the relation $x \ll y$ holds if for every sequence $(y_n)_{n=1}^\infty$ with $y \leq \sup_n y_n$ there exist $n \in \N$ such that $x \leq y_n$. Note that $[(a-\eps)_+] \ll [a]$ for all $a \in (A \otimes \K)_+$ and $\eps > 0$.  We occasionally use the earlier `incomplete' Cuntz semigroup $W(A)$, consisting of equivalence classes of all positive elements from $\bigcup_{n=1}^\infty M_n(A)$. This fails to have suprema of increasing sequences.

We now recall the key properties of the Cuntz semigroup that we will use in this paper (see \cite{winter2012nuclear,antoine2024pure,GP_Cuntz_semi}).

\begin{definition}\label{def:cuntz-omnibus} 
Let $A$ be a $C^*$-algebra.
\begin{enumerate}
    \item An element $x \in \Cu(A)$ is said to be \emph{almost divisible} if for any $x' \in \Cu(A)$ such that $x' \ll x$ and any integer $k \geq 1$ there exists $y \in \Cu(A)$ with $ky \leq x$ and $x' \leq (k+1)y$.\footnote{Note that in some references the ``cutdown'' from $x$ to $x'$ in the definition of almost divisibility is omitted, leading to the formally stronger requirement on $y$ that $ky \leq x \leq (k+1)y$; see \cite[Remark 7.3.5]{APT:MAMS} for a discussion of the reasons behind this.}
    \item If all $x \in \Cu(A)$ are almost divisible, then $\Cu(A)$ is said to be \emph{almost divisible}.
    \item The Cuntz semigroup $\Cu(A)$ is said to be \emph{almost unperforated} if for all $x,y \in \Cu(A)$, the implication
    \begin{equation}
        (k+1)x \leq ky\text{ for some }k\in\mathbb N\Longrightarrow x\leq y
    \end{equation}holds.  
    \item The $C^*$-algebra $A$ is said to be \emph{pure} if $\Cu(A)$ is both almost unperforated and almost divisible.
\end{enumerate}
\end{definition}

We now turn our attention to the traces on a $C^*$-algebra. By default, when we say ``trace'' we shall mean tracial state. We write $T(A)$ for the tracial state space of the $C^*$-algebra $A$ (and, occasionally, $T_{\leq 1}(A)$ for the set of tracial positive functionals of norm $\leq 1$). 

Given $\tau \in T(A)$, the corresponding \emph{rank function} $d_\tau:(A \otimes \K)_+ \rightarrow [0,\infty]$ is defined by $d_\tau(a) = \lim_{n\to\infty} \tau(a^{1/n})$, where $a \in A \otimes \K$ and  $\tau$ is implicitly extended to a lower-semicontinuous densely-defined tracial weight on $A \otimes \K$ by using $\tau \otimes \mathrm{Tr}_\K$. All states (suitably defined) on $\Cu(A)$ arise in this way where $\tau$ is allowed to be a lower-semicontinuous quasitrace; see \cite[Proposition~4.2]{ERS:AJM}, which goes back to \cite[Theorem~II.2.2]{BH82}. In the exact case, all quasitraces are traces by Haagerup \cite{Haagerup:Traces}.
By abuse of notation, we also write $d_\tau$ for the induced functional on $\Cu(A)$ given by $d_\tau([a]) = d_\tau(a)$.
For any $a,b \in (A \otimes \K)_+$, it is true that $a \precsim b$ implies that $d_\tau(a) \leq d_\tau(b)$ for all $\tau \in T(A)$. While the converse fails in general, for a simple $C^*$-algebra $A$, $\Cu(A)$ is almost unperforated if and only if whenever $a,b\in (A\otimes\K)_+$ with $d_\tau(a)<d_\tau(b)$ for all lower-semicontinuous quasitraces $\tau\colon A\to[0,\infty]$ we have $a\precsim b$ (see \cite[Proposition 6.2]{ERS:AJM} or \cite[Proposition 3.2]{Rordam:IJM}).

In this paper, our main focus is on the situation where all traces are bounded, and the tracial space is compact (allowing the algebra to be potentially non-simple). In this case, as with we prefer the following version of strict comparison (which is not necessarily equivalent to almost unperforation), which only uses bounded traces (instead of the cone of all densely-defined lower-semicontinuous quasitraces) and works with all matrix amplifications $M_n(A)$ instead of the stabilisation $A \otimes \K$.  To emphasise this form of strict comparison we will always specify the relevant set of bounded traces.

\begin{definition}[Strict comparison with respect to tracial states] 
\label{sc_bdd_trace_W}   
Let $A$ be a $C^*$-algebra and $X$ be a compact subset of $T(A)$. Then we say $A$ has \emph{strict comparison with respect to $X$} if 
\begin{equation}
\bigl(\forall \tau \in X,\ d_\tau(a) < d_\tau(b)\bigr)
\implies a \precsim b  
\end{equation}
for all $n \in \mathbb{N}$ and non-zero $a,b \in M_n(A)_+$.
\end{definition} 

We shall access this version of strict comparison via the following theorem, which follows directly from the equivalence of the usual notion of strict comparison for a simple $C^*$-algebra and almost unperforation (\cite[Proposition~3.2]{Rordam:IJM}) together with the fact that every state on $\Cu(A)$ is of the form $d_\tau$ for a densely-defined lower-semicontinuous quasitrace $\tau$ (cf.\ \cite[Corollary~4.6]{Rordam:IJM}).
\begin{theorem}\label{thm:comparison}
    Let $A$ be a simple $C^*$-algebra with $T(A)$ non-empty and compact such that all densely-defined lower-semicontinuous quasitraces on $A \otimes \K$ are densely-defined traces and are bounded on $A$.
    Suppose $\Cu(A)$ is almost unperforated.
    Then $A$ has strict comparison with respect to $T(A)$.
\end{theorem}

A first use of strict comparison is the following, the proof of which is borrowed from a forthcoming paper by three authors along with Jos\'e Carri\'on and Jamie Gabe (\cite{CGSTW.2}; thanks!).

\begin{proposition}\label{prop:computing-traces}
    Let $A$ be a unital $C^*$-algebra and let $X \subseteq T(A)$ be a compact convex set.  If $A$ has strict comparison with respect to $X$, then $X = T(A)$.
\end{proposition}
\begin{proof}
    Let $W(A)$ denote the uncompleted Cuntz semigroup on $A$, let $G(A)$ denote its Grothendieck group, and let $g \colon W(A) \rightarrow G(A)$ denote the canonical map.  Define
    \begin{equation}
        G(A)_+ = \{g(y) - g(x) : x, y \in W(A),\ x \leq y \} \subseteq G(A)
    \end{equation}
    and $1_{G(A)} = g([1_A])$.  With this structure, $G(A)$ is an ordered abelian group with order unit.

    Note that every $\tau \in T(A)$ induces a state $d_\tau \colon W(A) \rightarrow [0,\infty)$, which in turn induces a state $\bar d_\tau \colon G(A) \rightarrow \mathbb R$.  Suppose $h, k \in G(A)$ are such that $\bar d_\tau(h) < \bar d_\tau(k)$ for all $\tau \in X$.  Let $x, y, x', y' \in W(A)$ be such that $h = x - x'$ and $k = y - y'$.  
    Then $d_\tau(x + y') < d_\tau(x' + y)$ for all $\tau \in X$.  Since $A$ has strict comparison with respect to $X$, we have $x + y' \leq x' +y$, so $h \leq k$.

    By \cite[Corollary~7.11]{Goodearl:Book}, every state on $G(A)$ has the form $\bar d_\tau$ for some $\tau \in X$.  Now fix $\sigma \in T(A)$.  Then there is $\tau \in X$ such that $\bar d_\sigma = \bar d_\tau$.  Then $d_\sigma = d_\tau$ as states on $W(A)$.  Now, for all $a \in A_+$,
    \begin{equation}
        \sigma(a) = \int_0^\infty d_\sigma((a - t)_+) \, dt = \int_0^\infty d_\tau((a-t)_+) \, dt = \tau(a),
    \end{equation}
    so $\sigma = \tau \in X$.
\end{proof}

\subsection{Real rank zero and stable rank one}\label{subsec:RR0-SR1} The approximation properties of real rank zero and stable rank one and their implications for $K$-theory and the structure of positive elements play an important role in our analysis of tracially complete $C^*$-algebras and their ultrapowers. We collect the main facts we use below.

\subsubsection{Real rank zero}\label{SS:RR0} Brown and Pedersen defined a unital $C^*$-algebra $A$ to have \emph{real rank zero} if the invertible self-adjoint elements are norm dense in the self-adjoint elements (a non-unital $C^*$-algebra has real rank zero if its minimal unitisation does); (\cite{BrownPedersen:RR0}). Real rank zero is a condition which ensures a wealth of projections, and it passes to direct sums (trivially), matrix amplifications (\cite[Theorem 2.10]{BrownPedersen:RR0}), hereditary subalgebras (\cite[Theorem 2.10]{BrownPedersen:RR0}), and inductive limits (\cite[Proposition 3.1]{BrownPedersen:RR0}). A lot of early work on these $C^*$-algebras was done by Zhang (\cite{Zhang:Riesz-Docomposition,Zhang:PJM}), often using the formulation (equivalent to real rank zero by \cite[Theorem 2.6]{BrownPedersen:RR0}) of property (FS): the set of self-adjoint elements with finite spectrum is is norm dense in $A_{\mathrm{sa}}$.

For stably finite $C^*$-algebras with real rank zero, the Murray von Neumann semigroup of projections fully determines the Cuntz semigroup, and one only needs to study projections not all positive elements.  This goes back to Perera's \cite{Perera97} (see \cite[Corollary 5]{CEI} and \cite[Section 5.5]{APT:MAMS}), and is essentially why the combination of real rank zero and classification of projections gives rise to pureness in the pureness theorem.

We will need the following result along these lines.  The proof is standard, but we have been unable to find a reference in the literature.

\begin{proposition}\label{prop:rr0-comparison}
	Let $A$ is a $C^*$-algebra real rank zero and let $X \subseteq T(A)$ be a compact non-empty set of traces.  The following are equivalent:
	\begin{enumerate}
		\item\label{it:rr0-comp-pos} $A$ has strict comparison with respect to $X$;
		\item\label{it:rr0-comp-proj} For all $n \in \mathbb N$ and projections $p, q \in M_n(A)$, if $\tau(p) < \tau(q)$ for all $\tau \in X$, then $p \precsim q$.
	\end{enumerate}
\end{proposition}

\begin{proof}
	The implication \ref{it:rr0-comp-pos}$\Rightarrow$\ref{it:rr0-comp-proj} is clear.  Conversely, suppose \ref{it:rr0-comp-proj} holds.
    By replacing $A$ with $M_n(A)$, it suffices to show that if $a, b \in A_+$ and $d_\tau(a) < d_\tau(b)$ for all $\tau \in X$, then $a \precsim b$.  
    By \cite[Theorem 5]{CEI},\footnote{Although this result states that $A$ has stable rank one, the proof that real rank zero implies every element is the supremum of classes of projections does not use this hypothesis.} there exist projections $p_i,q_i$ such that $[a]=\sup_i [p_i]$ and $[b]=\sup_i [q_i]$, with the sequences increasing.
    For each $i$, since $\tau(p_i) < d_\tau(a)$ for all $\tau\in X$, by lower semicontinuity and compactness, there exists $\e>0$ such that $\tau(p_i)+\e<d_\tau(a)$ for all $\tau \in X$.
    Then since $d_\tau(a) = \sup_j \tau(q_j)$, by Dini's Theorem, there exists $j$ such that $\tau(p_i) +\e < \tau(q_j)$ for all $\tau \in X$.
    Applying comparison of projections, we have $[p_i] \leq [q_j] \leq [b]$.
    Since $i$ is arbitrary, we may conclude $[a]\leq[b]$.
\end{proof}

\subsubsection{Stable rank one}\label{SS:SR1}

Following Rieffel's seminal paper \cite{Rieffel83}, a unital $C^*$-algebra $A$ is said to have \emph{stable rank one} if the set of invertible elements is norm dense in $A$ (a non-unital $C^*$-algebra has stable rank one if its minimal unitisation does). Stable rank one can be described in terms of the existence of approximate unitary polar decompositions (see \cite[Theorem 5]{Pedersen:JOT}).
\begin{proposition}\label{it:sr1PolarDecomp}
A unital $C^*$-algebra $A$ has stable rank one if and only if for any $a \in A$ and $\e>0$ there exists a unitary $u\in A$ such that $\|a-u|a|\,\|<\e$. 
\end{proposition}

Stable rank one passes to direct sums (trivially), matrix amplifications (\cite[Theorem 3.3]{Rieffel83}), stable isomorphism (\cite[Theorem 3.6]{Rieffel83}), ideals (\cite[Theorem 4.4]{Rieffel83}), quotients (\cite[Theorem 4.3]{Rieffel83}), and inductive limits (\cite[Theorem 5.1]{Rieffel83}). We shall also need that fact that stable rank one passes to corners. On unpacking the notation, this is contained in \cite[Proposition 4.1]{Blackadar:StableRankCorner} and is based on a computation that goes back to \cite{Rieffel83}.\footnote{We found this reference in a comment made by Leonel Robert on Mathoverflow \url{https://mathoverflow.net/questions/381637/stable-rank-one-and-corners-of-c-ast-algebras}. On the same page, the following intuitive argument in the separable case was given by Jamie Gabe: if $p$ is a projection in a separable stable rank one $C^*$-algebra $A$, the ideal $I$ generated by $p$ has stable rank one, and $pAp$ is stably isomorphic to $I$ by Brown's theorem, so has stable rank one.} 

Let $A$ be a unital $C^*$-algebra with stable rank one. Then $A$ is stably finite. 
Moreover, it follows from Rieffel's work (\cite{Rieffel83}) that the Murray--von Neumann semigroup $V(A)$ has cancellation, i.e.\ $p \sim_{\mathrm{MvN}} q$ whenever $p \oplus r \sim_{\mathrm{MvN}} q \oplus r$ for projections $p,q,r \in M_n(A)$ for any $n \in \N$; see \cite[Proposition V.3.1.24]{blackadar2006operator}. Importantly for us, in the presence of real rank zero the converse holds as follows.

\begin{proposition}[Blackadar--Handelman, {\cite[Proposition III.2.4]{BH82}; see also the proof sketched in \cite[Proposition V.3.2.15]{blackadar2006operator}}] \label{prop:RR0CancellationSR1}
    If $A$ has real rank zero and cancellation of projections then $A$ has stable rank one.
\end{proposition}

Stable rank one gives rise to good non-stable $K$-theoretic behaviour. We use the following two facts. The first is very well-known and goes back to Rieffel's original work.  The second is due to Brown, and the final part of \ref{it:BrownTrick} uses a standard functional calculus argument (\cite[Proposition 2.2.6]{rordam:KtheoryBook}).

\begin{theorem}\label{thm:sr1facts}
    Let $A$ be a unital $C^*$-algebra with stable rank one. Then
    \begin{enumerate}
        \item (Rieffel, \cite[Theorem 2.10]{rieffel87}) \label{it:sr1Rieffel} $A$ is $K_1$-bijective, i.e.\ the canonical mapping $U(A)/U_0(A) \to K_1(A)$ is bijective.
        \item (Brown, \cite[Theorem~1]{Brown95}) \label{it:BrownTrick} If $p,q\in A$ are projections that are Murray--von Neumann equivalent, then they are homotopic and therefore unitarily equivalent via a unitary in the path component of the identity.
    \end{enumerate}
\end{theorem}

The second part of Theorem \ref{thm:sr1facts} will be used in conjunction with the following well known lifting result (which follows as one can write a unitary in the path component of the identity as a product of exponentials $e^{ih}$, where $h$ is self-adjoint).
\begin{proposition}\label{Prop:LiftingPathComponent}
    Let $I$ be an ideal in a unital $C^*$-algebra $A$.  If $u\in U(A/I)$ lies in the path component of the identity, then there exists a unitary $v\in A$ lifting $u$.
\end{proposition}

\subsection{Choquet theory}\label{subsec:Choquet}
Given a compact convex subset $X$ of a locally convex topological vector space, we write $\Aff(X)$ for the space of a continuous affine functions $X \rightarrow \R$. We write $\Aff(X)_+$ for the non-negative valued functions in $\Aff(X)$. 
The \emph{barycentre} of a Radon probability measure $\mu \in \mathrm{Prob}(X)$ is the uniquely determined $x \in X$ such that  
\begin{equation}
	\int_X f \, d\mu = f(x), \qquad f \in \Aff(X).
\end{equation}
(These exist by \cite[Proposition~I.2.1]{Alf71}.)
We write $\partial_e X \subseteq X$ for the set of extreme points of $X$.
A Radon measure $\mu$ is said to be \emph{supported} on $\partial_e X$ if for every Baire measurable set $E \subseteq X$ with $\partial_e X \subseteq E$, we have $|\mu|(X\setminus E) = 0$.\footnote{When $X$ is metrisable, the set $\partial_e X$ is a $G_\delta$-set, so $\mu \in \mathrm{Prob}(X)$ is supported on $\partial_e X$ if and only if $\mu(\partial_e X) = 1$.}  
A compact convex set $X$ is called a \emph{Choquet simplex} if for every $x \in X$ there is a unique $\mu \in \mathrm{Prob}(X)$ supported on $\partial_e X$ that has barycentre $x$.\footnote{The Choquet--Bishop--de Leeuw theorem ensures there is always at least one such $\mu$; see  \cite[Theorem~I.4.8 and Corollary~1.4.12]{Alf71}.} 

We now recall the fundamental result connecting Choquet theory and $C^*$-algebras, which we will use without citing further.
\begin{theorem}[{cf.\ \cite[Theorem~3.1.18]{Sak98}}]\label{thm:traces-choquet}
	Let $A$ be a unital $C^*$-algebra. Then $T(A)$ is a Choquet simplex.
\end{theorem}

In the remainder of this subsection, we collect some preliminary results about continuous affine functionals on a Choquet simplex. We recall that a convex subset $Y$ of a compact convex set $X$ is called a \emph{face} if for all $x_1,x_2 \in X$ with $\tfrac{1}{2}(x_1+x_2) \in Y$ it follows that $x_1,x_2 \in Y$. 

The following theorem of Edwards is the fundamental tool for building affine functionals on Choquet simplex.
\begin{theorem}[{see for example \cite[Theorem II.3.10]{Alf71}}]\label{thm:edwards}
	Let $X$ be a Choquet simplex $X$.
    Let $F,G:X \rightarrow \R$. 
    Suppose that $F$ is upper-semicontinuous and convex, $G$ is lower-semicontinuous and concave, and $F \leq G$.
    Then there exists $h \in \Aff(X)$ such that $F \leq h \leq G$. 
\end{theorem}

As a corollary, we obtain the following extension results with control given by a pre-selected affine function.
The proof is a minor adaptation of the proof of \cite[Theorem~II.5.19]{Alf71} (which is essentially the case that $g$ is constant).

\begin{corollary}[{cf.\  proof of \cite[Theorem~II.5.19]{Alf71}}]\label{cor:extending-affine-2}
    Let $X$ be a Choquet simplex and $Y$ be a closed face of $X$. Let $f\in \Aff(Y)_+$ and $g\in \Aff(X)_+$ be such that $f\leq g|_Y$. Then there exists $h\in \Aff(X)_+$, such that $h|_{Y}=f$ and $h\leq g$.
\end{corollary} 
\begin{proof}
    Define $F,G:X \rightarrow \R$ by
    \begin{equation}
        F(x) \coloneqq \begin{cases}
            f(x), & x \in Y, \\
            0,   & x \not\in Y,
        \end{cases}
        \quad \text{and} \quad
        G(x) \coloneqq \begin{cases}
            f(x), & x \in Y, \\
            g(x), & x \not\in Y.
        \end{cases}       
    \end{equation}
    Then $F$ is upper-semicontinuous and convex, $G$ is lower-semicontinuous and concave, and $0 \leq F \leq G$. 
    \footnote{Proof that $G$ is concave and lower-semicontinuous (similar proofs work for $F$): Let $x_1,x_2 \in X$ and $\lambda \in (0,1)$. Set $x_3 = \lambda x_1+(1-\lambda)x_2$. We claim that $G(x_3) \geq \lambda G(x_1)+ (1-\lambda)G(x_2)$. This is clear if $x_3 \not\in Y$ as then $G(x_3) \geq G(x_1), G(x_2)$. If $x_3 \in Y$, then $x_1,x_2 \in Y$ as $Y$ is a face, so we have $G(x_3) = \lambda G(x_1)+ (1-\lambda)G(x_2)$ as $G|_Y = f$ and $f$ is affine.  Clearly $G$ is continuous on $X \setminus Y$. Let $x_0 \in Y$ and suppose $G(x_0) > c$. Then $g(x_0) \geq f(x_0) > c$. By continuity of $g$ and $f$, there is $U$ open in $Y$ such that $f(x) > c$ for $x \in U$ and also $V$ open in $X$ such that $g(x) > c$ for $x \in V$. Choose an open neighbourhood $W$ of $x_0$ in $X$ such that $W \subseteq V$ and $W \cap Y \subseteq U$. Then $G(x) > c$ for $ x \in W$.  
    } 
    By Theorem~\ref{thm:edwards}, there exists $h \in \operatorname{Aff}(X)$ such that $F(x) \leq h(x) \leq G(x)$ for all $x \in X$.
    As $f\leq g|_Y$, we have $G \leq g$, and so $h \leq g$. As $F \geq 0$, we have $h \geq 0$. Since $F(y) = G(y) = f(y)$ for all $y \in Y$, it follows that $h|_Y = f$. 
\end{proof} 

We will apply the previous result in the following form to lift continuous affine maps into the tracial state space of finite-dimensional $C^*$-algebras.

\begin{lemma}\label{lem:aff-lift-valuesp}
    Let $X$ be a Choquet simplex, $Y$ be a closed face of $X$, and let $F$ be a finite-dimensional $C^*$-algebra.  Every continuous affine map $\alpha\colon Y\to T(F)$ admits an extension to a continuous affine map $\beta\colon X\to T(F)$.
    \end{lemma}

\begin{proof}
    Write $F=\oplus_{k=1}^mM_{n_k}$ and let $\alpha\colon Y \rightarrow T(F)$ be given. Then we can decompose $\alpha$ as a sum $\sum_{k=1}^m \alpha_k \mathrm{tr}_{n_k}$, where $\mathrm{tr}_{n_k}$ is the normalised trace on $M_{n_k}$ and for each $k=1,2,\hdots,m$,  $\alpha_k\colon Y\rightarrow [0,1]$ is a continuous affine function such that $\sum_{k=1}^m \alpha_k=1$.
    Inductively, for each $k=1,\dots,m-1$, using Corollary~\ref{cor:extending-affine-2}, we may find $\beta_k \in \Aff(X)_+$ such that $\beta_k|_Y=\alpha_k$ and $\beta_k \leq 1-(\beta_1+\cdots+\beta_{k-1})$. 
    Then set $\beta_m\coloneqq 1-(\beta_1+\cdots+\beta_{m-1})$, which will automatically satisfy $\beta_m|_Y=\alpha_m$, and finally set $\beta\coloneqq \sum_{k=1}^m \beta_k \mathrm{tr}_{n_k}$.
\end{proof}
    
We next record that $\Aff(X)$ has the Riesz interpolation and decomposition properties if (and only if) $X$ is a Choquet simplex (also partly a consequence of Edwards' Theorem). 

\begin{corollary}[{\cite[Proposition II.3.1 and Corollary II.3.11]{Alf71}}]\label{cor:riesz-properties}
    Let $X$ be a compact convex set. Then the following are equivalent.
    \begin{enumerate}
        \item $X$ is a Choquet simplex;
        \item Given $a_1,a_2,b_1,b_2 \in \Aff(X)$ with $a_i \leq b_i$ for $i \in \{1,2\}$. There exists $h \in \Aff(X)$ such that $a_i \leq h \leq b_i$ for for $i \in \{1,2\}$ (Riesz interpolation).
        \item Given $a, b_1,b_2 \in \Aff(X)$ with $0 \leq a \leq b_1 + b_2$. There exist $a_1, a_2 \in \Aff(X)$ such that $a = a_1 + a_2$ and $0 \leq a_i \leq b_i$ for $i \in \{1,2\}$ (Riesz decomposition).
    \end{enumerate}
\end{corollary}
\begin{remark}\label{Rem:MoreRieszDecomp}
    Note that from the Riesz decomposition property, if $X$ is a Choquet simplex and $a_1,\dots,a_n,b_1,\dots,b_m\in\Aff(X)_+$ satisfy $\sum_{i=1}^na_i\leq\sum_{j=1}^mb_j$, then there exist $a_{i,j}\in \Aff (X)_+$ for $i=1,\dots,n$ and $j=1,\dots,m$ such that 
    \begin{equation} a_i=\sum_{j=1}^ma_{i,j} \text{ for all }i\text{ and } \sum_{i=1}^na_{i,j}\leq b_j \text{ for all }j. \end{equation}
    See the proof of \cite[Proposition 2.2]{Goodearl:Book}, for a similar argument.
\end{remark}

\subsection{Tracially complete \cstar-algebras and ultrapowers} \label{subsec:tracially-complete}
In this section, we briefly recall the key definitions and constructions in the theory of tracially complete $C^*$-algebras. 
We refer the reader to the monograph \cite{carrion2023tracially} for the full story. We also discuss the various relevant notions of ultrapower that will be used in this paper. To that end we fix once and for all a free ultrafilter $\omega$ on $\N$.

\subsubsection{Tracially complete \cstar-algebras}\label{subsubsec:tracially-complete}
We begin with the definition of a tracially complete $C^*$-algebras and the uniform 2-norm.
\begin{definition}[{\cite[Definition~3.4]{carrion2023tracially}}]\label{def:TC}
	A \emph{tracially complete $C^*$-algebra} is a pair $(\M,X)$ where $\M$ is a unital $C^*$-algebra and $X \subseteq T(\M)$ is a non-empty compact convex set such that
	\begin{enumerate}
		\item the uniform 2-seminorm $\|\cdot\|_{2,X}$ is a norm on $\M$, 
		\item the $\|\cdot\|$-closed unit ball of $\M$ is $\|\cdot\|_{2,X}$-complete,
	\end{enumerate}
    where $\|\cdot\|_{2,X}:\M \rightarrow [0,\infty)$ is given by $\|a\|_{2,X} \coloneqq \sup_{\tau \in X} \tau(a^*a)^{1/2}$ for $a \in \M$.
\end{definition}
We write $\pi_\tau:\M \rightarrow \mathcal B(L^2(\M,\tau))$ for GNS representation corresponding to the trace $\tau \in X$.  We call the von Neumann algebra $\pi_\tau(\M)''$ the \emph{fibre} of $\M$ at $\tau \in X$. Note that $\tau$ descend to a faithful normal trace on $\pi_\tau(\M)''$. In particular, the fibres of $\M$ are finite von Neumann algebras.

A tracially complete $C^*$-algebra $(\M, X)$ is said to be \emph{factorial} if $X$ is a face in $T(\M)$ (\cite[Definition 3.13]{carrion2023tracially}). This happens if and only if the extreme fibres are factors, i.e.\ $\pi_\tau(\M)''$ is a factor for all $\tau\in\partial_e(X)$ (\cite[Proposition 3.14]{carrion2023tracially}).  We say that $(\M, X)$ is \emph{type} II$_1$ if all its fibres are type II$_1$ von Neumann algebras.

\begin{remark}\label{2909_rem_amplification}
    In a couple of places in this paper, we will make use of matrix amplifications of tracially complete $C^*$-algebras. Given a tracially complete $C^*$-algebra $(\M, X)$ and $n \in \N$, we identify $M_n \otimes \M$ and $M_n(\M)$ in the usual way and write $X_n$ for the set of traces on $M_n(\M)$ of the form $\mathrm{tr}_n \otimes \tau$ where $\tau \in X$. Then $(M_n(\M), X_n)$ is a tracially complete $C^*$-algebra. (Indeed, this is a special case of tensor products; see \cite[Section 3.5]{carrion2023tracially}.) Since $\mathrm{tr}_n$ is the normalised trace on $M_n$, we have that $\|x\|_{2,X} = \sqrt{n}\|x \oplus 0_{n-1}\|_{2,X_n}$ for all $x \in \M$ where we view $\M \subseteq M_n(\M)$ by embedding $\M$ in the top left corner.
\end{remark}

\subsubsection{Uniform tracial completions}\label{subsubsec:completions-tracial}
Let $A$ be a $C^*$-algebra with $T(A)$ non-empty and compact.
The \emph{uniform tracial completion} of $A$ with respect to $T(A)$ is the $C^*$-algebra
\begin{equation}\label{eq:completions-tracial}
	\completion{A}{T(A)} \coloneqq \frac{\{(a_n)_{n=1}^\infty \in \ell^\infty(A): (a_n)_{n=1}^\infty \text{ is }\|\cdot\|_{2,T(A)}\text{-Cauchy}\}}{\{(a_n)_{n=1}^\infty \in \ell^\infty(A): (a_n)_{n=1}^\infty \text{ is }\|\cdot\|_{2,T(A)}\text{-null}\}}.
\end{equation}
There is a \emph{canonical map} $A \rightarrow \completion{A}{T(A)}$ defined by considering the image of constant sequences in the quotient \eqref{eq:completions-tracial}.
This construction goes back to Ozawa in his seminal work on $W^*$-bundles (\cite{ozawa2008weak}). 

Every $\tau \in T(A)$ extends continuously to a trace $\tilde{\tau}$ on $\completion{A}{X}$ given by $\tilde{\tau}((a_n)_{n=1}^\infty) \coloneqq \lim_{n\rightarrow\infty} \tau(a_n)$. 
We identify $\tau$ with $\tilde{\tau}$ and view $T(A)$ as a subset of $T(\completion{A}{T(A)})$. The pair $(\completion{A}{T(A)}, T(A))$ is a factorial tracially complete $C^*$-algebra (see \cite[Proposition~3.23]{carrion2023tracially}).

\subsubsection{Ultrapowers of tracially complete \cstar-algebras}\label{subsubsec:ultrapowers-tc}
The \emph{ultrapower} $(\M^\omega,X^\omega)$ of the tracially complete $C^*$-algebra $(\M,X)$ is given by
\begin{equation}
    \M^\omega \coloneqq \frac{\ell^\infty(\M)}{\{(a_n)_{n=1}^\infty \in \ell^\infty(A): \lim_{n\to\omega}\|a_n\|_{2,X} = 0\}}.
\end{equation}

For every sequence of traces $(\tau_n)_{n=1}^\infty$ in $X$, there is a \emph{limit trace} on $\M^\omega$ given by $(a_n)_{n=1}^\infty \mapsto \lim_{n\to\omega} \tau_n(a_n)$, and we set $X^\omega \subseteq T(\M)$ to be the weak$^*$-closure of the set of all limit traces in $T(\M^\omega)$. The theory of ultrapowers (and more general reduced products) of tracially complete $C^*$-algebras is set out in \cite[Section 5.1]{carrion2023tracially}.  

\subsubsection{Ultrapowers and tracial ultrapowers of \cstar-algebras}\label{subsubsec:ultrapowers-tracial}
Given a $C^*$-algebra $A$. The (norm) \emph{ultrapower} of $A$ is given by 
\begin{equation}
    A_\omega \coloneqq \frac{\ell^\infty(A)}{\{(a_n)_{n=1}^\infty: \lim_{n\to\omega} \|a_n\| = 0 \}}. 
\end{equation}
Ultrapowers with respect to uniform trace norms were studied before the theory of tracially complete $C^*$-algebras was developed (as the quotient of the norm-ultrapower by the trace--kernel ideal; see \cite{KR14}, which builds on Matui and Sato's \cite{MS:Acta}, and become explicit in \cite{CETWW}). Given a $C^*$-algebra $A$ with $T(A)$ non-empty and compact, one can define the \emph{tracial ultrapower} of $A$ by
\begin{equation}
    A^\omega \coloneqq \frac{\ell^\infty(A)}{\{(a_n)_{n=1}^\infty: \lim_{n\to\omega} \|a_n\|_{2,T(A)} = 0 \}}. 
\end{equation}
Analogously to the previous subsection, one can also define set of limit traces $T_\omega(A) \subseteq T(A_\omega)$ and its weak$^*$ closure $\overline{T_\omega(A)}$. By abuse of notation, we also write $T_\omega(A)$ and $\overline{T_\omega(A)}$ for the corresponding set of traces on $A^\omega$. 

The uniform tracial ultrapower is naturally a quotient of the norm-ultrapower by the \emph{trace-kernel ideal} 
\begin{equation} 
 J_A:=\{a\in A_\omega: \tau(a^*a)=0\,\,\text{for all }\tau\in T_\omega(A)\},   
\end{equation} 
giving the \emph{trace-kernel extension} 
\begin{equation}\label{2609_deftrkerEx}
    0 \rightarrow J_A \xrightarrow{j_A} A_\omega \xrightarrow{q_A} A^\omega\rightarrow 0. 
\end{equation} 

The connection between uniform tracial ultrapowers and the theory of tracially complete $C^*$-algebras is given by the following proposition. 
\begin{proposition}[{\cite[Proposition 5.6]{carrion2023tracially}}]\label{prop:tracial-ultrapower-is-TC}
    Let $A$ be a $C^*$-algebra with $T(A)$ non-empty and compact. 
    Let $(\M,X)$ be the tracial completion of $A$ with respect to $T(A)$.
    Then canonical map $A \rightarrow \M$ induces an isomorphism $(A^\omega, \overline{T_\omega(A)}) \cong (\M^\omega, X^\omega)$.
    In particular, $A^\omega$ is unital and $(A^\omega, \overline{T_\omega(A)})$ is a tracially complete $C^*$-algebra.
\end{proposition}
 The proof of Proposition \ref{prop:tracial-ultrapower-is-TC} amounts to proving surjectivity of the induced map  $A^\omega \rightarrow \M^\omega$, which follows from a $\|\cdot\|_{2,X}$-variant of the Kaplansky density theorem.

\subsubsection{Traces on tracial ultrapowers}\label{subsubsec:ultrapowers-traces}
In general, not every trace on an tracial ultrapower will be a weak$^*$ limit of limit traces (the example we know of is Vaccaro's \cite{Vaccaro:MJM}, which is in the type I setting; we do not know of an example where this fails for ultrapowers of factorial II$_1$ tracially complete $C^*$-algebras).  In our situation, we will control the traces on tracial ultrapowers, and ultrapowers of tracially complete $C^*$-algebras through strict comparison, via the following, essentially standard, result.

\begin{proposition} 
\label{strict_compa_A_ultrap} 
    Let $A$ be a $C^*$-algebra with $T(A)$ non-empty and compact. Suppose $A$ has strict comparison with respect to $T(A)$. 
    Then $A^\omega$ has strict comparison with respect to $\overline{T_\omega(A)}$, and $T(A^\omega)=\overline{T_\omega(A)}$.  In particular if $(\M,X)$ is a factorial II$_1$ tracially complete $C^*$-algebra and $\M$ has strict comparison with respect to $X$, then $\M^\omega$ has strict comparison with respect to $X^\omega$.
\end{proposition}
\begin{proof}
    Note that since $T(A)$ is compact, $A^\omega$ is unital, and so $\overline{T_\omega(A)}$ is a compact subset of $T(A^\omega)$. It is standard that the norm ultrapower $A_\omega$ has strict comparison with respect to $\overline{T_\omega(A)}$ (see \cite[Lemma 1.23]{BBSTWW}, which only uses the unitality assumption through the compactness of $\overline{T_\omega(A)}$). This passes to the quotient $A^\omega$.\footnote{If $a,b\in M_n(A^\omega)_+$ have $d_\tau(a)<d_\tau(b)$ for all $\tau\in \overline{T_\omega(A)}$, then  positive lifts $\tilde{a},\tilde{b}\in M_n(A_\omega)_+$ and abusing notation to regard ($T_\omega(A)$ as traces on $A_\omega$) will have $d_\tau(\tilde{a})<d_\tau(\tilde{b})$ for all $\tau\in\overline{T_\omega(A)}$. Strict comparison of $A_\omega$ with respect to $\overline{T_\omega(A)}$ gives $\tilde{a}\precsim \tilde{b}$ in $A_\omega$, and hence $a\precsim b$ in $A^\omega$.}  Then $T(A^\omega)=\overline{T_\omega(A)}$ is a consequence of Proposition~\ref{prop:computing-traces}.
    The second part follows from applying the first part to $A=\M$, which will have $T(\M)=X$ by Proposition~\ref{prop:computing-traces}, and hence $X^\omega=\overline{T_\omega(\M)}$.   
    The fact that $T(A^\omega)=\overline{T_\omega(A)}$ is a consequence of Proposition~\ref{prop:computing-traces}.
\end{proof}

\subsubsection{Kirchberg's epsilon test}
A powerful tool in the study of ultrapowers is the ability to convert various approximate statements into exact statements.  This can be formulated through reindexing arguments, or through Łoś's theorem (viewing things through the lens of continuous model theory). In the operator algebra context a standard way to access such results is through the following lemma of Kirchberg known as Kirchberg's epsilon test. 

\begin{lemma}[Kirchberg's epsilon test, {\cite[Lemma A.1]{Kir06}}]
	\label{lem:EpsTest}
	Let $(\Omega_n)_{n=1}^\infty$ be a sequence of non-empty sets, and for $k,n\in\mathbb N$, let $h^{(k)}_n\colon \Omega_n\rightarrow [0,\infty]$ be a function.  Define functions $h_\omega^{(k)}\colon \prod_{n=1}^\infty \Omega_n \to [0,\infty]$ by
	\begin{equation}
		\label{eq:EpsTest1}
		h_\omega^{(k)}(x_1,x_2,\dots) \coloneqq \lim_{n\to\omega} h^{(k)}_n(x_n). 
	\end{equation}
	If for every $\epsilon>0$ and $k_0\in\mathbb N$, there exists $x\in \prod_{n=1}^\infty \Omega_n$ such that $h^{(k)}(x)<\epsilon$ for all $k\in\{1,\dots,k_0\}$, then there exists $y\in \prod_{n=1}^\infty \Omega_n$ such that $h^{(k)}(y)=0$ for all $k\in\mathbb N$.
\end{lemma}

The following elementary observation (which we will not mention explicitly when we use) is particularly useful when reformulating uniform tracial conditions in preparation for an application of Kirchberg's epsilon test.\footnote{The reference \cite[Lemma 5.3]{carrion2023tracially} is actually just the final statement of the following, but its proof instantly adapts.}

\begin{lemma}[{cf.\ \cite[Lemma 5.3]{carrion2023tracially}}]\label{lem:trace-trick}
    Let $A$ be a $C^*$-algebra, $X$ be a non-empty compact subset of $T(A)$, and $X^\omega$ be the weak$^*$ closure of the limit traces coming from $X$.
    Let $(a_n)_{n=1}^\infty \in \ell^\infty(A)$. Then 
    \begin{equation}
        \lim_{n \to \omega} \sup_{\tau \in X} |\tau(a_n)| = \sup_{\tau_\omega \in X^\omega} |\tau_{\omega}((a_n)_{n=1}^\infty)|. 
    \end{equation}
    Moreover, $\lim_{n \to \omega} \|a_n\|_{2,X} = \|(a_n)_{n=1}^\infty\|_{2,X^\omega}$. 
\end{lemma}

\begin{remark}
    We shall typically suppress the details of our applications of Kirchberg's epsilon test. Many examples can be found in the literature (particularly in \cite{BBSTWW}). For an example application in the tracially complete framework, see \cite[Lemma 5.4]{carrion2023tracially}.   However, we would like to make a crucial observation. The output of Kirchberg's epsilon test when working with tracial ultrapowers will always be estimates that hold for all \emph{limit traces}, and then by continuity for the closure of this set as well. If $T(A^\omega)$ contains traces not in $\overline{T(A_\omega)}$ this argument says nothing about those traces (notwithstanding that a positive element that vanishes on $T_\omega(A)$ is zero in the tracial ultrapower).
    In our applications we will typically use the last part of Proposition~\ref{strict_compa_A_ultrap} ensuring that there are no such traces.
\end{remark}

\subsubsection{Continuous affine functions on the trace space}
Let $A$ be a $C^*$-algebra with $T(A) \neq \emptyset$. Every $a \in A$ induces a continuous affine function $\hat{a}:T(A) \rightarrow \C$ via $\tau \mapsto \tau(a)$. Note that $\hat{a}$ is real-valued whenever $a$ is self-adjoint, and $\hat{a}$ is positive whenever $a$ is positive. The following well-known result, which goes back to the work of Cuntz and Pedersen in \cite{CP79}, asserts that every continuous affine function arises in this manner.

\begin{proposition}[{\cite[Proposition 2.7]{carrion2023tracially}}, cf.\ {\cite[Proposition 2.7]{CP79}} and {\cite[Lemma~6.2]{KR14}}]
	\label{prop:CP}
	Let $A$ be a unital $C^*$-algebra and let $f\colon T(A) \to \mathbb R$ be a continuous affine function.
	Then for any $\epsilon>0$, there is a self-adjoint element $a \in A$ such that
	\begin{equation} \label{eqn:CP}
        \|a\| \leq \|f\|_\infty + \epsilon \quad \text{and} \quad \tau(a) = f(\tau), \quad \tau \in T(A).
    \end{equation}
	If $f$ is strictly positive, i.e.\ $f(\tau) > 0$ for all $\tau \in T(A)$, then we may assume $a \geq 0$.
\end{proposition}

Kirchberg's epsilon test gives ultrapower version of this result, where the $\e$ in the norm bound of \eqref{eqn:CP} can be dropped and strict positivity can be weakened to positivity. 

\begin{proposition}\label{prop:CP-ultrapowers}
    Let $(\M,X)$ be a factorial tracially complete $C^*$-algebra and let $f\in \Aff(T(\M^\omega))$. Then there exists self-adjoint element $a \in \M^\omega$ such that 
    \begin{equation} \label{eqn:CP-ultrapowers}
       \|a\| \leq \|f\|_\infty \quad \text{and} \quad \tau(a) = f(\tau), \quad \tau \in X^\omega. 
    \end{equation}
    If $f$ is positive, then we may assume $a \geq 0$.
\end{proposition}
\begin{proof}
 The first statement is an immediate application of Kirchberg's epsilon test (using Lemma \ref{lem:trace-trick}). When $f\geq 0$, we translate and consider $g\coloneqq f-\|f\|_\infty/2$ so $\|g\|_\infty=\|f\|_\infty/2$. For a self-adjoint $b\in\M$ with $\|b\|\leq\|g\|_\infty$ which realises $g$, we have $a\coloneqq b+\|f\|_\infty/2\geq 0$ (recalling that $\M$ is unital) has $\|a\|\leq\|f\|_\infty$ and $a$ realises $f$.
\end{proof}

\subsection{Weak CPoU} The following theorem is our tool for performing partition-of-unity arguments using approximately central positive elements in a fashion which can be thought of as approximately subordinate to the affine open cover in a suitable sense. It is a special case of \cite[Theorem 6.16]{carrion2023tracially}, obtained by taking $q=1$ in that result.

\begin{theorem}[{Factorial tracially complete $C^*$-algebras have weak CPoU; see \cite[Theorem 6.16]{carrion2023tracially}}]\label{Thm:WCPoU} 
Let $(\mathcal M,X)$ be a
factorial tracially complete $C^*$-algebra. Let $\delta>0$ and $S\subset\mathcal M^\omega$ be $\|\cdot\|_{2,X^\omega}$-separable.  Suppose $a_1,\dots,a_N\in \mathcal M_+$ satisfy
\begin{equation}
\sup_{\tau\in X}\min_{1\leq n\leq N}\tau(a_n)<\delta.
\end{equation}
Then there exists positive contractions $e_1,\dots,e_N\in \mathcal M^\omega_+\cap \mathcal S'$ with $\sum_{n=1}^Ne_n=1$ and 
\begin{equation}\label{WCPoU:MainEquation}
\tau(a_ne_n)\leq \delta \tau(e_n),\quad n=1,\dots,N.
\end{equation}
\end{theorem}

In \cite{CETWW}, together with Castillejos and Winter, SE, AT and SW introduced the condition of having \emph{complemented partitions of unity (CPoU)} which requires the stronger conclusion that the $e_n$ in \eqref{WCPoU:MainEquation} are pairwise orthogonal (so projections).  This is a strong condition on a factorial tracially complete $C^*$-algebra (typically obtained from property $\Gamma$ via the arguments of \cite{CETWW,carrion2023tracially}). When it applies, CPoU gives a powerful tool for local-to-global arguments in tracially complete $C^*$-algebras, underpinning the existing implications of the structure theorem as well as being instrumental in the modern abstract approach to $C^*$-algebraic classification (see \cite{CETWW,CETW:IMRN,carrion2023tracially,CGSTW}).

The intuition behind \eqref{WCPoU:MainEquation}, articulated in the introduction to \cite{CETWW}, is that the affine function $\widehat{e}_n$ should be small on those traces where $\tau(a_n)\approx 1$, and in this way we can think of $\widehat{e}_n$ as loosely subordinate to the complement $1-\widehat{a}_n$. 

\subsection{Conditional expectations} 
In this section, we collect several standard facts about relative commutants of finite-dimensional subalgebras and conditional expectations onto these relative commutants. 

\begin{definition}\label{def:CondExp} 
Let $\M$ be a $C^*$-algebra. 
For a finite-dimensional $C^*$-algebra $F$ and a $^*$-homomorphism $\phi\colon F \to \M$, we use $E_{\phi(F)'}$ to denote the canonical conditional expectation $\M \to \M \cap \phi(F)'$; this is given by
\begin{equation} \begin{split} 
\label{eq:CondExp}
E_{\phi(F)'}(a) &\coloneqq \int_{u \in \mathcal U(F)} \phi(u)a\phi(u)^*\,du + (1-\phi(1_{F})) a (1-\phi(1_{F})).
\end{split} 
\end{equation} 
The integral part in the above is equal to an average over any finite subgroup of the unitary group of $F$ that generates $F$. If $F=M_k$ then the integral part is equal to $\frac1k \sum_{i,j} \phi(e_{ij})a\phi(e_{ji})$.
\end{definition} 

\begin{remark} 
\label{rem:cond_exp} \hfill
\begin{enumerate} 
\item 
\label{rem:con_exp_1} 
The conditional expectation $E_{\phi(F)'}$ defined in Definition~\ref{def:CondExp} is trace-preserving. Indeed, for any $\tau \in X$ and $a\in \mathcal{M}$, we have 
\begin{equation} 
\begin{split} 
\tau (E_{\phi(F)'} (a)) &= \int_{u\in \mathcal{U}(M_{d})} \tau (\phi(u) a \phi(u)^{*}) du + \tau ((1-\phi(1_{F}))a)\\ 
&= \tau (\phi(1_{k})a) + \tau ((1-\phi(1_{F}))a) = \tau (a). 
\end{split} 
\end{equation} 
\item 
\label{rem:con_exp_2} 
Let $\M$ be a $C^*$-algebra and let $\phi: F\rightarrow \M$ be a unital $^*$-homomorphism from the finite-dimensional algebra $F$, then for any positive element $a\in \M$, 
\begin{equation} 
\label{cond_exp_bound} 
a \leq C_F E_{\phi(F)'} (a), 
\end{equation} 
for a constant $C_F$ depending only on $F$.
Namely, one can take $C_F$ to be the size of a finite subgroup $G_F$ of unitaries which generate $F$ as a $C^*$-algebra, and then we have $E_{\phi(F)'} = \frac1{C_F} \sum_{u \in G_F} \phi(u)a\phi(u)^*$.
\end{enumerate} 
\end{remark} 

\begin{lemma}\label{Lem:RelCom} 
Let $B$ be a unital $C^*$-algebra and $F$ a finite-dimensional $C^*$-subalgebra containing the unit of $B$.  Let $r_1,\dots,r_S$ be the minimal central projections of $F$ (i.e.\ $F\cong \oplus_{s=1}^{S} Fr_{s}$) and let $e_1,\dots,e_S$ be minimal projections in $F$ with $e_s\leq r_s$. Then
\begin{equation}
B\cap F'=\oplus_{s=1}^S (r_sBr_s \cap (Fr_s)')\cong \oplus_{s=1}^S e_sBe_s.
\end{equation}
In particular if $B$ has real rank zero so too does $B\cap F'$.
\end{lemma} 

\section{Projections in pure tracial ultrapowers} \label{sec:tracial-ultrapowers}

This section represents steps \ref{Step.1}-\ref{Step.4} from the outline in Section~\ref{subsec:organisation}. Our main goal is to show that strict comparison of $(\M^\omega,X^\omega)$ and divisibility of the unit $1_{\M^\omega}$ imply further regularity properties for $(\M^\omega,X^\omega)$: classification of projections, real rank zero, stable rank one, and pureness; see Proposition~\ref{S4.2:Prop}.

However, we begin by establishing some sufficient conditions on a $C^*$-algebra $A$ for its tracial ultrapower to have a divisible unit and strict comparison with respect to (the weak$^*$ closure of) limit traces. This proposition will be used to access the results of this section when it comes to proving the main theorems from the introduction.

\begin{proposition}
\label{prop:sc_divisible_tracial_ultra} 
    Let $A$ be a $C^*$-algebra such that $T(A)$ is compact and non-empty. Suppose that $A$ has strict comparison with respect to $T(A)$, and the class of a strictly positive element of $A$ is almost divisible in $\Cu(A)$. Then $A^\omega$ has strict comparison with respect to $\overline{T_\omega(A)}$ and the unit of $A^\omega$ is divisible.
\end{proposition}
\begin{proof}
By Proposition~\ref{strict_compa_A_ultrap}, $A^\omega$ has strict comparison with respect to $\overline{T_\omega(A)}$.
We show divisibility of the unit of $A^\omega$, only using almost divisibility of the Cuntz class of a strictly positive element of $A$.
Fix $n \in \N$ and $\e>0$.
We will find a c.p.c.\ order zero map $\phi\colon M_n \to A$ such that $\tau(\phi(1_n))>1-\e$ for all $\tau$.
By Kirchberg's epsilon test, it then follows that there exists a c.p.c.\ order zero map $\phi\colon M_n \to A^\omega$ such that $\tau(\phi(1_n))=1$ for all $\tau \in T_\omega(A)$. This $\phi$ must be unital and so is a $^*$-homomorphism by \cite[Proposition 1.4]{BBSTWW}. Therefore, $1_{A^\omega}$ is divisible.

Let $h\in A_+$ be a strictly positive element. Then $d_\tau(h)=1$ for all $\tau \in T(A)$.
By a standard application of Dini's theorem,\footnote{Note that $(\tau \mapsto d_\tau((b-1/n)_+))_{n=1}^\infty$ is an increasing sequences of lower-semicontinuous functions converging pointwise to the continuous function 1 on the compact set $T(A)$.} there exists $\delta>0$ such that 
\begin{equation} \label{eq:deltaChoice}
d_\tau((h-\delta)_+)>1-\frac{1}{3}\e, \quad \tau\in T(A). 
\end{equation}
Let $k\in\N$ be such that 
\begin{equation}\label{eq:kChoice1} 
    \frac1{nk+1}<\frac{1}{3}\e.
\end{equation}
By almost divisibility of $[h]$, there exists $[b]\in \text{Cu}(A)$ such that
\begin{equation}\begin{split}
(nk)[b] &\leq [h] \label{eq:divide_nk-1},
\shortintertext{and}
[(h-\tfrac{1}{2}\delta)_+] &\leq (nk+1)[b].  
\end{split}\end{equation} 
Since $[(h-\delta)_+] \ll [(h-\tfrac{1}{2}\delta)_+]$, there exists $\eta > 0$ such that 
\begin{equation} 
\label{upper_bound_decrease} 
[(h-\delta)_+] \leq (nk+1) [(b-\eta)_{+}]. 
\end{equation}
By \cite[Lemma 2.4]{RR_divisible},\footnote{Note that there is a 1-1 correspondence between c.p.c.\ order zero maps $M_n \rightarrow B$ and $^*$-homomorphisms $C_0(0,1] \otimes M_n \rightarrow B$ for any $C^*$-algebra $B$ by \cite[Corollary 4.1]{WZ:Munster}.} it follows from \eqref{eq:divide_nk-1} that there exists a c.p.c.\ order zero map $\psi\colon M_{n}\rightarrow \overline{h(A\otimes\mathcal{K})h}\subseteq A$ (as $h\in A$) such that 
\begin{equation}\label{eq:kChoice}[\psi(e_{11})] = k[(b-\eta)_{+}].\end{equation}
Let $\tau \in T(A)$. Then 
\begin{equation} \label{eq:d-tau-estimate-1}\begin{array}{rcl}
d_{\tau} (\psi(1_{n})) &=& nd_\tau(\psi(e_{11})) \\
&\stackrel{\eqref{eq:kChoice}}{=}& nk d_{\tau}((b-\eta)_{+}) \\
&\stackrel{\eqref{upper_bound_decrease}}\geq& \frac{nk}{nk+ 1} d_{\tau}((h-\delta)_+) \\
&\stackrel{\eqref{eq:deltaChoice},\eqref{eq:kChoice1}}\geq& 1-\frac{2}{3}\e. 
\end{array}\end{equation} 

Applying Dini's theorem to the increasing sequence of continuous functions $f_m\colon T(A) \to \R$ given by $f_m(\tau) = \tau(\psi(1_n)^{1/m})$, there exists $m \in \N$ such that $\tau(\psi(1_n)^{1/m})\geq 1-\e$
for all $\tau\in T(A)$.
Then the c.p.c.\ order zero map $\phi\coloneqq \psi^{1/m}$ (defined via the order zero functional calculus of \cite{WZ:Munster}) has the required properties.
\end{proof}

\subsection{Classification of projections}\label{subsec:existence-and-uniqueness}

We begin with the notion of classification of projections for tracially complete $C^*$-algebras from \cite[Section 7.3]{carrion2023tracially}. Although this notion is not formally defined there, the results in that section naturally lead to the following definition.

\begin{definition}\label{def:classification-of-projections}
Let $(\mathcal{M}, X)$ be a II$_1$ tracially complete $C^*$-algebra.\footnote{We restrict to the II$_1$ setting because existence of projections should be defined differently outside the II$_1$ setting.} Then 
\begin{enumerate} 
\item $(\mathcal{M}, X)$ has \emph{comparison of projections} if, for any two projections $p, q \in \M$ such that $\tau(p) \leq \tau (q)$ for all $\tau \in X$, we have $p \precsim q$; 
\item $(\mathcal{M}, X)$ has \emph{existence of projections} if for any $f\in \textrm{Aff}(X)$ with $0\leq f \leq 1$, there exists a projection $p\in \mathcal{M}$ such that $\tau (p) = f(\tau)$ for all $\tau \in X$; 
\item $(\mathcal{M}, X)$ has \emph{classification of projections} if both (i) and (ii) hold. 
\end{enumerate} 
Occasionally, we will use a much weaker version of (ii), and say that \emph{all constant trace projections exist} in $(\M,X)$ when, for all $0\leq\gamma\leq1$, there exists a projection $p\in \mathcal M$ with $\tau(p)=\gamma$ for all $\tau\in X$.
\end{definition}

Classification of projections was established for II$_1$ factorial tracially complete $C^*$-algebras with CPoU in \cite[Theorems 7.17 and 7.18]{carrion2023tracially}.  Forthcoming work (\cite{JRW:inPrep}) of Jekel, Ryder, and SW will give new methods for obtaining classification of projections in a variety of $W^*$-bundles which do not have CPoU, including all trivial bundles.

Note that comparison of projections immediately implies \emph{uniqueness of projections}, which means that projections with the same trace are unitarily (and hence Murray--von Neumann) equivalent, as follows.
\begin{proposition}\label{Prop:MvN}
    Let $(\mathcal M,X)$ be a tracially complete $C^*$-algebra with comparison of projections.
    Suppose $p,q \in\M$ are projections with $\tau(p) = \tau(q)$ for all $\tau \in X$, then there exists a unitary $u\in\M$ with $upu^*=q$.
\end{proposition}
\begin{proof}
    Since $\tau(p) \leq \tau(q)$ and $\tau(q) \leq \tau(p)$ for all $\tau \in X$, it follows that $p \precsim q$ and $q \precsim p$ by comparison of projections. Since $\M$ has a faithful set of traces, it is stably finite. It therefore follows that $p \sim_{\mathrm{MvN}} q$.  Since the same argument gives $1-p \sim_{\mathrm{MvN}} 1-q$, it follows that $p$ and $q$ are unitarily equivalent.
\end{proof}
Using Proposition \ref{Prop:MvN}, classification of projections immediately gives a classification of finite-dimensional $C^*$-subalgebras (consider first the minimal projections in the centre, and then the matrix units for each simple summand).\footnote{This is recorded for the factorial CPoU case in \cite[Lemma 9.5]{carrion2023tracially}, but that proof only uses CPoU through the existence and comparison of projections.}

\begin{proposition}\label{Class:FD}
    Let $(\mathcal M,X)$ be a II$_1$ tracially complete $C^*$-algebra with classification of projections.  Then $(\mathcal M,X)$ has classification of  finite-dimensional subalgebras, in the sense that for a finite-dimensional $C^*$-algebra $F$:
    \begin{enumerate}
        \item \label{it:ClassFDExistence}
        given a projection $p\in \M$ and continuous affine map $\alpha\colon X\to T_{\leq 1}(F)$ such that $\alpha(\tau)(1_F) \leq \tau(p)$ for all $\tau \in X$, there exists a $^*$-homomorphism $\theta\colon F\to p\mathcal Mp$ with $\tau \circ \theta=\alpha (\tau)$ for all $\tau\in X$; if $\alpha(\tau)(1_F)=\tau(p)$ for all $\tau \in X$ then $\theta(1_F)=p$; 
        \item \label{it:ClassFDUniqueness}
        two $^*$-homomorphisms $\phi,\psi\colon F\to\mathcal M$ are unitarily equivalent if and only if $\tau\circ\phi=\tau\circ\psi$ for all $\tau\in X$.
    \end{enumerate}
\end{proposition}

\subsection{Classification of projections, real rank zero, and stable rank one for tracial ultrapowers}

In this subsection, we establish that $(\M^\omega,X^\omega)$ has classification of projections and real rank zero when it has strict comparison with to $X^\omega$ and divisible unit.  

\begin{proposition}\label{S4.2:Prop}
    Let $(\M,X)$ be a II$_1$ factorial tracially complete $C^*$-algebra such that $\mathcal M^\omega$ has strict comparison with respect to $X^\omega$ and $\M^\omega$ has a divisible unit, i.e.\ $M_n\hookrightarrow \M^\omega$ unitally for all $n\in\mathbb N$.
    Then:
    \begin{enumerate}
        \item\label{it:classif-proj-omega} $\M^\omega$ has classification of projections, and hence also of finite-dimensional subalgebras;
        \item\label{it:rr0-omega} $\M^\omega$ has real rank zero;
        \item\label{it:sr1-omega} $\M^\omega$ has stable rank one;
        \item\label{it:pure-omega} $\M^\omega$ is pure.
    \end{enumerate}
\end{proposition}

\begin{remark}
    We note that the hypotheses that $(\M,X)$ is type II$_1$ and factorial are 
    is redundant under the assumptions that $\M^\omega$ has strict comparison with respect to $X^\omega$ and $\M^\omega$ has a divisible unit. 
    \begin{enumerate}
        \item Divisibility of the unit and order zero lifting give order zero maps $M_n\to\M$ which are not contained in the kernel of any $\pi_\tau$ for all $\tau \in X$. Accordingly, each $\pi_\tau(\M)''$ is type II$_1$.
        \item Strict comparison of $\M^\omega$ with with respect to $X^\omega$ implies $T(\M^\omega)=X^\omega$ by Proposition~\ref{prop:computing-traces}.
        Suppose $\sigma\in T(\M)$ is in the face generated by $X$.  Then there is a constant $C > 0$ and a trace $\tau \in X$ such that $\tau \leq C\sigma$, and in particular, $\sigma$ is $\|\cdot\|_{2, X}$-continuous.
        Then the limit trace $\sigma_\omega$ given by just using $\sigma$ is in $T(\M^\omega)=X^\omega$.
        The continuous affine map $T(\M^\omega) \rightarrow T(\M)$ given by $\tau \mapsto \tau|_\M$ maps $X^\omega$ into $X$ since $X$ is closed.
        Therefore, $\sigma = \sigma_\omega|_\M \in X$.
    \end{enumerate}
    Nevertheless, we will include the type II$_1$ and factorial hypotheses in the statements of results in this section in accordance with Definition~\ref{def:classification-of-projections}.
\end{remark}

Comparison of projections is essentially immediate from strict comparison and divisibility of the unit; the only detail is to add an extra gap by adding a small trace projection in the $2\times 2$ matrix amplification $(M_2(\M^\omega),X^\omega_2)$, which, by definition, also has strict comparison. 

\begin{proof}[Proof of Proposition \ref{S4.2:Prop}\ref{it:classif-proj-omega}: comparison of projections]
Let $p, q \in \M^\omega$ be projections such that $\tau(p) \leq \tau(q)$ for all $\tau\in X^\omega$.
Let $n\in\mathbb N$. As $\M^\omega$ has divisible unit, there exists a projection $r\in \M^\omega$ of constant trace $\tfrac{1}{n}$.  
Then, in $M_2(\M^\omega)$, we have
\begin{equation}
\tau(p\oplus 0)<\tau(q\oplus r)
\end{equation}
for every $\tau\in X_2^\omega$. Since $M_2(\M^\omega)$ has strict comparison with respect to $X^\omega_2$, there exists $v\in M_2(\M^\omega)$ with $v^*v=p \oplus 0$ and $vv^*\leq q\oplus r$.  Let $w \in \mathcal M_2^\omega$ be such that  $w \oplus 0 =(q \oplus 0) v(p \oplus 0) \in\mathcal M^\omega$. Then $qw=w$ and
\begin{equation}
    \|w^*w-p\|_{2,X^\omega}^2
    =2\|v^*(q\oplus 0)v-v^*(q\oplus r)v\|_{2,X^\omega_2}^2=\frac{2}{n}.
\end{equation}
Kirchberg's epsilon test provides $w\in \M^\omega$ with $w^*w=p$ and $qw=w$. Therefore, $p\precsim q$.
\end{proof}

For the existence of projections, we use the tracial projectionisation trick from \cite[Lemma 2.4]{CETWW}; this is powered by the divisibility of the unit which allows us to work in the commutant of large matrix subalgebras. This is the point where we go from being able to divide the unit of $\M^\omega$ to realising all continuous rank functions.

\begin{proof}[Proof of Proposition \ref{S4.2:Prop}\ref{it:classif-proj-omega}: existence of projections]
Let $f \in \mathrm{Aff}(X)$ be such that $0 \leq f \leq 1$.  Since $\M^\omega$ has strict comparison with respect to $X^\omega$, Proposition~\ref{prop:computing-traces} gives $T(\M^\omega)=X^\omega$. By Proposition \ref{prop:CP-ultrapowers}, there exists a positive contraction $a\in \M^\omega$  with $\tau(a)=f(\tau)$ for all $\tau\in T(\M^\omega)$.

Let $n\in\mathbb N$. We'll now produce $p\in \M^{\omega}$ such that $\|p-p^2\|_{2,X^\omega}^2 \leq \frac{1}{2^n}$ and $\tau (p) = \tau (a)$ for any $\tau\in X^\omega$.
Once this is done, Kirchberg's epsilon test produces a projection $p\in\M^\omega$ with $\tau(p)=f(\tau)$ for all $\tau\in X^\omega$, as required.

Using divisibility of the unit, fix a unital embedding $\iota\colon M_{n} \to \M^\omega$ and set
\begin{equation} b \coloneqq E_{\iota(M_{2^n})'}(a)\in \M^\omega\cap\iota(M_{n})' \end{equation}
where $E_{\iota(M_{n})'}$ is the canonical conditional expectation from
Definition~\ref{def:CondExp}.
Then $b$ is a positive contraction such that $\tau(b)=\tau(a)$ for all $\tau\in X^\omega$ by Remark~\ref{rem:cond_exp}\ref{rem:con_exp_1}.

The rest of this proof closely follows the ``tracial projectionisation'' argument of \cite[Lemma~2.4]{CETWW}.
Next, define functions $g_1,\dots,g_n \in C([0,1])$ by requiring that
$g_i|_{[0,\frac{i-1}{n}]} = 0$, $g_i|_{[\frac{i}{n},1]} = 1$, and $g_i$ is affine on $[\frac{i-1}{n},\frac{i}{n}]$. Note that 
\begin{equation} 
\label{eq:Projectionization2}
\mathrm{id}_{[0,1]} = \frac{1}{n} \sum_{i=1}^n g_i, \quad \text{and} 
\end{equation} 
\begin{equation}  
\label{eq:Projectionization3} 
0\leq \sum_{i=1}^n (g_i-g_i^2) \leq 1_{C([0,1])}.
\end{equation}
Then define
\begin{equation}\label{eq:def-p}
    p\coloneqq \sum_{i=1}^n \iota(e_{ii})g_i(b),
\end{equation}
which is a positive contraction by orthogonality of the $e_{ii}$.

For $\tau \in X^\omega$, since $M_n$ has unique trace, $\tau(g_i(b)\iota(\cdot))$ is a scalar multiple of this trace, and so
\begin{equation}\label{eq:Projectionization1}
    \tau(g_i(b)\iota(e_{ii}))=\frac{\tau(g_i(b))}{n}.
\end{equation}
Thus we have
\begin{equation}
    \tau(p) \stackrel{\eqref{eq:def-p}}{=} \tau \left(\sum_{i=1}^{n} \iota(e_{ii})g_{i}(b)\right) \stackrel{\eqref{eq:Projectionization1}}= \sum_{i=1}^{n} \frac{1}{n}\tau (g_{i}(b))
\stackrel{\eqref{eq:Projectionization2}}= \tau(b),\quad \tau\in X^\omega.
\end{equation}
Moreover, for $\tau \in X^\omega$, using that the $\iota(e_{ii})$ are orthogonal and commute with $b$,
\begin{equation} \begin{array}{rcl}
\tau(p-p^2)
&=& \tau\left(\sum_{i=1}^n \iota(e_{ii})\left(g_i(b)-g_i(b)^2\right)\right) \\
&\stackrel{\eqref{eq:Projectionization1}}=& \sum_{i=1}^n \frac1N\tau\left(g_i(b)-g_i(b)^2\right) \stackrel{\eqref{eq:Projectionization3}}\leq \tfrac1n \tau(1_{\M^\omega}) = \tfrac1n.
\end{array}
\end{equation}
Since $p-p^2$ is a positive contraction, we get
\begin{equation} \|p-p^2\|_{2,X^\omega}^2=\sup_{\tau\in X^\omega}\tau((p-p^2)^2) \leq \sup_{\tau \in X^\omega} \tau(p-p^2) \leq \tfrac1n, 
\end{equation} as required.
\end{proof} 

 Now we start on Proposition~\ref{S4.2:Prop}\ref{it:rr0-omega}.  It is a standard fact that in a separable $C^*$-algebra $A$ of real rank zero, for every $a \in A_+$, the hereditary subalgebra $\overline{aAa}$ has an increasing sequential approximate unit of projections; see \cite[Proposition 2.9]{BrownPedersen:RR0}.
 As a step towards showing that $\M^\omega$ has real rank zero, we prove a uniform 2-norm version of this for $\M^\omega$ using classification of projections.
 To do so, we will need to move larger trace projections over smaller ones \emph{while staying in this hereditary subalgebra}.
 The point of the functions $f_n$ in the following lemma is to give hereditary subalgebras where we can use comparison to move a larger projection over a smaller projection, while retaining overall control of its location.

 \begin{lemma}\label{Lemma:RealRank0IncreasingProject}
 Let $(\M,X)$ be a II$_1$ factorial tracially complete $C^*$-algebra such that $\mathcal M^\omega$ has strict comparison with respect to $X^\omega$ and $\M^\omega$ has a divisible unit. Suppose $a\in \mathcal M_+$ such that $d_\tau(a)>0$ for all $\tau\in X^\omega$. Then, for all $\epsilon>0$, there exists a projection $p\in \overline{a\mathcal M^\omega a}$ with $\|ap-a\|_{2,X^\omega}<\epsilon$.
 \end{lemma}

\begin{proof} 
 Without loss of generality, we can assume that $a$ is contractive. Since $\tau \mapsto d_{\tau} (a)$ is a lower-semicontinuous function on the compact space $X^\omega$, and the function is strictly positive, there exists some $\delta > 0$ such that $d_\tau(a) > \delta$ for every $\tau\in X^\omega$. Again by compactness of $X^\omega$ and lower semicontinuity of $d_{\tau}$, there exists $\epsilon > 0$ such that $d_\tau((a-\epsilon)_{+}) > 0$ for every $\tau\in X^\omega$.

 Our objective is to obtain an increasing sequence of projections $q_1\leq q_2\leq\dots$ in $\overline{a\M^\omega a}$ with $\tau(q_n)\to d_\tau(a)$ for all $\tau\in X^\omega$.  Suppose this has been achieved. Then for $\tau\in X^\omega$, we have that $\pi_\tau(q_n)$ is an increasing sequence of projections in $\pi_\tau(\M^\omega)''$ which converge strongly to $\chi_{(0,\infty)}(\pi_\tau(a))$ and so, for all $n \in \N$,
 \begin{equation}  
 \tau(q_na^2)=\tau(aq_na)\leq\tau(aq_{n+1}a)=\tau(q_{n+1}a^2)
 \end{equation}
 and $\lim_{n\to\infty}\tau(q_na^2)=\tau(a^2)$. By Dini's theorem this convergence is uniform, so
 \begin{equation}
     \|q_na-a\|_{2,X^\omega}^2=\sup_{\tau\in X^\omega}\tau(a^2-q_na^2)\to 0.
 \end{equation}
 Taking $n$ large enough proves the result.

 We now construct the desired sequence.  For each $n\in \mathbb{N}$, define $f_{n}\in C([0,1])$ to be the continuous piecewise affine function, 
 \begin{equation} 
 f_{n}(t) \coloneqq 
 \begin{cases} 
 0, & t\in \left [0, \frac{\eps}{n+1}\right]; \\ 
 \text{affine}, & t\in \left [\frac{\eps}{n+1}, \frac{\eps}{n}\right]; \\ 
 1, & t\in \left[\frac{\eps}{n}, 1\right]. 
 \end{cases} 
 \end{equation} 
 Note that $f_{n}$ is an increasing sequence of positive functions converging on $[0,1]$ to $\chi_{(0,1]}$ pointwise, which implies that 
 \begin{equation} 
 \tau (f_{n}(a)) \rightarrow d_{\tau} (a), \quad \tau \in X^\omega. 
 \end{equation} 
 Moreover, for all $n \in \N$, we have $f_n(a)(a-\eps)_+ = (a-\eps)_+$ and
 \begin{equation} 
 \label{act_as_identity} 
 f_{n+1}(a)f_{n}(a) = f_{n}(a).
 \end{equation} 
 For each $n\in \mathbb{N}$, by the existence part of Proposition \ref{S4.2:Prop}(i), we can find a projection $r_{n}\in \M^{\omega}$ such that 
 \begin{equation} 
 \label{exist_proj_converg} 
 \tau(r_{n}) = \left(1-\frac{1}{n}\right) \tau (f_{n} (a))< d_{\tau}(f_{n}(a)), \quad \tau \in  X^\omega.  
 \end{equation} 
 The strict inequality above holds because $d_{\tau}(f_{n}(a)) \geq d_{\tau}( (a-\epsilon)_{+} )>0$ for any $\tau \in X^\omega$ and $n\in \mathbb{N}$ by the choice of $\epsilon$. Since $\M^{\omega}$ has strict comparison with respect to $X^\omega$, we have $r_{n}\precsim f_{n}(a)$ in $\mathcal M^{\omega}$. By R{\o}rdam's lemma (\cite[Theorem 2.7 (iv)]{GP_Cuntz_semi}), there exists a projection $p_{n} \in \overline{f_n(a) \M^\omega f_n(a)}$ such that $p_{n}\sim r_{n}$ in $\M^{\omega}$. In particular, we have 
 \begin{equation}\label{eq:def-pn}
 \tau (p_{n}) = \tau (r_{n}) = \left(1-\frac{1}{n}\right) \tau (f_{n} (a)), \quad \tau \in  X^\omega. 
 \end{equation} 
 Thus we obtain a sequence of projections $(p_{n})_{n}$, where $p_{n}\in \overline{f_n(a) \M^\omega f_n(a)}$, and $\tau (p_{n})\rightarrow d_{\tau}(a)$ for every $\tau \in X^\omega$. We next modify these to produce the required increasing sequence.

 Take $q_1 \coloneqq p_1$. Suppose $q_1 \leq q_2 \leq \cdots \leq q_n$ have been constructed with $p_{i} \sim q_{i}$ and $q_{i} \in \overline{f_i(a) \M^\omega f_i(a)}$ for $i\in \{1, \cdots, n\}$. Since $f_{n}(a) \leq f_{n+1}(a)$, \eqref{eq:def-pn} implies
 \begin{equation} 
 \tau (q_{n}) < \tau (p_{n+1}), \quad \tau \in  X^\omega. 
 \end{equation} 
 By strict comparison of $\M^{\omega}$ with respect to $X^\omega$, there exists $v\in \M^{\omega}$ such that $v^*v = q_n$ and $vv^* \leq p_{n+1}$. Note that $v\in \overline{f_{n+1}(a) \M^\omega f_{n+1}(a)}$ as $q_{n}$ and $p_{n+1}$ are both in $\overline{f_{n+1}(a) \M^\omega f_{n+1}(a)}$. It remains to move $p_{n+1} - vv^{*}$ beneath $f_{n+1}(a) - q_{n}$. Since $q_n \in \overline{f_{n}(a) \M^\omega f_{n}(a)}$ and $f_{n+1}(a)f_n(a) = f_n(a)$ (equation \eqref{act_as_identity}), we have $f_{n+1}(a) q_{n} = q_{n}$. Thus $f_{n+1}(a) - q_{n}$ is a positive element and for all $k\in \mathbb{N}$, 
 \begin{equation} 
 \label{act_identity_power} 
 (f_{n+1}(a) - q_n)^k = f_{n+1}(a)^k - q_n.
 \end{equation} 
 Hence, we have the following for any $\tau \in X^\omega$: 
 \begin{equation} 
 \begin{array}{rcl} 
 d_\tau(f_{n+1}(a) - q_n) &\stackrel{\eqref{act_identity_power}}{=}& d_\tau(f_{n+1}(a)) - \tau(q_n)\\
 &\stackrel{\eqref{eq:def-pn}}{>}& \tau(p_{n+1}) - \tau(q_n) \\
 &=&\tau(p_{n+1} - vv^*).
 \end{array} 
 \end{equation}  
 Using strict comparison of $\M^{\omega}$ with respect to $X^\omega$, combined with R{\o}rdam's lemma (\cite[Theorem 2.7 (iv)]{GP_Cuntz_semi}), there exists some $w\in \M^{\omega}$ such that 
 \begin{equation} 
 w^*w \in \overline{(f_{n+1}(a) - q_n)  \mathcal M^\omega (f_{n+1}(a) - q_n)} \subseteq \overline{f_{n+1}(a) \mathcal M^\omega f_{n+1}(a)}, 
 \end{equation} 
 and $ww^* = p_{n+1} - vv^*$. Just as above, we have $w \in \overline{f_{n+1}(a) \M^\omega f_{n+1}(a)}$. Since $ww^* \perp vv^*$ and $w^*w \perp q_n = v^*v$, it follows that $u = v+w \in \overline{f_{n+1}(a) \M^\omega f_{n+1}(a)}$ is a partial isometry with $q_n = v^*v \leq u^*u$ and $uu^* = p_{n+1}$. Then we may set $q_{n+1} = u^*u$.  By induction we get the required increasing sequence $(q_n)_{n=1}^\infty$, which proves the lemma. 
 \end{proof} 

To get real rank zero, we would like to apply the previous lemma to the positive part of a self-adjoint $b\in\mathcal M^\omega$ and then reindex to get a projection $p\in\mathcal M_\omega$ with $pb=b_+$. 
Since the previous lemma requires the element to be positive on all traces, we add a small trace projection as in the proof of non-strict comparison of projections.  Recall  from Remark~\ref{2909_rem_amplification} that $X^\omega_2$ consists of the normalised extensions of traces in $X^\omega$ to $M_2(\M^\omega)$, so that $\|x\|_{2,X^\omega}=\sqrt{2}\|(x\oplus 0)\|_{2,X_2^\omega}$ for $x\in\M^\omega$.
 
\begin{proof}[Proof of Proposition \ref{S4.2:Prop}\ref{it:rr0-omega}: real rank zero]
     Let $b \in \M^\omega$ be self-adjoint. Now write $b=b_+-b_-$ as the difference of its positive and negative parts  $b_{+},b_{-}\in\mathcal M^\omega_+$.  Fix $0<\epsilon<1$. We will show there is an positive contraction $p_{11}\in \mathcal M^\omega$ with
     \begin{equation}
         p_{11}b_-=0,\ \|p_{11}b_+-b_+\|_{2,X^\omega}<\epsilon\text{ and }\|p_{11}^2-p_{11}\|_{2,X^\omega}<\epsilon.
     \end{equation}
Once this is achieved, Kirchberg's epsilon test provides a projection $r\in \M^\omega$ with $rb_+=b_+$ and $rb_-=0$. This gives real rank zero as for any $\delta>0$, the self-adjoint element 
\begin{equation}
   h\coloneqq r(b+\delta 1_{\mathcal M^\omega})+(1-r)(b-\delta 1_{\mathcal M^\omega})
\end{equation}
is invertible with $\|h-b\|=\delta$.

To prove the claim, use divisibility of the unit to find a projection $q\in \mathcal M^\omega$ with constant trace $\tau(q)<\frac{\epsilon^2}{27}$ for all $\tau\in X^\omega$.  Applying Lemma \ref{Lemma:RealRank0IncreasingProject} to $a\coloneqq b_+\oplus q\in M_2(\mathcal M^\omega)_+$ (and noting that $M_2(\M^\omega)$ inherits strict comparison with respect to $X^\omega_2$ from $\M^\omega$), we obtain a projection 
\begin{equation}\label{e:sw:423}
p=\begin{pmatrix}p_{11}&p_{12}\\p_{21}&p_{22}\end{pmatrix}\in \overline{(b_+\oplus q)M_2(\mathcal M^\omega)(b_+\oplus q)}
\end{equation}
such that 
\begin{equation}
\|(b_+\oplus q)p- (b_+\oplus q) \|_{2,X_2^\omega}<\frac{\epsilon^2}2.
\end{equation}
Therefore, we have
\begin{equation}\begin{split}
    \|b_+p_{11}-b_+\|_{2,X^\omega}^2\leq 2\|(b_+\oplus q)p - (b_+\oplus q)\|_{2,X^\omega_2}^2< \epsilon^2,
\end{split}\end{equation}
where the factor of $2$ arises due to the rescaling of traces from $X^\omega$ to $X_2^\omega$.
It follows from \eqref{e:sw:423} that $p_{11}\in\overline{b_+\mathcal M^\omega b_+}$, $p_{12}q=p_{12}$, $qp_{21}=p_{21}$ and $p_{22}\in q\mathcal M^\omega q$. The first of these gives $p_{11}b_-=0$. The latter three give
\begin{equation}
    \|p_{12}\|_{2,X^\omega}^2,\|p_{21}\|^2_{2,X^\omega},\|p_{22}\|_{2,X^\omega}^2\leq \|q\|_{2,X^\omega}^2<\frac{\e^2}{27}.
\end{equation}
Hence,
\begin{equation}
 \|p-(p_{11}\oplus 0)\|_{2,X^\omega_2}^2<\frac{3}{2}\cdot\frac{\epsilon^2}{27}=\frac{\epsilon^2}{18}. 
\end{equation}
Therefore,
\begin{equation}
    \|(p_{11}\oplus 0)-(p_{11}\oplus 0)^2\|_{2,X^\omega_2}\leq 3\|p-(p_{11}\oplus 0)\|_{2,X^\omega_2}<\frac\epsilon{\sqrt2},
\end{equation}
i.e.\ $\|p_{11}-p_{11}^2\|_{2,X^\omega}<\epsilon$, as required.
\end{proof}

Next, we use Blackadar and Handelman's result that stable rank one follows from real rank zero whenever the Murray--von Neumann semigroup has cancellation.
\begin{proof}[Proof of Proposition \ref{S4.2:Prop}\ref{it:sr1-omega}: stable rank one]
    Note that the hypotheses of Proposition~\ref{S4.2:Prop} are invariant under matrix amplification, so Proposition \ref{S4.2:Prop}\ref{it:classif-proj-omega} and Proposition~\ref{Prop:MvN} immediately gives cancellation in the Murray--von Neumann semigroup $V(\M^\omega)$.\footnote{Indeed, suppose $n \in \N$ and $p, q, r \in M_n(\M)$ are projections such that $p\oplus r\sim q\oplus r$. Let $\tau \in X^\omega_n$, and write $\tau_2\in X^\omega_{2n}$ for the corresponding normalised trace on $M_{2n}(\M)$. Then
    \begin{equation}
        \frac{1}{2}(\tau(p)+\tau(r))=\tau_2(p\oplus r)=\tau_2(q\oplus r)=\frac{1}{2}(\tau(q)+\tau(r)),
    \end{equation}
and so $\tau(p)=\tau(q)$. By Proposition \ref{Prop:MvN}, we have $p\sim q$.} Since $\M^\omega$ has real rank zero by Proposition \ref{S4.2:Prop}\ref{it:rr0-omega}, it has stable rank one by Proposition~\ref{prop:RR0CancellationSR1}.
\end{proof}

We access pureness via the following proposition. As proof of this proposition runs exactly as the proof of \cite[Theorem~D(iv)]{ET:RR0}, we omit it.

\begin{proposition}\label{prop:ETpure}
    Let $(\mathcal{N},Y)$ be a factorial tracially complete $C^*$-algebra with real rank zero such that $V(\mathcal{N})\cong \Aff(Y)_{+}$ as ordered semigroups.
    Then $\mathcal{N}$ is pure.
\end{proposition}

We can now prove the final part of Proposition~\ref{S4.2:Prop}.

\begin{proof}[Proof of Proposition \ref{S4.2:Prop}\ref{it:pure-omega}: pureness]
    The hypotheses of Proposition \ref{S4.2:Prop} are invariant under matrix amplification. Hence, $M_n(\M^\omega)$ also has classification of projections for all $n\in\mathbb N$. Therefore, the Murray-von Neumann semigroup $V(\M^\omega)$ is isomorphic $\Aff(X^\omega)_+$ as ordered semigroups. Pureness of $\M^\omega$ now follows from Proposition~\ref{prop:ETpure}.
\end{proof}

We end this subsection with the following technical corollary of Proposition~\ref{S4.2:Prop}\ref{it:rr0-omega}.

\begin{corollary}\label{Cor:FakeSpectralProjection} 
Let $(\M,X)$ be a II$_1$ factorial tracially complete $C^*$-algebra such that $\mathcal M^\omega$ has strict comparison with respect to $X^\omega$ and $\M^\omega$ has a divisible unit. 
Let $F\subseteq \M^{\omega}$ be a finite-dimensional subalgebra. 
For $b\in\M^\omega_+ \cap F'$ with and $\delta>0$, there exists a projection $q\in\M^\omega \cap F' \cap \{b\}'$ such that $qb\leq \delta q$ and $(1-q)b\geq \delta (1-q)$. 
\end{corollary}
\begin{proof}
    Fix $\epsilon>0$. Since $\M^\omega$ has real rank zero by Proposition \ref{S4.2:Prop}(ii), $\M^{\omega}\cap F'$ also has real rank zero by Lemma~\ref{Lem:RelCom} (we may assume that $F$ contains the unit to apply this, since unitising $F$ does not change the commutant). 
    Then there exists a finite spectrum element $c\in \M^\omega \cap F'$ with $\|c-b\|< \epsilon$. Let $q \in \M^\omega \cap F'$ be the spectral projection of $c$ for $[0,\delta]$ so that $qc = cq$,  $qc\leq \delta q$ and $(1-q)c\geq \delta(1-q)$. Then $\|[q, b]\| < 2 \epsilon$, $qbq\leq(\delta+\epsilon)q$, and $(1-q)b(1-q)\geq (\delta+\epsilon)(1-q)$. The result follows from Kirchberg's epsilon test. 
\end{proof} 

\section{\safeK-stability} \label{sec:KStable} 

In this section, we prove our $K$-stability results: Theorem~\ref{Intro:KStable1} and Theorem~\ref{Intro:KStable2}. As explained in the introduction, using the trace--kernel extension, these will be consequences of the properties we established for tracial ultrapowers in Section~\ref{sec:tracial-ultrapowers}.
We first obtain $K$-stability for $A^{\omega}$.

\begin{proposition} 
\label{2609_Aloweromega_kstable} 
    Let $A$ be a unital $C^*$-algebra with $T(A)\neq \emptyset$. 
    Suppose $A$ has strict comparison with respect to $T(A)$ and $[1_A]\in \mathrm{Cu}(A)$ is almost divisible. 
    Then $A^\omega$ is $K$-stable. 
\end{proposition} 

\begin{proof} 
Let $(\M, X)$ be the tracial completion of $A$ with respect to $T(A)$. 
By Proposition \ref{prop:tracial-ultrapower-is-TC}, we have $(A^{\omega}, \overline{T_{\omega}(A)}) \cong (\M^{\omega}, X^{\omega})$.
Hence, by Proposition \ref{prop:sc_divisible_tracial_ultra}, our strict comparison and almost divisibility hypotheses on $A$ imply that $\M^{\omega}$ has strict comparison with respect to $X^{\omega}$ and $\M^{\omega}$ has a divisible unit. 
Therefore, $\M^{\omega}$ has stable rank one by Proposition~\ref{S4.2:Prop} and has divisibility in the sense of Rieffel (see \cite[Definition~4.1]{rieffel87}).  
Combining these results with \cite[Theorem 4.13]{rieffel87} gives $K$-stability of $\M^{\omega}$. Hence, $A^{\omega}$ is $K$-stable.  
\end{proof} 

We now turn to the trace-kernel ideal $J_A$. This is non-unital, and thus we need the following notion. For a (not necessarily unital) $C^*$-algebra $B$, we denote the minimal proper unitisation by $\tilde{B}$ and the \emph{generalised unitary group} is 
\begin{equation} 
\mathcal{V}(B) \coloneqq \{u\in U(\tilde{B}): u - 1_{\tilde{B}} \in B\}.\footnote{$\mathcal{V}(B)$ is homeomorphic to the topological group of quasi-unitaries in $B$ defined in \cite{thomsen}, and thus using it to define $K$-stability coincides with Thomsen's original \cite[Definition 3.1]{thomsen}. } 
\end{equation} 
If $B$ is unital, $\mathcal{V}(B)\cong U(B)$ as topological groups. 
For $n\geq 2$, we denote $\mathcal{V}(M_{n}\otimes B)$ by $\mathcal{V}_{n}(B)$. 
There are canonical inclusions $i_{n, n+1}: \mathcal{V}_{n}(B)\rightarrow \mathcal{V}_{n+1}(B)$. Then $B$ is \emph{$K$-stable} is for every $k\geq 0$ and $n\geq 1$, the induced map $\pi_{k}(i_{n, n+1})$ are isomorphisms.\footnote{In this section, $\pi_k$ is the functor taking pointed topological spaces to their $k$-th homotopy groups. Fortunately, we will not need to explicitly talk about the fibres of tracially complete $C^*$-algebras in this section!} Almost by definition, stable $C^*$-algebras are $K$-stable. Unfortunately $J_A$ is not stable, but it is separably-stable: every separable subalgebra $B_0$ is contained in a stable separable subalgebra $B_1$. This is enough for $K$-stability.

\begin{proposition}
\label{prop:J_A_K_stable}
Let $A$ be a unital $C^*$-algebra with $T(A) \neq \emptyset$. Suppose that $A$ has strict comparison with respect to $T(A)$. Then $J_{A}$ is $K$-stable. 
\end{proposition} 

\begin{proof} 
By the proof of \cite[Lemma 6.11]{CGSTW},\footnote{The argument in \cite{CGSTW} has additional assumptions: simplicity and $\Z$-stability, but in fact only relies on $A_{\omega}$ having strict comparison with respect to $\overline{T_\omega(A)}$ which follows from strict comparison of $A$ with respect to $T(A)$; see \cite[Lemma 1.23]{BBSTWW}. See the discussion of hypotheses in \cite[Section 1.3.6]{CGSTW}.} since $A$ has strict comparison with respect to $T(A)$, the trace-kernel ideal $J_A$ is separably stable. Fix $k \geq 0$ and $n\geq 1$. 
We show injectivity of $\pi_{k}(i_{n, n+1})$ and surjectivity follows from a similar strategy. 

To prove injectivity of $\pi_{k}(i_{n, n+1})$, take a continuous function $f: S^{k}\rightarrow \mathcal{V}_{n}(J_{A})$ such that $i_{n, n+1} \circ f$ is homotopic to the constant function $1_{n+1}$. 
Let $F: [0,1]\times S^{k}\rightarrow \mathcal{V}_{n+1}(J_{A})$ be the homotopy. 
Then matrix entries of $f(x) - 1_{n}$ and $F(y) - 1_{n+1}$ for $x\in S^{k}$ and $y\in [0,1]\times S^{k}$ generate a separable $C^*$-subalgebra as $S^{k}$ is compact and metrisable. This is contained in a separable stable subalgebra $B\subseteq J_{B}$ by separable stability. As $B$ is stable, it is $K$-stable. By the definition of $B$, we can view $f$ as a map into $\mathcal{V}_{n}(B)$ such that $i_{n, n+1}\circ f$ is homotopic to $1_{n+1}$ in $\mathcal{V}_{n+1}(B)$. The $K$-stability of $B$ implies that $f$ is homotopic to $1_{n}$ in $\mathcal{V}_{n}(B)$. Hence, $f$ is homotopic to $1_{n}$ in $\mathcal{V}_{n}(J_A)$). 
\end{proof} 

The final ingredient needed is the following technical result which allows us to pass from $K$-stability of $A_\omega$ to $K$-stability of $A$. 

\begin{lemma}\label{2609_AloweromegatoA} 
    Let $A$ be a unital $C^*$-algebra and $X$ be a compact Hausdorff space. 
    If $C(X,A_\omega)$ is $K_1$-injective (respectively, $K_1$-surjective), then  $C(X,A)$ is $K_1$-injective (respectively,  $K_1$-surjective). In particular, if $A_\omega$ is $K$-stable, then $A$ is $K$-stable.
\end{lemma} 

\begin{proof} 
    For any $C^*$-algebra $B$, we write $\iota_B:B \rightarrow B_\omega$ for the canonical inclusion of $B$ coming from constant sequences.
    Consider the canonical inclusions 
    \begin{equation}
       C(X,A) \xrightarrow{\phi} C(X,A_\omega) \xrightarrow{\psi} C(X,A)_\omega,
    \end{equation}
    where $\phi = \mathrm{id}_{C(X)}\otimes \iota_{A}$, 
    and $\psi$ determined on elementary tensors by 
    $\psi(f \otimes (a_n)_{n=1}^{\infty}) = (f \otimes a_n)_{n=1}^{\infty}$ for every $f \in C(X)$ and $(a_n)_{n=1}^{\infty} \in A_\omega$.\footnote{We are identifying $C(X,A_\omega) = C(X) \otimes A_\omega = C(X) \otimes_{\mathrm{max}} A_\omega$. The map $\psi$ is well-defined on the maximal tensor product as $C(X)$ and $A_\omega$ are commuting subalgebras of $C(X,A)_\omega$.}
    Then $\psi \circ \phi = \iota_{C(X,A)}$. 
    
    Hereinafter, set $D \coloneqq C(X,A)$ for brevity.
     We first observe that if $u_{n}$ and $v_{n}$ are unitaries in $U_{d}(D)$ for some $d$, such that $u = (u_{n})_{n=1}^{\infty}$ and $v = (v_{n})_{n=1}^{\infty}$ are homotopic in $U_{d}(D_{\omega})$, then 
    \begin{equation} 
    \label{homotopy_coordinatewise} 
    \{n\in \mathbb{N}: u_{n}\sim_{h} v_{n}\text{ in }U_{d}(D) \}\in \omega. 
    \end{equation}
    Indeed, $uv^{*}$ is homotopic to $1$ and thus there are self-adjoints $h_1,\dots, h_k \in M_{d}(D_{\omega})$ such that $uv^{*} = e^{ih_1} \cdots e^{ih_k}$. Take representative sequences of self-adjoints $(h_{i}(n))_{n=1}^{\infty}$ of $h_{i}$ for each $i\in \{1, \dots, k\}$. Then 
    \begin{equation} 
    (u_{n}v_{n}^{*})_{n=1}^{\infty} = (e^{ih_{1}(n)}\cdots e^{ih_{k}(n)})_{n=1}^{\infty} 
    \end{equation} 
    in $U_d (D_{\omega})$, which implies \eqref{homotopy_coordinatewise} as unitaries that are close in norm are homotopic.  

    Suppose $C(X, A_{\omega})$ is $K_{1}$-injective. Take a unitary $u\in D$ with $[u]_{1} = 0$. Then $[\phi(u)]_1 = 0 \in K_1(C(X,A_\omega))$, which implies that $\phi(u) \sim_{h} 1$ in $U(C(X,A_\omega))$ by its $K_{1}$-injectivity. Applying $\psi$ gives $\iota_{D}(u) \sim_{h} 1$ in $U(D_{\omega})$ and by the observation above, we conclude that $u \sim_{h} 1$ in the unitary group of $D$. 

    Suppose that $C(X, A_{\omega})$ is $K_{1}$-surjective. Take some $x \in K_1(D)$ and $u\in U_{d}(D)$ for some $d$ such that $[u]_{1} = x \in K_{1}(D)$. By $K_{1}$-surjectivity of $C(X, A_{\omega})$, there exists some unitary $v\in C(X, A_{\omega})$ such that $K_{1}(\phi)(x) = [v]_{1}\in K_{1}(C(X, A_{\omega}))$. Then $[\iota_{D}(u)]_{1} = K_{1}(\psi\circ \phi)(x)  = [\psi(v)]_{1} \in K_{1}(D_{\omega})$. There exists $m\geq d$ such that 
    \begin{equation} 
    \iota_{D}(u) \oplus 1^{\oplus (m-d)} \sim_{h} \psi (v)\oplus 1^{\oplus (m-1)}\quad \text{in } U_{m}(D_{\omega}). 
    \end{equation} 
    Choosing a unitary lift $(v_{n})_{n=1}^{\infty} \in \ell^{\infty}(D)$ of $\psi(v)$. 
    By the observation above, for some $n_{0}\in \mathbb{N}$, $u \oplus 1^{\oplus (m-d)} \sim_{h} v_{n_{0}}\oplus 1^{\oplus (m-1)}$ in $U_{m}(D)$, and thus $[v_{n_{0}}]_{1} = x \in K_{1}(D)$. 

    For the second part, if $A_{\omega}$ is $K$-stable, then $C(S^{n}, M_{m}(A_{\omega})) \cong C(S^{n}, M_{m}(A)_{\omega})$ is $K_{1}$-bijective for all $n \geq 0$ and $m\geq 1$. Hence by the first part of the proof,  $C(S^{n}, M_{m}(A))$ is $K_{1}$-bijective for all $n \geq 0$ and $m\geq 1$. Therefore, $A$ is $K$-stable. 
\end{proof} 

We conclude this section by putting together the previous results to complete the proof of Theorem~\ref{Intro:KStable2} from the introduction, which implies Theorem~\ref{Intro:KStable1} as a special case. 

\begin{theorem}[Theorem~\ref{Intro:KStable2}]
Let $A$ be a unital $C^*$-algebra which has strict comparison with respect to $T(A)$ and an almost divisible unit. Then $A$ is $K$-stable.
\end{theorem}

\begin{proof}
By Proposition~\ref{2609_Aloweromega_kstable}, $A^{\omega}$ is $K$-stable. By Proposition \ref{prop:J_A_K_stable}, $J_{A}$ is $K$-stable. 
Applying permanence of $K$-stability for extensions (\cite[Theorem~3.11]{thomsen}) to the trace-kernel extension \eqref{2609_deftrkerEx}, it follows that $A_\omega$ is $K$-stable. 
Lemma~\ref{2609_AloweromegatoA} now implies that $A$ is $K$-stable. 
\end{proof}

\begin{corollary}[Theorem~\ref{Intro:KStable1}]
Let $A$ be a unital simple pure $C^*$-algebra whose quasitraces are traces. Then $A$ is $K$-stable.
\end{corollary}

\section{Correcting \safeK-theory in (small) corners and classifying projections}\label{Sec:CentralDivisible}

This section is devoted to showing $(\mathcal M, X)$ has classification of projections for the tracially complete $C^*$-algebras of interest in this paper, fulfilling Steps \ref{Step.5} and \ref{Step.6} of the outline in Section~\ref{subsec:organisation}.
In particular, we prove the following result, generalising \cite[Theorems~7.17 and~7.18]{carrion2023tracially}, where the result was obtained for II$_1$ factorial tracially complete $C^*$-algebras with CPoU.
As a consequence, for the tracially complete $C^*$-algebras $(\M,X)$ of interest, we also show a lifting result for certain $^*$-homomorphisms from finite-dimensional $C^*$-algebras to $\M^\omega$, which is Step~\ref{Step.7}.

\begin{theorem}\label{thm:classif-proj}
   Let $(\M,X)$ be a II$_1$ factorial tracially complete $C^*$-algebra such that $\M^\omega$ has strict comparison with respect to $X^\omega$ and $\M^\omega$ has divisible unit. Then $(\mathcal M, X)$ has classification of projections.
\end{theorem}

Similar results were obtained for certain classes of $W^*$-bundles (including all trivial $W^*$-bundles) in by Jekel, Ryder, and SW in the forthcoming work \cite{JRW:inPrep} (which actually predates the results here).  We note in particular that the results in \cite{JRW:inPrep} work from selflessness  and go much more directly, without the need for an intertwining argument. 

Note that we've already shown that under the hypotheses of Theorem~\ref{thm:classif-proj}, $(\mathcal M^\omega, X^\omega)$ has classification of projections in Proposition~\ref{S4.2:Prop}\ref{it:classif-proj-omega}.  It is a standard theme in classification that existence and uniqueness of maps into ultrapowers, or sequence algebras gives rise to existence results into the underlying algebra by means of intertwining arguments.  The extra difficulty here is that we do not know unitaries in $\mathcal M^\omega$ lift to sequences of unitaries in $\mathcal M$.

This is overcome using that $\mathcal M^\omega$ has stable rank one (Proposition~\ref{S4.2:Prop}\ref{it:sr1-omega}) together with various non-stable $K$-theory results.  For the existence of projections, we use Brown's Theorem~\ref{thm:sr1facts}\ref{it:BrownTrick}, which implies that unitarily equivalent projections in $\mathcal M^\omega$ are equivalent via a liftable unitary.

The existence result given here is essentially the same as that of \cite[Theorem~7.18]{carrion2023tracially}, which used the intertwining by reparameterisation argument in \cite[Theorem~5.11]{carrion2023tracially} (using sequence algebras in place of ultrapowers).  While we do not know if the hypothesis that all unitaries in $\mathcal M^\omega$ lift in that theorem is necessarily satisfied, the unitaries appearing in the proof can be chosen to be liftable.

Since the intertwining by reparameterisation argument is somewhat awkward in the case of ultrapowers, we give a bare-hands intertwining argument, which essentially amounts to specialising \cite[Theorem~5.11]{carrion2023tracially} to the case of a one-point space.  We start with the following classification of approximate projections in $\mathcal M$.

\begin{proposition}\label{prop:proj-approx-classif}
    Let $(\mathcal M, X)$ be a II$_1$ factorial tracially complete $C^*$-algebra such that $\mathcal M^\omega$ has strict comparison with respect to $X^\omega$ and $\mathcal M^\omega$ has divisible unit.
    \begin{enumerate}
        \item\label{it:proj-approx-unique} For all $\epsilon > 0$, there is $\delta > 0$ such that if $p, q \in \mathcal M$ are positive contractions with 
        \begin{equation}\label{eq:proj-approx-unique-input}
            \|p^2 - p \|_{2, X} < \delta, \quad \|q^2 - q \|_{2, X} < \delta, \quad \text{and} \quad \sup_{\tau \in X} |\tau(p) - \tau(q)| < \delta,
        \end{equation}
        then there is a unitary $u \in \mathcal M$ such that 
        \begin{equation}\label{eq:proj-approx-unique-output}
            \|u pu^* - q \|_{2, X} < \epsilon.
        \end{equation}
         \item\label{it:proj-approx-exist} For all $f \in \mathrm{Aff}(X)$ with $0 \leq f \leq 1$ and $\delta > 0$, there is a positive contraction $p \in \mathcal M$ such that 
        \begin{equation}\label{eq:proj-approx-exist}
            \|p^2 - p \|_{2, X} < \delta \quad \text{and} \quad \sup_{\tau \in X} |\tau(p) - f(\tau)| < \delta.
        \end{equation}
    \end{enumerate}
\end{proposition}

\begin{proof}
    Suppose \ref{it:proj-approx-unique} fails for some $\epsilon > 0$.  For each $n \in \mathbb N$, let $\delta_n = 1/n$ and choose positive contractions $p_n, q_n \in \mathcal M$ as in \eqref{eq:proj-approx-unique-input} such that there is no unitary $u_n \in \mathcal M$ as in \eqref{eq:proj-approx-unique-output}.     
    The sequences $(p_n)_{n=1}^\infty$ and $(q_n)_{n=1}^\infty$ in $\mathcal M$ 
    define projections $p, q \in \mathcal M^{\omega}$ with $\tau(p) = \tau(q)$ for all $\tau \in X^\omega$.  
    Hence by Proposition~\ref{S4.2:Prop}\ref{it:classif-proj-omega}, $p$ and $q$ are unitarily equivalent.  Since $\M^\omega$ has stable rank one by Proposition \ref{S4.2:Prop}\ref{it:rr0-omega}, by Theorem~\ref{thm:sr1facts}\ref{it:BrownTrick} (due to L.\ Brown), there is a unitary $u \in U(\mathcal M^\omega)$ in the path component of $1_{\mathcal M^\omega}$ such that $upu^* = q$.  Since $u$ is in the path component of $1_\mathcal{M^\omega}$, there is a sequence $(u_n)_{n=1}^\infty$ of unitaries in $\M$ lifting $u$ (see Proposition~\ref{Prop:LiftingPathComponent}).  Then
    \begin{equation}
        \lim_{n \rightarrow\omega} \|u_n p_n u_n^* - q_n \|_{2, X} = 0.
    \end{equation}
    This is a contradiction, proving \ref{it:proj-approx-unique}.

    For \ref{it:proj-approx-exist}, Proposition~\ref{S4.2:Prop}\ref{it:classif-proj-omega} implies there is a projection $p \in \mathcal M^\omega$ such that $\tau(p) = f(\tau|_\mathcal M)$ for all $\tau \in X^\omega$.  Let $(p_n)_{n=1}^\infty \subseteq \mathcal M$ be a sequence of positive contractions lifting $p$.  Then
    \begin{equation}
        \lim_{n \rightarrow \omega} \|p_n^2 - p_n \|_{2, X} = 0 \quad \text{and} \quad \lim_{n \rightarrow \omega} \sup_{\tau \in X} | \tau(p_n) - f(\tau) | = 0.
    \end{equation}
    For large enough $n \in \mathbb N$, $p_n \in \mathcal M$ satisfies \eqref{eq:proj-approx-exist}.
\end{proof}

The previous result and a standard intertwining argument gives the existence of projections.  The idea is simply to use Proposition~\ref{prop:proj-approx-classif}\ref{it:proj-approx-exist} to construct a sequence of approximate projections $(p_n)_{n=1}^\infty \subseteq \mathcal M$ approximately realising a given trace map and then to use Proposition~\ref{prop:proj-approx-classif}\ref{it:proj-approx-unique} to conjugate the $p_n$ into a $\|\,\cdot\,\|_{2, X}$-Cauchy sequence.

\begin{proof}[Proof of Theorem \ref{thm:classif-proj}: existence of projections]
    Fix $f \in \mathrm{Aff}(X)$ with $0 \leq f \leq 1$,  let $(\epsilon_n)_{n=1}^\infty \subseteq (0, \infty)$ be such that $\sum_{n=1}^\infty \epsilon_n < \infty$, and for each $n \in \mathbb N$, let $\delta_n \in (0,\delta_{n-1})$ be given by applying Proposition~\ref{prop:proj-approx-classif}\ref{it:proj-approx-unique} to $\epsilon_n$ and assume $\delta_n \rightarrow 0$.  Apply Proposition~\ref{prop:proj-approx-classif}\ref{it:proj-approx-exist} to find a positive contraction $p_n \in M$ such that 
    \begin{equation}
        \|p_n^2 - p_n\|_{2, X} < \delta_n/2 \quad \text{and} \quad \sup_{\tau \in X} | \tau(p_n) - f(\tau)| < \delta_n / 2.
    \end{equation}
    Then we have
    \begin{equation}
        \|p_n^2 - p_n\|_{2,X} < \delta_n,\ \|p_{n+1}^2 - p_{n+1}\|_{2,X} < \delta_n,\ \text{and}\ \sup_{\tau \in X} |\tau(p_n) - \tau(p_{n+1})| < \delta_n.
    \end{equation}
    By Proposition~\ref{prop:proj-approx-classif}\ref{it:proj-approx-unique} there are unitaries $v_n \in U(\mathcal M)$ such that $\|v_n p_n v_n^* - p_{n+1}\|_{2, X} < \epsilon_n$.  Define a sequence of projections by
    \begin{equation}
        q_n \coloneqq v_1^* \cdots v_{n-1}^* p_n v_{n-1} \cdots v_1 \in \mathcal M,
    \end{equation}
    so that all $n \in \mathbb N$,
    \begin{equation}
        \|q_n - q_{n+1} \|_{2, X} < \epsilon_n,\ \|q_n^2 - q_n\|_{2, X} < \delta_n/2,\ \text{and}\ \sup_{\tau \in X} |\tau(q_n) - \tau(f)| < \delta_n/2.
    \end{equation}
    Then the sequence $(q_n)_{n=1}^\infty$ is $\|\,\cdot\,\|_{2, X}$-Cauchy, and the limit $p \in \mathcal M$ is a projection with $\tau(p) = f(\tau)$ for all $\tau \in X$.
\end{proof}

Comparison of projections is more subtle.  
As in the proof of \cite[Theorem~7.17(i)]{carrion2023tracially} (where the result was proven assuming CPoU), it is enough to prove a uniqueness theorem for projections: two projections in $\mathcal M^\omega$ with the same trace are unitarily equivalent.  Proposition~\ref{prop:proj-approx-classif}\ref{it:proj-approx-unique} immediately gives uniqueness up to $\|\cdot\|_{2, X}$-approximate unitary equivalence.  

The uniqueness theorem in \cite[Theorem~7.17(ii)]{carrion2023tracially} (again depending on CPoU) used the classification of projections in $\mathcal M^\omega$ and a polar decomposition argument to show that projections $p, q \in \mathcal M$ which agree on trace and are $\|\cdot\|_{2, X}$-close are $\|\cdot\|_{2, X}$-approximately conjugate via a unitary $\|\cdot\|_{2, X}$-close to $1_\mathcal M^\omega$, and hence with a unitary in $\mathcal M$ using that unitaries lift (\cite[Lemma 7.16]{carrion2023tracially}).  This was then used to construct a $\|\cdot\|_{2, X}$-Cauchy sequence of unitaries realising the $\|\cdot\|_{2, X}$-approximate unitary equivalence between $p$ and $q$.

The overall strategy here is similar, but there is an extra difficulty.  Given projections $p, q \in \mathcal M$ which agree on traces, they are unitarily equivalent in $\mathcal M^\omega$ by Proposition~\ref{S4.2:Prop}\ref{it:classif-proj-omega}.  
Using $\mathcal M^\omega$ has stable rank one (\ref{S4.2:Prop}\ref{it:classif-proj-omega}), on the one hand, the unitary in $\mathcal M^\omega$ realising this equivalence can be chosen to lift by Brown's Theorem~\ref{thm:sr1facts}\ref{it:BrownTrick}, and on the other, using the existence of (approximate) unitary polar decompositions in $\mathcal M^\omega$ (Lemma~\ref{Polar}), 
when $\|p-q\|_{2,X}$ is small a unitary in $\M^\omega$ conjugating $p$ to $q$ can be taken $\|\cdot\|_{2, X}$-close to the identify using the proof of \cite[Lemma~7.16]{carrion2023tracially}. However, it is not clear if one can do both simultaneously.

To work around this, if $p$ and $q$ are projections in $\mathcal M$ which agree on traces and are $\|\cdot\|_{2, X}$-close, we will construct a unitary in $\mathcal M^\omega$ with $upu^* = q$ which is $\|\cdot\|_{2, X}$-close to the identity.  Adding a small projection $r$ to $p$ and $q$ in $\mathcal M_2(\mathcal M^\omega)$ will give us the room needed to conjugate to find a $\|\cdot\|_{2, X}$-perturbation of $u \oplus r$ conjugating $p \oplus 0$ to $q \oplus 0$ and which lifts to a sequence in $M_2(\mathcal M^\omega)$.  This is done using another application of stable rank one and Rieffel's $K_1$-bijectivity theorem (Theorem~\ref{thm:sr1facts}\ref{it:sr1Rieffel}) to find a unitary $w \in r \mathcal M^\omega$ with the same $K_1$-class as $u$ and then using that $u \oplus w^*$ vanishes in $K_1$.

Iterating this construction will lead to a unitary equivalence between $p \oplus 0$ and $q \oplus 0$ in $M_2(\mathcal M)$ and hence a Murray--von Neumann equivalence between $p$ and $q$.  Applying the same argument to $1_{\mathcal M^\omega}  - p$ and $1_{\mathcal M^\omega} - q$ and combining the partial isometries gives the desired unitary equivalence.

The following is the inductive step which, loosely speaking, obtains a version of \cite[Lemma~7.16]{carrion2023tracially} at the expense of a small amplification.  

\begin{lemma}\label{lem:proj-unique-induction}
    Let $(\M,X)$ be a II$_1$ factorial tracially complete $C^*$-algebra such that $\M^\omega$ has strict comparison with respect to $X^\omega$ and $\M^\omega$ has divisible unit.  Suppose $p, q, r_1, r_2 \in \mathcal M$ are projections such that $\tau(p) = \tau(q)$ for all $\tau \in X$, $p, q \leq r_1 \leq r_2$, and $r_1$ and $r_2$ have constant trace $t_{1}$ and $t_{2}$ with $t_{2} > t_{1} > 0$. 
    Then for every $\epsilon > 0$, there is a unitary $u \in r_2 \mathcal M r_2$ such that 
    \begin{equation}
        \|u p u^* - q\|_{2, X} < \epsilon \quad \text{and} \quad \| u - r_2 \|_{2, X} < 2 \|p - q\|_{2, X} + 2 \sqrt{ (t_2 - t_1)}  
        + \epsilon.
    \end{equation}
\end{lemma}

\begin{proof}
    By Proposition~\ref{S4.2:Prop}\ref{it:classif-proj-omega}, $(\mathcal M^\omega, X^\omega)$ has classification of projections, so $p \sim q$ and $r_1 - p \sim r_1 - q$ in $\mathcal M^\omega$.  It follows that $p$ and $q$ are unitarily equivalent in the corner $r_1 \mathcal M^\omega r_1$.  Let $v \in r_1 \mathcal M^\omega r_1$ be a unitary with $v p v^* = q$.  Define a contraction by
    \begin{equation}
        a \coloneqq pvp + (r_1 - p)v (r_1 - p) \in r_1 \mathcal M^\omega r_1 \cap \{p\}',
    \end{equation}
so that
\begin{equation}\label{E5.13}
    \begin{split}
        \|a - v\|_{2, X^\omega}^2 
        &= \| pvp - pv + (r_1 - p) v (r_1 - p) - (r_1 - p)v \|_{2, X^{\omega}}^2 \\
        &= \|pvp - pv\|_{2, X^\omega}^2 + \|(r_1 - p)v (r_1 - p) - (r_1-p)v\|_{2, X^\omega}^2 \\
        &= \| p(q - p)\|_{2, X}^2 + \|(r_1-p)(p -q)\|_{2, X}^2 \\
        &=\|r_1 (p - q) \|_{2, X}^2 \\
        &= \|p - q\|_{2, X}^2
    \end{split}
    \end{equation}

    Note that $r_1\mathcal M^\omega r_1 \cap \{p\}'$ is a direct sum of two corners of $\mathcal M^\omega$ (given by the projections $p$ and $r_1 - p$), so as noted in Section~\ref{SS:SR1}, it inherits stable rank one from $\mathcal M^\omega$ (Proposition~\ref{S4.2:Prop}\ref{it:sr1-omega}).
    Reindexing the approximate polar decompositions given by 
    Proposition \ref{it:sr1PolarDecomp} provides a unitary $w \in r_1 \mathcal M^\omega r_1 \cap \{p\}'$ such that $a=w|a|$.
    Let $\tilde{u}\coloneqq vw^*$ and note that $\tilde{u}$ is a unitary in $r_1 \mathcal M^\omega r_1$ with $\tilde{u}p\tilde{u}^* = q$.

    Calculating similarly to \eqref{E5.13}, we have
    \begin{equation}
    \begin{split}
        \|w - a \|_{2, X^\omega}^2 &= \|r_1 - |a|\|_{2, X^\omega}^2 \\
        &\leq \|r_1 - |a|^2\|_{2, X^\omega}^2 \\
        &\leq \|p - pv^*pvp \|_{2, X^\omega}^2 \\
        &\qquad + \|(r_1 - p) - (r_1 - p)v^*(r_1 - p) v (r_1 - p)\|_{2, X^\omega}^2
        \\
        &= \|pv^*(q - p)v p \|_{2, X^\omega}^2 \\
        &\qquad + \|(r_1 - p)v^*(p - q) v (r_1 - p)\|_{2, X^\omega}^2
        \\
        &\leq \|pv^*(p -q) \|_{2, X^\omega}^2 + \|(r_1 - p) v^*(p -q)\|_{2, X^\omega}^2 \\
        &= \|r_1 v^*(p - q)\|_{2, X^\omega}^2 \\
        &\leq \|p-q\|_{2,X}^2.
    \end{split}
    \end{equation}
    Combining this with \eqref{E5.13} gives
    \begin{equation}
        \|\tilde{u} - r_1 \|_{2, X^\omega} = \| v - w\|_{2, X^\omega} \leq 2 \|p - q\|_{2, X}^2.
    \end{equation}
    
    Note that $(r_2 - r_1) \mathcal M^\omega (r_2 - r_1)$ is a full corner of $\mathcal M^\omega$ (since $t_{2} - t_{1} > 0$ and by strict comparison), and so its inclusion into $\M^\omega$ induces an isomorphism on $K_1$.  
    Using $K_1$-surjectivity of $(r_2-r_1)\M^\omega(r_2-r_1)$ from Theorem~\ref{thm:sr1facts}\ref{it:sr1Rieffel} (again using that stable rank one passes to corners; see Section \ref{SS:SR1}), there is a unitary $w_2 \in (r_2 -r_1)\mathcal M^\omega (r_2 - r_1)$ such that
    \begin{equation} 
        [r_1 + w_2 + (1 - r_2)] = [\tilde{u} + (1 - r_1)] \text{ in }K_1(\mathcal M^\omega).
    \end{equation}
    Using that $r_2 \mathcal M r_2$ is a full corner of $\mathcal M^\omega$, it follows that $u\coloneqq \tilde{u} + w_2^* \in r_2\M^\omega r_2$ has trivial $K_1$-class.
    Hence by $K_1$-injectivity of $r_2 \mathcal M^\omega r_2$ (again from Theorem~\ref{thm:sr1facts}\ref{it:sr1Rieffel} using stable rank one on a corner), $u$ is in the path component of the unit in $U(r_2\mathcal M^\omega r_2)$, and so it lifts to a sequence $(u_n)_{n=1}^\infty$ of unitaries in $r_2\M r_2$. Note that $upu^* = q$ and
    \begin{equation}
    \begin{split}
        \|u - r_2 \|_{2, X^\omega} &\leq \|\tilde{u} - r_1\|_{2, X^\omega} +   \| w_2^* - (r_2-r_1)\|_{2, X^\omega} \\
        &\leq 2\|p - q\|_{2, X} + 2 \sqrt{(t_2 - t_1)}.  
    \end{split}
    \end{equation}
    Hence the required properties are satisfied by some $u_n$.
\end{proof}

Iterating Lemma~\ref{lem:proj-unique-induction}, similarly to the construction in \cite[Theorem~7.17(ii)]{carrion2023tracially}, will lead to uniqueness (and hence comparison) of projections in $(\mathcal M, X)$ as stated in Theorem~\ref{thm:classif-proj}.  In the proof that follows, note that the hypotheses on $(\M,X)$ -- that $\M^\omega$ has strict comparison with respect to $X^\omega$ and $
\M^\omega$ has divisible unit -- pass to constant trace corners $(p\M p,X|_{p\M p})$.   Indeed, divisibility in $(p \mathcal M p)^\omega = p \mathcal M^\omega p$ follows from classification of projections in $\M^\omega$ (Proposition~\ref{S4.2:Prop}\ref{it:classif-proj-omega}). 

\begin{proof}[Proof of Theorem \ref{thm:classif-proj}: comparison of projections]
     It suffices to show uniqueness of projections in the sense that if $p, q \in \mathcal M$ and $\tau(p) = \tau(q)$, then $p$ and $q$ are Murray and von Neumann equivalent.  Indeed, if this holds and $p, q \in \mathcal M$ are such that $\tau(p) \leq \tau(q)$ for all $\tau \in X$, the existence of projections implies there is a projection $r \in \mathcal M$ such that $\tau(p) + \tau(r) = \tau(q)$, and then uniqueness of projections applied to $M_2(\mathcal M)$ implies $p \oplus r \precsim q \oplus 0$.  Hence $p \precsim q$, proving comparison of projections.  Finally, when verifying uniqueness of projections, after replacing $\mathcal M$ with $M_2(\mathcal M)$ (which shares the same hypotheses as $\M$) and $p$ and $q$ with $p \oplus 0$ and $q \oplus 0$ in $M_2(\mathcal M)$, we may assume there is a projection $r$ with constant trace $t< 1$ such that $p, q \leq r$. 

     Now we prove uniqueness under these additional hypotheses.  So assume $p, q, r_0 \in \mathcal M$ are projections such that $\tau(p) = \tau(q)$ for all $\tau \in X$, $p, q \leq r_0$, and $r_0$ has constant trace $t_0 < 1$.  Let $(\epsilon_n)_{n=1}^\infty \subseteq (0, \infty)$ be decreasing such that $\sum_{n=1}^\infty \epsilon_n < \infty$.  We will inductively construct an increasing sequence of projections $(r_n)_{n=1}^\infty$ in $\M$, where each $r_n$ has constant trace $t_n < 1$ and a sequence of normal partial isometries $(u_n)_{n=1}^\infty$ in $\mathcal M$ such that each $u_n$ is a unitary in $r_{n}\M r_{n}$,  
    \begin{equation}
        \|u_n p u_n^* - q \|_{2, X} < \epsilon_n, \qquad  \text{and} \qquad  \|u_{n+1} - u_n\|_{2, X} < 3 \epsilon_n.
    \end{equation}

    Let $u_0 = r_{0}$ and suppose $r_n$ and $u_n$ have been constructed.  Choose $\delta_n > 0$ such that 
    \begin{equation} 
    \label{choose_delta_n}
        t_n + \delta_n < 1 \qquad  \text{and} \qquad 2 \|u_n p u_n^* - q\|_{2, X} +  3 \sqrt{\delta_n} < 2\epsilon_n.
    \end{equation}
    Since $r_n$ has constant trace $t_n < 1$, applying the existence of projections part of Theorem \ref{thm:classif-proj} in the corner $(1-r_n) \mathcal M (1 - r_n)$, we can find a projection $r_n'$ orthogonal to $r_n$ with constant trace $\delta_n$.  Set $r_{n+1} = r_n + r_n'$ and note that this is a projection in $\mathcal M$ with constant trace $t_{n+1} = t_n + \delta_n \in (0, 1)$ and $r_n \leq r_{n+1}$.

    Note that $u_n p u_n^*$ and $q$ are projections in $\mathcal M$ which agree on trace and satisfy $u_n p u_n^*, q \leq r_n$.  By Lemma~\ref{lem:proj-unique-induction}, there is a unitary $v_{n+1}$ in the corner $r_{n+1} \mathcal M r_{n+1}$ such that
    \begin{equation} 
        \|v_{n+1} u_n p u_n^* v_{n+1}^* - q \|_{2, X} < \epsilon_{n+1}
    \end{equation} 
    and 
    \begin{equation} 
    \label{close_identity_unitary} 
        \|v_{n+1} - r_{n+1} \|_{2, X}  < 2 \|u_n p u_n^* - q \|_{2, X} + 2 \sqrt{(t_{n+1} - t_n)} + \epsilon_{n+1}  
    \end{equation} 
    Set $u_{n+1} = v_{n+1} (u_n+ r_{n+1}-r_{n})$, which is a unitary in $r_{n+1} \M r_{n+1}$, and note that 
    \begin{equation}
    \begin{array}{rcl}
        \|u_{n+1} - u_n \|_{2, X} &=& \| (v_{n+1}-r_{n+1}) (u_{n} + r_{n+1} -r_{n}) + r_{n+1}-r_{n}\|_{2, X} \\
        & \leq & \|v_{n+1} - r_{n+1}\|_{2, X} +  \|r_{n+1} - r_{n}\|_{2, X}\\ 
        & \stackrel{\eqref{close_identity_unitary}}< & 2 \|u_n p u_n^* - q \|_{2, X} + 2 \sqrt{\delta_{n}} + \epsilon_{n+1} + \sqrt{\delta_{n}}\\
        & \stackrel{\eqref{choose_delta_n}}< & 3\epsilon_{n}. 
    \end{array} 
    \end{equation}
    completing the induction. 

    Since $\sum_{n=1}^\infty \epsilon_n < \infty$, the sequence $(u_n)_{n=1}^\infty$ is $\|\cdot\|_{2, X}$-Cauchy.  Let $u \in \mathcal M$ denote the limit and note that $u$.  By construction $upu^* = q$.  Since each $u_n$ is a normal partial isometry whose support contains $p$,  this also holds for $u$.  Define $w = up \in \mathcal M$ and note that $w^*w = p$ and $ww^* = upu^* = q$.  
\end{proof} 

As recorded in Proposition~\ref{Class:FD}, classification of projections in $(\mathcal M, X)$ implies classification of finite-dimensional subalgebras, giving part \ref{it:ExistenceToM} of the next result. Part \ref{it:LiftingToG} is a lifting result which will be crucial in Section~\ref{S:LocalHyperfinite}.  It combines the classification subalgebras with the same non-stable $K$-theory argument in Lemma~\ref{lem:proj-unique-induction} to perturbing a unitary to one in the path-component of the identity up to a small amplification.

\begin{proposition}\label{Prop:ClassifyToM}
Let $(\M,X)$ be a II$_1$ factorial tracially complete $C^*$-algebra such that $\M^\omega$ has strict comparison with respect to $X^\omega$ and $\M^\omega$ has divisible unit.
    \begin{enumerate}
        \item \label{it:ExistenceToM}
        $(\M,X)$ has classification of unital finite-dimensional subalgebras as in Proposition~\ref{Class:FD}.
        \item \label{it:LiftingToG}
        If $G$ is a finite-dimensional $C^*$-algebra, $\gamma \in [0, 1)$, and $\phi \colon G \rightarrow \mathcal M^\omega$ is a $^*$-homomorphism such that $\phi(1_G)$ has constant trace $\gamma$, then there is a sequence of $^*$-homomorphisms $(\phi_n \colon G \rightarrow \mathcal M)_{n=1}^\infty$ lifting $\phi$ such that $\phi_n(1_G)$ has constant trace $\gamma$ for all $n \in \mathbb N$. 
    \end{enumerate}
\end{proposition}

\begin{proof}
    \ref{it:ExistenceToM} follows immediately from Propositions \ref{Class:FD} and Theorem \ref{thm:classif-proj}.
    
    \ref{it:LiftingToG}   
    Let $p_1, \ldots, p_k \in Z(G)$ denote the minimal central projections in $G$.  For each $i = 1, \ldots, k$, lift $\phi(p_i)$ to a sequence of positive contractions $(a_{i, n})_{n=1}^\infty \subseteq \mathcal M$. Then since
    \begin{equation}\label{eq:sumain}
        \lim_{n \rightarrow \omega} \sup_{\tau \in X} \Big| \sum_{i=1}^k \tau(a_{i, n}) - \gamma \Big| = 0,
    \end{equation}
    by rescaling, we may assume that $\sum_{i=1}^k \tau(a_{i,n}) \leq \gamma$ for all $\tau\in X$ and $n\in\mathbb N$.
    Define $f_{i,n} \in \Aff X$ for $i=1,\dots,k$ and $n\in \mathbb{N}$ by 
    \begin{equation} \begin{split}
        f_{i,n}(\tau)\coloneqq \begin{cases} \tau(a_{i,n}),\quad &i<k; \\ \gamma-\sum_{j=1}^{k-1} \tau(a_{j,n}),\quad &i=k. \end{cases}
    \end{split}\end{equation} 
    Note that for any sequence $(\tau_n)_{n=1}^\infty$ in $X$, if $\tau\in X^\omega$ is the associated limit trace, then
    \begin{equation}\label{eq:LimitTracefin} \lim_{n\to\omega} f_{i,n}(\tau_n) = \tau(p_i), \quad i=1,\dots,k. 
    \end{equation}
    (For the case $i=k$, this uses \eqref{eq:sumain}). 

    By classification of projections in $\mathcal M$ (Theorem~\ref{thm:classif-proj}), Proposition~\ref{Class:FD} provides, for each each $n \in \mathbb N$, a $^*$-homomorphism $\psi_n \colon G \rightarrow \mathcal M$ such that\footnote{An affine function $\alpha\colon T(\M)\to T_{\leq 1}(G)$ is fully determined by the values $\alpha(\tau)(p_i)$ on the minimal central projections.}
    \begin{equation}
        \tau(\psi_n(p_i)) = f_{i,n}(\tau), \quad i = 1, \ldots, k,\ \tau \in X.
    \end{equation}
    In particular, $\tau(\psi_n(1_G)) = \gamma$ for all $\tau \in X$. 
    Define $\psi\colon G \to \M^\omega$ using the sequence $(\psi_n)$.
    Then using \eqref{eq:LimitTracefin} (and since limit traces are dense in $X^\omega$), it follows that $\tau(\psi(p_i))=\tau(\phi(p_i))$ for all $\tau\in X^\omega$ and $i=1,\dots,k$.
    Therefore, $\tau \circ \psi=\tau \circ \phi$ for all $\tau \in X^\omega$. 
    By Proposition~\ref{Class:FD}\ref{it:ClassFDUniqueness} (again using classification of projections, this time in $\M^\omega$ provided by Proposition~\ref{S4.2:Prop}\ref{it:classif-proj-omega}), $\psi$ and $\phi$ are unitarily equivalent, so there exists a unitary $v \in \M^\omega$ such that $\mathrm{Ad}(v)\circ\psi=\phi$.

    Note that the projection $q \coloneqq 1_{\mathcal M^\omega}  - \phi(1_k)$ satisfies $\tau(q)  = 1 - \gamma > 0$ for all $\tau \in X^\omega$, and since $\mathcal M^\omega$ has strict comparison with respect to $X^\omega$, $q$ is a full projection.  It follows that the inclusion $q \mathcal M^\omega q \rightarrow \mathcal M^\omega$ induces an isomorphism on $K_1$.  Since $\mathcal M^\omega$ has stable rank one by Proposition~\ref{S4.2:Prop}\ref{it:sr1-omega}, so does the corner $q \mathcal M^\omega q$ (as noted in Section~\ref{SS:SR1}), and so both $\M^{\omega}$ and the corner are $K_1$-bijective (Theorem~\ref{thm:sr1facts}\ref{it:sr1Rieffel}).
    By $K_1$-surjectivity, there is a unitary $w \in q\mathcal M^\omega q$ such that $[w + \phi(1_G)] = [v^*] \in K_1(\mathcal M^\omega)$.  In particular, $u = (w + \phi(1_G))v \in \mathcal M^\omega$ satisfies $\mathrm{Ad}(u) \circ \psi = \phi$ and $[u] = 0 \in K_1(\mathcal M^\omega)$.  By $K_1$-injectivity of $\M^{\omega}$, $u$ is in the path component of the identity in $\mathcal M^\omega$ and hence lifts to a sequence of unitaries $(u_n)_{n=1}^\infty \subseteq \mathcal M$ (by the standard lifting result recorded in Proposition \ref{Prop:LiftingPathComponent}).
    Setting $\phi_n\coloneqq \mathrm{Ad}(u_n)\circ\psi_n$ completes the proof. 
\end{proof} 

The next lemma will be the starting point to pass from local hyperfiniteness to global hyperfiniteness in Section~\ref{S:LocalHyperfinite}. In slightly more detail, fix an amenable II$_1$ factorial tracially complete $C^*$-algebra $(\mathcal M, X)$ such that $\mathcal M^\omega$ has strict comparison with respect to $X^\omega$ and $\mathcal M^\omega$ has divisible unit.  For a finite set $\mathcal F \subseteq \mathcal M$, amenability of $\mathcal M$ implies hyperfiniteness of $\pi_\tau(\mathcal M)''$ via Connes' theorem, and hence there is a finite-dimensional subalgebra $F \subseteq \pi_\tau(\mathcal M)''$ which $\|\cdot\|_{2, \tau}$-approximately contains $\mathcal F$.  The following will allow us to take $F$ to be a finite-dimensional subalgebra of $\mathcal M$, while still $\|\cdot\|_{2, \tau}$-approximately containing $\mathcal F$.

\begin{lemma}\label{Lem:LiftMatrix} 
Let $(\mathcal M,X)$ be a factorial tracially complete $C^*$-algebra such that $\M^\omega$ has comparison with respect to $X^\omega$ and $\mathcal M^\omega$ has divisible unit. 
 Let $F$ be a finite-dimensional $C^*$-algebra and $\tau \in X$. Given $\epsilon>0$ and unital $^*$-homomorphism $\tilde{\theta}\colon F\to \pi_\tau(\mathcal M)''$, there exists a unital $^*$-homomorphism $\theta\colon F\to \mathcal M$, such that for all $x\in F$ with $\|x\|\leq 1$ 
 \begin{equation} 
 \|\pi_\tau( \theta (x))- \tilde{\theta}(x)\|_{2, \tau} <\epsilon. 
 \end{equation}
 \end{lemma} 

 \begin{proof} 
 Write $\mathcal N \coloneqq \pi_\tau(\mathcal M)''$. Let $\gamma\colon T(\mathcal N) \rightarrow T(\mathcal{M})$ be the continuous affine function induced by $\pi_\tau$ (i.e.\ $\gamma(\rho)\coloneqq\rho\circ\pi_\tau$ for $\rho\in T(\mathcal N)$).  By \cite[Lemma 2.10]{carrion2023tracially},  any $\rho\in T(\mathcal N)$ has that $\rho\circ\pi_\tau$ lies in the closed face $Y$ of $X$ generated by $\tau$, i.e.\ 
 $\gamma(T(\mathcal N))$ is contained in the closed face $Y$ of $X$ generated by $\tau$. In fact it equals this face as we now argue. Note first that $\gamma(T(\mathcal N))$ contains $\tau$ and is convex and compact. 
 To see that it is a face, it suffices to check that if $\rho \in T(\mathcal N)$ and $\sigma \in X$ satisfies $\lambda \sigma \leq \gamma(\rho)$ for some $\lambda>0$ then $\sigma\in \gamma(T(\mathcal N))$.
 For this, note that $\sigma$ extends by $\|\cdot\|_{2,\rho}$-continuity to a trace $\hat\sigma$ on $\pi_{\gamma(\rho)}(\M)''$, and so $\hat\sigma \circ \pi_\rho \in T(\mathcal N)$ and $\sigma=\gamma(\hat\sigma\circ\pi_\rho)$. 
 This concludes the proof that $\gamma\colon T(\mathcal N) \to Y$ is an isomorphism.

 The unital $^*$-homomorphism $\tilde{\theta}: F\rightarrow \mathcal N$ induces a continuous affine function $\beta\colon T(\mathcal N)\to T(F)$.  By Lemma \ref{lem:aff-lift-valuesp} the composition 
 \begin{equation}
     Y\stackrel{\gamma^{-1}}{\longrightarrow}T(\mathcal N)\stackrel{\beta}{\longrightarrow}T(F)
 \end{equation}
 extends to a continuous affine map $\alpha\colon X\rightarrow T(F)$, and, as $(\mathcal M,X)$ has classification of unital finite-dimensional $C^*$-algebras by Proposition \ref{Prop:ClassifyToM}\ref{it:ExistenceToM}, $\alpha$ is realised by a unital $^*$-homomorphism $\phi: F\rightarrow \mathcal{M}$ such that 
 \begin{equation}\label{2509_EqEx}
     \rho \circ \phi = \alpha (\rho),\quad \rho \in T(\mathcal{M}).
 \end{equation}
 In particular, 
 \begin{equation} 
 \rho \circ \tilde{\theta} = \beta (\rho) = \alpha(\gamma(\rho)) = \alpha(\rho\circ\pi_\tau)\stackrel{\eqref{2509_EqEx}}=\rho\circ\pi_\tau\circ\phi,\quad \rho\in T(\mathcal N).
 \end{equation} 
 By the classification of maps from finite-dimensional $C^*$-algebras into finite von Neumann algebras by traces,\footnote{This is ancient history and due to Murray and von Neumann.} there is a unitary $u\in \mathcal N$ with 
 \begin{equation}
     \tilde{\theta}(x)=u\pi_\tau(\phi(x))u^*,\quad x\in F. \end{equation}
 By Kaplansky's density theorem, there is an unitary $v\in \mathcal M$ such that 
 \begin{equation}\|\pi_\tau(v)-u\|_{2,\tau}<\frac{1}{2}\epsilon.
 \end{equation} 
 Define $\theta\coloneqq \mathrm{Ad}(v)\circ\phi\colon F\to \mathcal M^\omega$. Then, for $x\in F$ with $\|x\| \leq 1$, we have
\begin{equation}
 \begin{split}
     \|\pi_\tau(\theta(x))-\tilde\theta(x)\|_{2,\tau} &=\|\mathrm{Ad}(\pi_\tau(v))(\phi(x))-\mathrm{Ad}(u)(\phi(x))\|_{2, \tau} \\      &\leq 2\|\pi_\tau(v)-u\|_{2,\tau}\|x\| \\
     &<\eps,
 \end{split}
 \end{equation}
 as required. 
 \end{proof} 

\section{Local hyperfiniteness} 
\label{S:LocalHyperfinite} 
In this section we use the structural results we have obtained in Sections~\ref{sec:tracial-ultrapowers} and~\ref{Sec:CentralDivisible} to establish local hyperfiniteness for amenable tracially complete $C^*$-algebras in the spirit of the  arguments of Connes (\cite{Connes:Annals}) and Popa (\cite{Popa:InjectiveHyperfinite}).  We isolate four small calculations in Section \ref{SS6.1}, before giving the main ``small local approximation'' argument. 

\subsection{Preliminary estimates}\label{SS6.1} We start with the following standard polar decomposition estimate. Many results of this type can be found in the literature (for example, \cite[Lemma 2.7]{Christensen75}). 

\begin{lemma}\label{Polar}
Let $B$ be a unital $C^*$-algebra and $p,q$ be projections in $B$ such that $qpq\leq \delta q$ for some $0<\delta<1$.  Then the partial isometry $v$ in the polar decomposition of $(1-p)q$ lies in $B$ and satisfies 
\begin{equation}  
\begin{split} 
v^*v&=q,\\ vv^*&\leq (1-p),\\
\|v-q\|&\leq\delta^{1/2}.
\end{split} 
\end{equation} 
\end{lemma}
\begin{proof}
By definition, $(1-p)q=v|(1-p)q|$ in $B^{**}$ and $v^*v \leq q$. By hypothesis 
\begin{equation}
q\geq a\coloneqq |(1-p)q|^2=q(1-p)q\geq (1-\delta)q,
\end{equation}
and therefore $a$ is invertible in $qBq$ and using $a^{-1}$ to denote this inverse, we have $q\leq a^{-1}\leq (1-\delta)^{-1}q$. Therefore $a^{-1/2}\in B$ and so $v=(1-p)qa^{-1/2}\in B$. 
Next,
\begin{equation}
v^*v =a^{-1/2}q(1-p)qa^{-1/2}=q,
\end{equation} 
while $vv^*=(1-p)qa^{-1}q(1-p)$, and so $(1-p)vv^*=vv^*$.

For the last inequality, we see that\footnote{The very last inequality is justified by
\begin{equation}
    0 \leq (1-(1-\delta)^{1/2})(1-\delta^{1/2}) = 1-(1-\delta)^{1/2}- \delta^{1/2}+ (\delta(1-\delta))^{1/2}, 
\end{equation} 
and hence
$(1-\delta)^{1/2}-1+\delta^{1/2} \leq (\delta(1-\delta))^{1/2}$.
The desired inequality follows by dividing by $(1-\delta)^{1/2}$. Though we found it easier to draw a graph. }  
\begin{equation}\begin{split} 
\|v-q\|&=\|(1-p)qa^{-1/2}-q\|\\
&\leq \|qa^{-1/2}-q\|+\|pqa^{-1/2}\|\\
&\leq 1-\frac{1}{(1-\delta)^{1/2}}+\frac{\delta^{1/2}}{(1-\delta)^{1/2}} \\
&\leq\delta^{1/2}.
\end{split}\end{equation} 
\vspace*{-2.5em}

\end{proof} 

The following routine estimate is isolated to slightly shorten the proof of Proposition \ref{SmallLocalApproximations}.
\begin{lemma} 
\label{2409_LemPosiSelf} 
    Let $B$ be a $C^*$-algebra, $b\in B_+$ a positive contraction and $h\in B$  self-adjoint. Then, for every $\tau\in T(B)$, 
    \begin{equation}\|[b,h]\|_{2,\tau}^2\leq 2\tau(b^{2} h^2).\end{equation}  
\end{lemma} 

\begin{proof} 
We have the following computation 
    \begin{equation} 
        \begin{split} 
            \|[b, h]\|_{2, \tau}^2 &= 
            \tau ((bh-hb)^{*} (bh-hb)) \\
            &= \tau (hb^{2}h - hbhb - bhbh + bh^{2}b) \\
            &= 2\tau (b^{2}h^{2}) - 2\tau(hbhb)\\ 
        \end{split}
    \end{equation} 
Since $\tau(hbhb) = \|b^{1/2}h b^{1/2}\|_{2, \tau}^{2}\geq 0$, we have $\|[b, h]\|_{2, \tau}^{2}\leq 2\tau (b^{2}h^{2})$. 
\end{proof} 

We also need the following Pythagorean estimate (similar to that at the end of \cite[Page 271]{Popa:InjectiveHyperfinite}, and which is best seen by drawing a picture of matrix blocks). 

\begin{lemma}\label{Pythagoras} 
    Let $r_1, \ldots, r_n$ be pairwise orthogonal projections in a unital $C^*$-algebra $A$, and suppose $x \in A$ and $y_i \in r_i Ar_i$ for $i = 1, \ldots, n$.  Write $r = \sum_{i=1}^n r_i$ and $y = \sum_{i=1}^n y_i$.  Then for all $\tau \in T(A)$,
    \begin{equation}
        \|x - (1 - r) x (1 - r) - y\|_{2, \tau}^2 \leq \sum_{i=1}^n \| x - (1 - r_i) x (1 - r_i) - y_i\|_{2, \tau}^2
    \end{equation}
\end{lemma} 
\begin{proof}
    On expanding and using Pythagoras with respect to the matrix block decomposition $r_1,\dots,r_n,1-r$ the right-hand side exceeds the left-hand side by
    \begin{equation}
        2\sum_{i\neq j}\|r_ixr_j\|_{2,\tau}^2\geq 0.\qedhere
    \end{equation}
\end{proof}
 
Finally, the following lemma allows us to realise prescribed trace values by subprojections in the relative commutant of a unital finite-dimensional subalgebra. 

\begin{lemma} 
\label{lem:TraceProjInCommutant} 
    Let $(\M,X)$ be a II$_{1}$ factorial tracially complete $C^*$-algebra with classification of projections and $X$ is a Choquet simplex,  and $F\subseteq \M$ be a finite-dimensional unital subalgebra.
    If $q \in \M \cap F'$ is a projection and $h \in \Aff(X)_+$ satisfies $h(\tau) \leq \tau(q)$ for all $\tau \in X$, then there exists a projection $z\in \M \cap F'$ such that
    \begin{equation} z \leq q\quad\text{and}\quad \tau(z)=h(\tau),\quad \tau \in X. \end{equation} 
\end{lemma} 

\begin{proof} 
Let $p_{1}, \dots, p_{k}$ be the minimal central projections in $F$ and note that 
\begin{equation}\label{eq:new-24}
	q = \sum_{i=1}^k q p_{i} \in \mathcal M \cap F', 
\end{equation} 
as $F$ is a unital subalgebra of $\mathcal M$. 
Since $X$ is a Choquet simplex, Corollary~\ref{cor:riesz-properties} gives that $\Aff(X)$ has Riesz decomposition, so there are continuous affine functions $h_1, \ldots, h_k \in \Aff(X)_+$ such that $h = \sum_{i=1}^k h_i$ and 
\begin{equation}\label{eq:new-25}
	0 \leq h_i(\tau) \leq \tau(q p_{i}), \quad i = 1, \ldots, k,\ \tau \in X. 
\end{equation}
	Identifying $p_{i} \mathcal M p_i \cap (F p_i)'$ with a corner of $\mathcal M$, then it also has classification of projections as this property passes to corners. 
    In particular, for each $i = 1, \dots, k$, 
    there is a projection $z_i \in p_{i} \mathcal M p_i \cap (F p_i)'$ such that $z_{i} \leq q p_i$ and
	\begin{equation}\label{eq:new-26}
		\tau(z_i) = h_i(\tau), \quad \tau \in X. 
	\end{equation}
	Set $z \coloneqq \sum_{i = 1}^k z_i \in \mathcal M \cap F'$. Note that $z$ is a projection with $z \leq q$ and $\tau(z) = h(\tau)$ for all $\tau \in X$.  
\end{proof}

\subsection{Small local approximations} 
We now come to the crucial technical result, which provides small finite-dimensional approximations allowing us to run the maximality arguments discussed in Section \ref{S:SmallMatricesOutline}. 

We continue to work with the standing hypotheses that $\M^\omega$ has strict comparison with respect to $X^\omega$ and divisible unit. For ease of comparison with the discussion in Section \ref{S:SmallMatricesOutline}, the consequences of this which will be used in the proof are exactly the following by the order of appearance.
\begin{itemize}
    \item Every constant trace $^*$-homomorphism from a finite-dimensional $C^*$-algebra $G$ to $\M^{\omega}$ lifts to a sequence of $^*$-homomorphisms from $G$ to $\M$ with the same constant trace (provided by Proposition~\ref{Prop:ClassifyToM}\ref{it:LiftingToG}). This reduces the proof to finding finite-dimensional approximations in $\M^{\omega}$.
    \item $(\mathcal{M},X)$ has classification of unital finite-dimensional $C^*$-algebras (provided by Proposition \ref{Prop:ClassifyToM}\ref{it:ExistenceToM}), which leads to global finite-dimensional subalgebras in $\M$ for fibrewise or local approximations (as shown in Lemma~\ref{Lem:LiftMatrix}).
    \item Through real rank zero of $\M^{\omega}$ (Proposition~\ref{S4.2:Prop}\ref{it:rr0-omega}), we have the spectral projection lemma (Corollary \ref{Cor:FakeSpectralProjection}), which will be applied to positive elements $a_{n}\in \M$ controlling the commutation error and overlaps with projections produced in previous stages.
    \item $T(\M^\omega)=X^\omega$ as a consequence of strict comparison (by Proposition~\ref{prop:computing-traces}), and so $X^\omega$ is a Choquet simplex, thus satisfying the Riesz interpolation property and Riesz decomposition property (Corollary \ref{cor:riesz-properties}). 
    \item $\M^{\omega}$ has classification of projections (provided by Proposition~\ref{S4.2:Prop}\ref{it:classif-proj-omega}), which is used to carve out projections beneath $q$ with prescribed tracial data, and obtain a projection in $\M^{\omega}$ that fills in tracial defects. 
\end{itemize}

\begin{theorem}\label{SmallLocalApproximations} 
	Let $(\mathcal M, X)$ be a II$_1$ factorial amenable tracially complete $C^*$-algebra such that $\mathcal M^\omega$ has strict comparison with respect to $X^\omega$ and $\mathcal M^\omega$ has divisible unit.  For all $L \in \mathbb N$, $\epsilon \in (0, 1)$, and positive contractions $x_1, \ldots, x_L \in \mathcal M$, there exist $\gamma \in (0, 1)$, a finite-dimensional $C^*$-algebra $G \subseteq \mathcal M$, and positive contractions $y_1, \ldots, y_L \in G$ such that $1_G$ has constant trace $\gamma$ and 
	\begin{equation}\label{eq:new-1}
		\|x_\ell - (1_\mathcal M - 1_G) x_\ell (1_\mathcal M - 1_G)  - y_\ell \|_{2, X} \leq \epsilon\|1_G\|_{2, X}, \quad \ell = 1, \ldots, L.
	\end{equation}
\end{theorem}

\begin{proof}
	It suffices to construct a finite-dimensional $C^*$-algebra $G \subseteq \mathcal M^\omega$ and positive contractions $y_1, \ldots, y_L \in G$ such that $1_G$ has constant trace $\gamma$ and 
	\begin{equation}\label{eq:new-2}
		\|x_\ell - (1_{\mathcal{M}^\omega} - 1_G) x_\ell (1_{\mathcal{M}^\omega} - 1_G)  - y_\ell \|_{2, X^\omega} \leq \epsilon\|1_G\|_{2, X^{\omega}}, \quad \ell = 1, \ldots, L.
	\end{equation}
	Indeed, from here, Proposition \ref{Prop:ClassifyToM}\ref{it:LiftingToG} provides a sequence of $^*$-homomorphisms $(\phi_n \colon G \rightarrow \mathcal M)_{n=1}^\infty$ lifting the inclusion $G \rightarrow \mathcal M^\omega$, such that $\phi_n(1_G)$ has constant trace $\gamma$ for all $n \in \mathbb N$.  Then for each $\ell = 1 , \ldots, L$, 
	\begin{equation}\label{eq:new-3}
		\lim_{n \rightarrow \omega} \|x_\ell - (1_{\mathcal M} - 1_{\phi_n(G)})x_\ell (1_{\mathcal M} - 1_{\phi_{n}(G)}) - \phi_{n}(y_\ell)\|_{2, X} \leq \epsilon\|\phi_{n}(G)\|_{2, X}. 
	\end{equation}
	The result will then follow (with $2 \epsilon$ in place of $\epsilon)$ by taking $G$ and $y_1, \ldots, y_L$ to be $\phi_n(G)$ and $\phi_n(y_1), \ldots, \phi_n(y_L)$ for some $n \in \mathbb N$.
	
	With this reduction in place, define
	\begin{equation}\label{eq:new-4}
		\delta \coloneqq \frac{\epsilon^2}{28 + 8 \sqrt{2}} \qquad \text{and} \qquad \eta \coloneqq \frac{3\delta}{4L}.
	\end{equation}
	For each $\tau \in X$, note that $\pi_\tau(\mathcal M)''$ is hyperfinite by Connes' theorem (\cite{Connes:Annals}).\footnote{The hypothesis of amenability for a tracially complete $C^*$-algebra is the completely positive approximation property in the $\|\cdot\|_{2,X}$ norm.  As set out in \cite[Theorem 1.2]{carrion2023tracially}, is equivalent to amenability of the fibres $\pi_\tau(\mathcal M)''$ (passing through \cite[Theorem 4.12(i)$\Rightarrow$(iii)]{carrion2023tracially}, which, in turn, uses Brown's work on uniformly amenable traces (\cite[Theorem 3.2.2]{Brown:MAMS}) and Connes' theorem).  In the situation of most interest to us, when $A$ is nuclear and $\mathcal M=\completion{A}{T(A)}$, then the fibres $\pi_\tau(\mathcal M)''$ are $\pi_\tau(A)''$, which are injective. Hence our view that this is really Connes' theorem.}  Therefore, there is a unital finite-dimensional $C^*$-algebra $\tilde F_\tau \subseteq \pi_\tau(\mathcal M)''$ and positive contractions $\tilde y_{1, \tau}, \ldots, \tilde y_{L, \tau} \in \tilde F_\tau$ such that
	\begin{equation}\label{eq:new-5}
		\|\pi_\tau(x_\ell) - \tilde y_{\ell,\tau} \|_{2, \tau} < \sqrt\eta, \quad \ell = 1, \ldots, L.
	\end{equation}
	By Lemma~\ref{Lem:LiftMatrix}, we may find a unital finite-dimensional subalgebra $F_\tau \subseteq \mathcal M$ (isomorphic to $\tilde F_\tau$) along with positive contractions $y_{1, \tau}, \ldots, y_{L, \tau} \in F_\tau$ such that 
	\begin{equation}\label{eq:new-6}
		\|x_\ell - y_{\ell, \tau} \|_{2, \tau} < \sqrt\eta, \quad \ell = 1, \ldots, L.
	\end{equation}
	By compactness of $X$, there are $N \in \mathbb N$ and $\tau_1, \ldots, \tau_N \in X$ together with open neighbourhoods $U_1, \ldots, U_N \subseteq X$ covering $X$ such that for each $n = 1, \dots, N$, 
	\begin{equation}\label{eq:new-7}
		\|x_\ell - y_{\ell, \tau_n} \|_{2, \tau} < \sqrt\eta, \quad \ell = 1, \ldots, L,\ \tau_n \in U_n.
	\end{equation}
	Define
	\begin{equation}\label{eq:new-8}
		\beta \coloneqq \frac{\delta}{2N} \qquad \text{and} \qquad \gamma \coloneqq \frac{\beta \delta}{4 \max_n C_{F_n}},
	\end{equation}
    where $C_{F_n}$ is the constant from \eqref{cond_exp_bound}.
	
	Further define positive elements $a_1, \ldots, a_N \in \mathcal M_+$ by
	\begin{equation}\label{eq:new-9}
		a_n \coloneqq \sum_{\ell = 1}^L (x_\ell - y_{\ell, \tau_n})^2, \quad n = 1, \ldots, N,
	\end{equation}
	and note that \eqref{eq:new-7} implies
	\begin{equation}\label{eq:new-10}
		\sup_{\tau \in X} \min_{1 \leq n \leq N} \tau(a_i) < L \eta.
	\end{equation}
	Let $T \subseteq \mathcal M$ denote the separable $C^*$-algebra generated by $F_{\tau_1} \cup \cdots \cup \mathcal F_{\tau_N}$ and apply the weak CPoU condition recalled in Theorem~\ref{Thm:WCPoU} to obtain positive contractions $e_1, \ldots, e_N \in \mathcal M^\omega \cap T'$ such that 
	\begin{equation}\label{eq:new-11}
		\tau(a_n e_n) \leq \eta L \tau(e_n), \quad n = 1, \ldots, N,\ \tau \in X^\omega,\ 
	\end{equation}
	
	We will inductively construct partial isometries $v_1, \ldots, v_N \in \mathcal M^{\omega}$ such that for all $n = 1, \ldots, N$, the following properties hold.
	\begin{enumerate}
		\item\label{it:new-pi-1} $v_n^* v_n$ commutes with $F_{\tau_n}$;
		\item\label{it:new-pi-2} $v_nv_n^*$ is orthogonal to the projections $v_1 v_1^*, \ldots, v_{n-1} v_{n-1}^*$;
		\item\label{it:new-pi-3} $\gamma \tau(e_n) - \beta \gamma \leq \tau(v_n^*v_n) \leq \gamma \tau(e_n)$ for all $\tau \in X^\omega$;
		\item\label{it:new-pi-4}$ \tau(a_n v_n^* v_n) \leq \delta \tau(v_n^* v_n)$ for all $\tau \in X^\omega$;
		\item\label{it:new-pi-5} $\|v_n - v_n^*v_n\| < \sqrt\delta$.
	\end{enumerate}
	To this end, fix some $n = 1, \ldots, N$ and suppose $v_1, \ldots, v_{n-1}$ have been constructed.  Define $p = \sum_{i=1}^{n-1} v_i v_i^*$ (when $n = 1$, we interpret this as $p = 0$).  For later use, we record the inequality
	\begin{equation}\label{eq:new-12}
		\tau(e_n p) \leq \tau(p) = \sum_{i=1}^{n-1} \tau(v_nv_n^*) \stackrel{\ref{it:new-pi-3}}{\leq} \sum_{i=1}^{n-1} \gamma \tau(e_i) \leq \gamma, \quad \tau \in X^\omega.
	\end{equation}
	
	Let $E \coloneqq E_{F_{\tau_n}'} \colon \mathcal M^\omega \rightarrow \mathcal M^\omega \cap F_{\tau_n}'$ denote the conditional expectation from Definition~\ref{def:CondExp} and define\footnote{The first term $E(a_n)$ is to control the distance from elements $x_\ell \in \mathcal M$ to the corresponding elements $y_{\ell, \tau_n} \in F_{\tau_n}$, and the second term $C_{F_n} E(p)$ controls the overlap with the projections $v_1v_1^*, \ldots, v_{n-1}v_{n-1}^*$ already constructed.}
	\begin{equation}\label{eq:new-13}
		b \coloneqq E(a_n) + C_{F_n} E(p) \in \mathcal M^\omega_{+} \cap F_{\tau_n}'.
	\end{equation}
	By Corollary~\ref{Cor:FakeSpectralProjection}, there is a projection $q \in \mathcal M^\omega \cap F_{\tau_n}' \cap \{b\}'$
	 such that
	\begin{equation}\label{eq:new-14}
	 	qb \leq \delta q \qquad \text{and} \qquad (1 -q)b \geq \delta (1 - q).  
	\end{equation}
	 
\begin{claim}
   The following estimates hold:
	\begin{equation}\label{eq:new-15}
		\max \{0, \gamma \tau(e_n) - \gamma \beta\} \leq \min \{\tau(q), \gamma \tau(e_n) \}, \quad \tau \in X^\omega.
	\end{equation}
\end{claim}
\begin{claimproof}[Proof of Claim]
	To see this, note that the only non-trivial claim in \eqref{eq:new-15} is
	\begin{equation}\label{eq:new-16}
		\gamma \tau(e_n) - \gamma \beta \leq \tau(q), \quad \tau \in X^\omega.
	\end{equation}
	For the rest of this paragraph, let $\tau \in X^\omega$ be fixed.
	Using the second part of \eqref{eq:new-14}, 
	\begin{equation}\label{eq:new-17}
	\begin{split}
		\delta \tau(e_n(1 -q)) 
		&= \delta \tau(e_n^{1/2}(1-q)e_n^{1/2}) \\
		&\leq \tau(e_n^{1/2} b^{1/2} (1 - q) b^{1/2}e_n^{1/2}) \\
		&\leq \tau(e_n^{1/2} b e_n^{1/2}) \\
		&= \tau(e_n b).
	\end{split}
	\end{equation}
	By the definition of $b$ in \eqref{eq:new-13}, and as $E$ preserves traces on $\mathcal M^\omega$ (Remark~\ref{rem:cond_exp}\ref{rem:con_exp_1}), we have
	\begin{equation}\label{eq:new-18}
	\begin{array}{rcl}
		\tau(e_n b) & \leq & \tau(e_n a_n) + C_{F_{\tau_n}} \tau(e_n p) \\
		& \stackrel{\eqref{eq:new-11},\eqref{eq:new-12}}{\leq} & L\eta \tau(e_n) + C_{F_{\tau_n}} \gamma \\
		& \stackrel{\eqref{eq:new-8}}{\leq} & L\eta \tau(e_n) + \frac{\beta\delta}4.
	\end{array}
	\end{equation}
	Combining \eqref{eq:new-17} and \eqref{eq:new-18}, we see that
	\begin{equation}\label{eq:new-19}
		\tau(e_n (1 - q)) \stackrel{\eqref{eq:new-17}}{\leq} \frac1\delta \tau(e_nb) \stackrel{\eqref{eq:new-18}}{\leq} \frac{L\eta}{\delta} \tau(e_n) + \frac{\beta}{4},
	\end{equation}
	and hence
	\begin{equation}\label{eq:new-20}
	\begin{array}{rcl}
		\tau(q) & \geq & \tau(e_n q) \\
		& = & \tau(e_n) - \tau(e_n(1-q)) \\
		& \stackrel{\eqref{eq:new-19}}{\geq} & \big(1 - \frac{L\eta}{\delta} \big) \tau(e_n) - \frac\beta4 \\
		& \stackrel{\eqref{eq:new-4}}{=} & \frac14 (\tau(e_n) - \beta).
	\end{array}
	\end{equation}
	There are now two cases to consider.  If $\tau(e_n) - \beta < 0$, then certainly
	\begin{equation}\label{eq:new-21}
		\tau(q) \geq 0 \geq \gamma (\tau(e_n) - \beta).
	\end{equation}
	Otherwise, $\tau(e_n) - \beta \geq 0$.  Since $\gamma \leq \frac14$ (by \eqref{eq:new-8} and \eqref{eq:new-4} since we assume $\epsilon < 1$), we then have
	\begin{equation}\label{eq:new-22}
		\tau(q) \geq \frac14 (\tau(e_n) - \beta) \geq \gamma(\tau(e_n) - \beta).
	\end{equation}
	This proves \eqref{eq:new-16} and hence also the claim \eqref{eq:new-15}.
    \end{claimproof}
	
	Since $\mathcal M^\omega$ has strict comparison with respect to $X^\omega$, Proposition~\ref{prop:computing-traces} implies $X^\omega = T(\mathcal M^\omega)$.  In particular, $X^\omega$ is a Choquet simplex.  By Corollary~\ref{cor:riesz-properties}, $X^\omega$ has Riesz interpolation, and applying this to \eqref{eq:new-15} yields a continuous affine function $h \colon X^\omega \rightarrow \mathbb R$ such that 
	\begin{equation}\label{eq:new-23}
		\max \{0, \gamma \tau(e_n) - \gamma \beta \} \leq h(\tau) \leq \min \{ \tau(q), \gamma \tau(e_n) \}, \quad \tau \in X^\omega.
	\end{equation} 
    As $\M^{\omega}$ has classification of projections and $X^{\omega}$ is a Choquet simplex, an application of Lemma~\ref{lem:TraceProjInCommutant} to $(\M^{\omega}, X^{\omega})$ and $F_{\tau_{n}}$ produces a projection $z\in \M^{\omega}\cap F_{\tau_{n}}'$ satisfying $z\leq q$ and $\tau (z) = h(\tau)$ for all $\tau \in X^{\omega}$. 
    In particular, \eqref{eq:new-23} implies
	\begin{equation}\label{eq:new-27}
		\gamma \tau(e_n) - \gamma \beta \leq \tau(z) \leq \gamma \tau(e_n), \quad \tau \in X^\omega.
	\end{equation}
	
	For $\ell = 1, \ldots, L$, since $z$ commutes with $y_{\ell, \tau_n} \in F_{n_\tau}$, we have for $\tau \in X^\omega$,
	\begin{equation}\label{eq:new-28}
	\begin{array}{rcl}
		\|[z, x_\ell]\|_{2, \tau}^2 & = & \|[z, x_\ell- y_{\ell, \tau_n}] \|_{2, \tau}^2 \\ 
		& \leq & 2 \tau(z (x_\ell - y_{\ell, \tau_n})^2) \\
		& \stackrel{\eqref{eq:new-9}}{\leq}  & 2\tau(a_n z),
	\end{array} 
	\end{equation} 
    using Lemma~\ref{2409_LemPosiSelf} and that $z$ is a projection.  Now as $z \leq q$, $z, q \in \mathcal M^\omega \cap F_{\tau_n}'$, and $E$ preserves traces on $\mathcal M^\omega$ (Remark~\ref{rem:cond_exp}\ref{rem:con_exp_1}), 
	\begin{equation}\label{eq:new-29}
		\tau(a_n z) = \tau(E(a_n) z) \leq \tau(b z) \stackrel{\eqref{eq:new-14}}{\leq} \delta \tau(z), \quad \tau \in X^\omega,
	\end{equation}
	by using \eqref{eq:new-14} at the stated point and compressing by $z$.
	Accordingly,
	\begin{equation}\label{eq:new-30}
		\|[z, x_\ell]\|_{2, \tau} \stackrel{\eqref{eq:new-28}}{\leq} \sqrt{2\tau(a_{n}z)} \stackrel{\eqref{eq:new-29}}{\leq} \sqrt{2\delta\tau(z)} = \sqrt{2\delta} \|z\|_{2, \tau}, \quad \tau \in X^\omega.
	\end{equation}
	
	We will now define $v_n$ by conjugating $z$ under $1 - p$.  As
	\begin{equation}\label{eq:new-31}
		qpq \stackrel{\eqref{cond_exp_bound}}{\leq} C_{F_{\tau_n}} q E(p) q \stackrel{\eqref{eq:new-13}}{\leq} qbq \stackrel{\eqref{eq:new-14}}{\leq} \delta q,
	\end{equation}
	we can compress by $z$ to get $zpz \leq \delta z$.  Then Lemma~\ref{Polar} produces a partial isometry $v_n \in \mathcal M^\omega$ such that
	\begin{equation}\label{eq:new-32}
		v_n^*v_n = z, \quad vv^* \leq 1 - p, \quad \text{and} \quad \|v_n - z\| < \sqrt\delta.
	\end{equation}
	Note that \ref{it:new-pi-1}, \ref{it:new-pi-2}, and \ref{it:new-pi-5} are immediate from \eqref{eq:new-32} by the definition of $z$ and $p$ (defined above \eqref{eq:new-27} and \eqref{eq:new-12}, respectively).  Also, \ref{it:new-pi-3} and \ref{it:new-pi-4} follow from \eqref{eq:new-27} and \eqref{eq:new-29}, respectively.  This completes the construction of $v_1, \ldots, v_N$.

    We note that \ref{it:new-pi-5} implies that for $\tau \in X^\omega$,
    \begin{equation}
    \label{eq:new-99} 
    \|v_n-v_n^*v_n\|_{2,\tau} = \|v_n-v_n^*v_n\|\,\|v_n^*v_n\|_{2,\tau} \leq \sqrt{\delta}\|v_n^*v_n\|_{2,\tau}
    \end{equation} 
	Using \ref{it:new-pi-3} and that $\sum_{n=1}^N e_n = 1_{\mathcal M^\omega}$, we have
	\begin{equation}\label{eq:new-33}
		\sum_{n=1}^N \tau(v_nv_n^*) \leq \gamma, \quad \tau \in X^\omega.	
	\end{equation}
	As $\mathcal M^\omega$ has classification of projections (Proposition~\ref{S4.2:Prop}\ref{it:classif-proj-omega}), there is a projection $d \in \mathcal M^\omega$ orthogonal to $v_1v_1^*$, \ldots, $v_N v_N^*$ such that
	\begin{equation}\label{eq:new-34}
		\tau(d) = \gamma - \sum_{n=1}^N \tau(v_n v_n^*), \quad \tau \in X^\omega.
	\end{equation} 
	For each $n \in \mathbb N$, $v_n^*v_n$ commutes with $F_{\tau_n}$ by \ref{it:new-pi-1}, so $v_n F_{\tau_n} v_n^* \subseteq \mathcal M^\omega$ is a finite-dimensional $C^*$-algebra.  Define a finite-dimensional $C^*$-algebra 
	\begin{equation}\label{eq:new-35}
		G \coloneqq v_1 F_{\tau_1} v_1^* \oplus \cdots \oplus v_N F_{\tau_N} v_N^* \oplus \mathbb C d \subseteq \mathcal M^\omega
	\end{equation}
	and note that $1_G$ has constant trace $\gamma$ by \eqref{eq:new-34}.  Further define
	\begin{equation}\label{eq:new36}
		y_\ell \coloneqq \sum_{n=1}^N v_n y_{\ell, \tau_n} v_n^* \in G.
	\end{equation}
	
    It remains to prove \eqref{eq:new-2}.  For the rest of the proof, fix some $\ell \in \{1, \ldots, L\}$ and $\tau \in X^\omega$.  Note that Lemma~\ref{Pythagoras} yields
    \begin{equation}\label{eq:new-44}
    \begin{split}
		 & \|x_\ell - (1_{\mathcal{M}^\omega} - 1_G) x_\ell(1_{\mathcal{M}^\omega} - 1_G) - y_\ell \|_{2, \tau}^2 \\
         \leq & \sum_{n =1}^N \|x_\ell - (1_{\mathcal M^\omega} - v_n v_n^*) x_\ell (1_{\mathcal M^\omega} - v_n v_n^*) - v_ny_{\ell, \tau_n}v_n^*\|_{2, \tau}^2 \\
         & + \|x_\ell - (1_{\mathcal M^\omega} - d)x_\ell (1_{\mathcal M^\omega} -d)|\|_{2, \tau}^2. 
	\end{split}
    \end{equation} 
    We will bound these terms separately. Starting with the last term, since $d$ is a projection, we have
    \begin{equation}
    \label{eq:new-43}
	\begin{array}{rcl}
		\|x_\ell - (1_{\mathcal M^\omega} - d) x_\ell (1_{\mathcal M^\omega} - d) \|_{2, \tau}^2 & =&  \|d x_{\ell} \|_{2, \tau}^2  + \| (1_{\mathcal M^\omega} - d)x_\ell d\|_{2, \tau}^2 \\
		&\leq& 2 \tau(d) \\ 
		&\stackrel{\eqref{eq:new-34}}{=} & 2\gamma - 2\sum_{n=1}^N \tau(v_n v_n^*) \\
		& \stackrel{\ref{it:new-pi-3}}{\leq} & 2N\beta \gamma \\
        &\stackrel{\eqref{eq:new-8}}{=}& \delta \gamma.
	\end{array}
	\end{equation}
    Next for each $n = 1, \ldots, N$, we have 
    \begin{equation}\label{eq:new-37}
		\|(x_\ell - y_{\ell, \tau_n})v_n^*v_n\|_{2, \tau}^2 = \tau((x_\ell - y_{\ell, \tau_n} )^2 v_n^*v_n) \stackrel{\eqref{eq:new-9}}{\leq} \tau(a_n v_n^* v_n) \stackrel{\ref{it:new-pi-4}}{\leq} \delta \tau(v_n^*v_n),
	\end{equation}
	and hence 
    \begin{equation}\label{eq:new-39}
	\begin{array}{rcl}
		\|[x_\ell, v_n^*v_n]\|_{2, \tau}^2 &\stackrel{\ref{it:new-pi-1}}{=} & \|[x_\ell - y_{\ell, \tau_n}, v_n^*v_n]\|_{2, \tau}^2 \\ 
		& \leq & 2 	\tau (v_{n}^{*}v_{n} (x_{\ell} - y_{\ell, \tau_{n}})^{2}) \\
        & = & 2 \|(x_\ell - y_{\ell, \tau_n})v_n^*v_n\|_{2, \tau}^2\\ 
		& \stackrel{\eqref{eq:new-37}}{\leq} & 2 \delta \tau(v_n^*v_n),
	\end{array}
	\end{equation}
    where the first inequality is true by Lemma~\ref{2409_LemPosiSelf} and that $v_{n}^{*}v_{n}$ is a projection. 
	This in turn implies that for all $n = 1, \ldots N$, and $\ell = 1, \ldots, L$,
	\begin{equation}\label{eq:new-40}
	\begin{array}{rcl}
		\|[x_\ell, v_n v_n^*]\|_{2, \tau} & \leq & 2 \|v_n^* v_n - v_n v_n^* \|_{2, \tau} + \|[x_\ell, v_n^* v_n]\|_{2, \tau} \\
		& \stackrel{\eqref{eq:new-39}}{\leq} &  2 (\|v_n^* v_n -v_n\|_{2,\tau}+\|v_n-v_nv_n^*\|_{2,\tau}) + \sqrt{2\delta \tau(v_n^*v_n)} \\
        & \leq &  2 (\|v_n^* v_n -v_n\|_{2,\tau}+\|v_n(v_n^*v_n - v_n^*) \|_{2, \tau}) + \sqrt{2\delta \tau(v_n^*v_n)} \\
        &\stackrel{\eqref{eq:new-99}}\leq & (4 + \sqrt{2}) \sqrt{\delta}\|v_n^*v_n\|_{2,\tau}.
	\end{array}
	\end{equation}
	Also, for $n = 1, \ldots, N$ and $\ell = 1, \ldots, L$, we have
    \begin{equation}\label{eq:new-45}
    \begin{array}{rcl}
        &&\hspace*{-6em} \|v_n v_n^* x_\ell v_n     v_n^* - v_n y_{\ell, \tau_n} v_n^* \|_{2, \tau} \\ 
        &=& \| v_{n}^{*} x_{\ell} v_{n} - v_{n}^{*} v_{n} y_{\ell, \tau_{n}} v_{n}^{*} v_{n} \|_{2, \tau} \\ 
        &=& \|v_{n}^{*} v_{n} (x_{\ell}-y_{\ell, \tau_{n}}) v_{n}^{*} v_{n} + v_{n}^{*}x_{\ell} (v_{n} - v_{n}^{*} v_{n}) + (v_{n}^{*} - v_{n}^{*} v_{n}) x_{\ell} v_{n}^{*} v_{n} \|_{2, \tau}\\ 
        & \leq & \|v_{n}^{*} v_{n} (x_{\ell}-y_{\ell, \tau_{n}}) v_{n}^{*} v_{n}\|_{2, \tau} + 2 \|v_{n} - v_{n}^{*} v_{n}\|_{2, \tau}\\
        &\stackrel{\eqref{eq:new-37},\eqref{eq:new-99}}{\leq}& 3\sqrt\delta \|v_n^*v_n\|_{2,\tau}, 
    \end{array} 
    \end{equation} 
    where the first equality is true as multiplying by $v_{n}^{*}$ and $v_{n}$ on each side does not change the $2$-norm. 
    
    Returning to the first term in \eqref{eq:new-44}, for each $n = 1, \ldots, N$, we compute\footnote{The first inequality is best seen by thinking of a $2 \times 2$ matrix decomposition with respect to $v_n v_n^*$ and its complement.} 
    \begin{equation}\label{eq:new-46}
    \begin{array}{rcl}
        &&\hspace*{-6em}\|x_\ell - (1_{\mathcal M^\omega} - v_n v_n^*) x_\ell (1_{\mathcal M^\omega} - v_n v_n^*) - v_ny_{\ell, \tau_n}v_n^*\|_{2, \tau}^2 \\
        & = & \|[x_\ell, v_nv_n^*]\|_{2, \tau}^2 + \|v_n v_n^* x_\ell v_nv_n^* - v_n y_{\ell, \tau_n} v_n^*\|_{2, \tau}^2  \\
        &\stackrel{\eqref{eq:new-40},\eqref{eq:new-45},\ref{it:new-pi-3}}\leq& (4 + \sqrt2)^2 \delta \gamma \tau(e_n) + 9\delta \gamma \tau(e_n)  \\
        &=& (27 + 8\sqrt2)\delta\gamma \tau(e_n),
    \end{array}
    \end{equation}  
    Combining \eqref{eq:new-46} and \eqref{eq:new-43} with \eqref{eq:new-44} yields
    \begin{equation}
        \|x_\ell - (1_{\mathcal{M}^\omega} - 1_G) x_\ell(1_{\mathcal{M}^\omega} - 1_G) - y_\ell \|_{2, \tau}^2 \leq (28 + 8 \sqrt{2})\delta\gamma = \epsilon^2 \tau(1_G),
    \end{equation}
    using the definition of $\delta$ in \eqref{eq:new-4} for the final equality.  Now \eqref{eq:new-2} follows by taking square roots. 
\end{proof}

\begin{remark} 
It is in establishing the claim in the proof of Theorem~\ref{SmallLocalApproximations} where we use that the elements given by weak CPoU form a partition of unity which is morally subordinate to the given convex open cover.  The main difficulty is bounding $\tau(q)$ from below, which splits into cases according to whether or not $\tau (e_{n}) < \beta$
for a small value $\beta$. This should be compared with the more classical partition of unity argument in the sketch proof of Lemma~\ref{RobotLemma} in Section~\ref{RobotSection}, where the corresponding lower bound comes from separately considering $\tau(f_n)$ for $\tau\in U_n$, and that $f_n$ is supported $U_n$.  In both cases, this gives enough space to obtain the required projection via the classification of projections.
\end{remark}

\subsection{The structure theorem for factorial amenable tracially complete \cstar-algebras} 
We are now in position to follow the last step of Popa's proof of injectivity implies hyperfiniteness from \cite{Popa:InjectiveHyperfinite}. 

Recall that, in exact analogy with tracial von Neumann algebras, a tracially complete $C^*$-algebra $(\mathcal M, X)$ is \emph{hyperfinite} (\cite[Definition~8.1]{carrion2023tracially}) if all $n \in \mathbb N$, $x_1, \ldots, x_n \in \mathcal M$, and $\epsilon > 0$, there are a finite-dimensional $C^*$-algebra $F \subseteq \mathcal M$ and $y_1, \ldots, y_n \in F$ such that $\|x_i - y_i\|_{2, X} < \epsilon$ for all $i = 1, \ldots, n$.  In the separable, factorial setting, this is equivalent to being an inductive limit of finite-dimensional tracially complete $C^*$-algebras (\cite[Theorem~8.4]{carrion2023tracially}). Just like in Popa's proof that injective implies hyperfinite (and in Connes' original proof from \cite{Connes:Annals}), it is the local characterisation of hyperfiniteness which will be verified.  

\begin{theorem}\label{PureAmenable=>Hyperfinite}	
Let $(\mathcal M, X)$ be a II$_1$ factorial amenable tracially complete $C^*$-algebra such that $\mathcal M^\omega$ has strict comparison with respect to $X^\omega$ and $\mathcal M^\omega$ has divisible unit.  Then $(\mathcal M,X)$ is hyperfinite. 
\end{theorem}
\begin{proof}
Fix a finite set $\mathcal F=\{x_1,\dots,x_L\}$ of positive contractions in $\mathcal M$ and $\epsilon>0$. 
    We consider families 
\begin{equation}
    (I,G,y)=\Big\{\bigoplus_{i\in I}G_i,(y_{1,i},\dots,y_{L,i})_{i\in I}\Big\}
\end{equation}
where:
\begin{itemize}
    \item $G_i$ are finite-dimensional subalgebras of $\mathcal M$ with pairwise orthogonal units $1_{G_i}$ which have constant trace; 
    \item $y_{\ell,i}$ are contractions in $G_i$ for $i\in I$, and $\ell=1,\dots,L$; and
    \item For $\ell=1,\dots,L$, we have
\begin{equation}\label{ZL.1}
    \Big\|x_\ell-(1-\sum_{i\in I}1_{G_i})x_{\ell}(1-\sum_{i\in I}1_{G_i})-\sum_{i\in I}y_{\ell,i}\Big\|_2^2\leq\epsilon^2\sum_{i\in I}\tau(1_{G_i}).
\end{equation}
\end{itemize}
Note that the sums in \eqref{ZL.1} converge in $\|\cdot\|_{2,X}$-norm, from the assumption that each $1_{G_i}$ has constant trace.

We order such families by saying that $(I,G,y)\leq (J,H,z)$, when there is an injection $\iota:I\to J$ such that $H_{\iota(i)}=G_i$ and $z_{\ell,\iota(i)}=y_{\ell,i}$ for all $i\in I$ and $\ell=1,\dots,L$. This says that $\bigoplus_{i\in I}G_i\subset\bigoplus_{j\in J}H_j$, and the relevant terms of the approximating elements are equal.

Given a chain, the union gives an upper bound (the approximation condition in \eqref{ZL.1} follows by continuity).  So by Zorn's lemma we can take a maximal element $(I,G,y)$.  Suppose, by way of reaching a contradiction that $\sum_{i\in I}\tau(1_{G_i})<1$. Then $q=\sum_{i\in I}1_{G_i}$ is $\|\cdot\|_{2,X}$-convergent.  We now work in the corner $(1-q)\M(1-q)$, and note that this corner inherits the hypotheses on $(\M,X)$ when equipped with the traces 
\begin{equation}
    \Big\{ \frac{1}{\tau(1-q)} \tau|_{(1-q)\mathcal M(1-q)} : \tau \in X \Big\} \subseteq T((1-q)\mathcal M(1-q)).
\end{equation}  
Indeed, strict comparison of $((1-q) \mathcal M^\omega (1 -q))^\omega = (1 - q)\mathcal M^\omega (1 -q)$ follows immediately from strict comparison in $\mathcal M^\omega$, and divisibility of the unit comes from classification of finite-dimensional subalgebras of $\mathcal M^\omega$, recorded in Proposition~\ref{S4.2:Prop}\ref{it:classif-proj-omega}.

By Proposition \ref{SmallLocalApproximations}, there exists some non-zero finite-dimensional subalgebra $H$ of $(1-q)\mathcal M(1-q)$ whose unit has constant trace, and approximates $z_1,\dots,z_L\in H$ with\footnote{The trace is initially computed in the corner, and then rescaled back to $\M$}
\begin{equation}
    \|(1-q)x_\ell(1-q)-(1-q-1_H)x_\ell(1-q-1_H)-z_\ell\|_{2,X}\leq \epsilon\|1_H\|_{2,X}.
\end{equation}
For each $\ell=1,\dots,L$, and $\tau\in X$, we then compute (identically to Popa, and again best done by drawing some block matrix decompositions), 
\begin{equation}
    \begin{split}
        &\|x_\ell-(1-\sum_{i\in I}1_{G_i}-1_H)x_\ell(1-\sum_{i\in I}1_{G_i}-1_H)-(\sum_{i\in I}y_{\ell, i}+z_{\ell})\|_{2,\tau}^2\\
        =&\|x_\ell-(1-\sum_{i\in I}1_{G_i})x_{\ell}(1-\sum_{i\in I}1_{G_i})-\sum_{i\in I}y_{\ell, i}\|_{2, \tau}^2\\
        &+\|(1-\sum_{i\in I}1_{G_i})x_{\ell}(1-\sum_{i\in I}1_{G_i})-(1-\sum_{i\in I}1_{G_i}-1_H)x_{\ell}(1-\sum_{i\in I}1_{G_i}-1_H)-z_{\ell}\|_{2,\tau}^2\\ 
        &\leq \epsilon^2(\sum_{i\in I}\tau(1_{G_i})+\tau(1_H)).
    \end{split}
\end{equation} 
This contradicts maximality, and hence $\sum_{i\in I}\tau(G_i)=1$.  

For all $\ell=1,\dots,L$, we have
\begin{equation}
        \|x_\ell-\sum_{i\in I}y_i\|_{2,\tau}\leq\epsilon, 
\end{equation}
so we can find a finite $I_0$ such that
\begin{equation}
        \|x_\ell-\sum_{i\in I_0}y_i\|_{2,\tau}\leq2\epsilon.
\end{equation}
Taking $F=\bigoplus_{i\in I_0}G_i$, shows that $(\M,X)$ is hyperfinite.
\end{proof}

Combining Theorem \ref{PureAmenable=>Hyperfinite} with the structure theorem from \cite{carrion2023tracially} gives our structure theorem for II$_1$ factorial amenable tracially complete $C^*$-algebras.  Note that $\|\cdot\|_{2,X}$-separability is not required; all the conditions below are defined locally (but also recall that McDuffness of $(\M,X)$ is only equivalent to the tensorial absorption $(\M,X)\cong (\M,X)\,\overline{\otimes}\,(\mathcal R,\{\tau_{\mathcal R}\})$ when $(\M,X)$ is $\|\cdot\|_{2,X}$-separable).

\begin{corollary}[{Theorem \ref{TC:Structure}}]\label{TC:Structure:MainBody}
    Let $(\M,X)$ be a II$_1$ factorial amenable tracially complete $C^*$-algebra. The following are equivalent:
    \begin{enumerate}
        \item $(\M,X)$ is  hyperfinite;\label{TC:Structure:MainBody.1}
        \item $(\M,X)$ is McDuff;\label{TC:Structure:MainBody.2}
        \item[\namedlabel{TC:Structure:MainBody.2'}{(ii$'$)}] $(\M,X)$ has property $\Gamma$;
\item[\namedlabel{TC:Structure:MainBody.2''}{(ii$''$)}] $(\M,X)$ has CPoU;
        \item $(\M,X)$ has the property that $\M^\omega$ has strict comparison by $X^\omega$ and the unit of $\M^\omega$ is divisible.\label{TC:Structure:MainBody.3}
    \end{enumerate}
\end{corollary}
\begin{proof}
    The equivalence of \ref{TC:Structure:MainBody.1}, \ref{TC:Structure:MainBody.2}, \ref{TC:Structure:MainBody.2'},
    and \ref{TC:Structure:MainBody.2''} is \cite[Theorem~9.15]{carrion2023tracially}.
    The implication \ref{TC:Structure:MainBody.2''}$\Rightarrow$\ref{TC:Structure:MainBody.3} holds without the amenability assumption. This is the combination of \cite[Proposition 7.2]{carrion2023tracially} (which shows $\M^\omega$ has real rank zero), \cite[Proposition 7.4]{carrion2023tracially} (comparison of projections in $(\M^\omega,X^\omega)$), and \cite[Theorem 7.18]{carrion2023tracially} (existence of projections from CPoU, which applies to $(\M^\omega,X^\omega)$ as this inherits CPoU from $\M$ by \cite[Corollary 7.7(i)]{carrion2023tracially}).\footnote{The authors of \cite{carrion2023tracially} neglected to explicitly collect the existence of projections in $\M^\omega$ when $(\M,X)$ has CPoU, focusing instead on the existence of projections in $(\M,X)$.}   Then classification of projections in $\mathcal M^\omega$ implies the classification of finite-dimensional subalgebras by Proposition~\ref{Class:FD}, which in turn implies $\mathcal M^\omega$ has divisible unit.  Then Proposition~\ref{prop:rr0-comparison} implies $\mathcal M^\omega$ has strict comparison with respect to $X^\omega$. 
    The final piece is \ref{TC:Structure:MainBody.3}$\Rightarrow$\ref{TC:Structure:MainBody.1}, which we have just established in Theorem \ref{PureAmenable=>Hyperfinite}.
\end{proof}
 
\section{The structure theorem for simple separable nuclear \cstar-algebras}\label{S:StructureNuclear}

With the structure theorem for II$_1$ factorial amenable tracially complete $C^*$-algebras in place, we can apply Matui and Sato's theorem and obtain the structure theorems for simple separable nuclear $C^*$-algebras set out in the introduction.

To handle the general non-unital stably finite case, we will pass to a suitable choice of hereditary subalgebra of the stabilisation where all trace are bounded and the traces simplex is both non-empty and compact. By Brown's theorem (\cite{Brown:PJM}), the hereditary subalgebra will be stably isomorphic to the original $C^*$-algebra, so all the relevant regularity properties can be transferred between the original $C^*$-algebra and the hereditary subalgebra. This idea is well-known to experts and variants of it have previously been used to generalise progress on the unital Toms--Winter conjecture to the non-unital setting (see \cite{Tikuisis:MathAnn,CE:APDE,CE:PAMS}).

\begin{lemma}\label{lem:reduction-a-la-brown}
    Let $A$ be a stably finite simple separable nuclear pure $C^*$-algebra. 
    Then there is a stably finite simple separable nuclear pure $C^*$-algebra $B$ that is stably isomorphic to $A$ such that the trace space $T(B)$ is compact and $B$ has strict comparison with respect to $T(B)$. In particular, $A$ is $\Z$-stable if and only if $B$ is $\Z$-stable.
\end{lemma}
\begin{proof}
    We note that the following properties are invariant under stable isomorphism: stable finiteness, simplicity, separability, nuclearity and $\Z$-stability. (The latter is \cite[Section 3]{TomsWinter:TAMS} and the rest are standard; see \cite[Proposition 2.3]{CE:APDE} and the references therein for an overview). 
    Pureness is also invariant under stable isomorphism because it is a property of the Cuntz semigroup, which is by definition a stable isomorphism invariant.
    We will construct $B$ such that $T(B)$ is compact and all densely-defined lower-semicontinuous traces on $B$ are bounded.
    By pureness, it then follows that $B$ has strict comparison with respect to $T(B)$.
    
    There are two cases to consider: either $A \otimes \K$ contains a non-zero projection or it does not.
    
    First, suppose there is a non-zero projection $p \in A \otimes \K$. 
    Set $B = p(A \otimes \K)p$.
    Clearly both $B$ and $A\otimes\K$ are separable, so have strictly positive elements. 
    As $A$ is simple, so is $A \otimes \K$; hence, $B$ is a full hereditary subalgebra of $A \otimes \K$. 
    Therefore, by Brown's theorem \cite[Theorem 2.8]{Brown:PJM}, $B$ is stably isomorphic to $A$.
    Clearly, $B$ is unital (with unit $p$). 
    Since $B$ is simple and unital, all densely-defined lower-semicontinuous traces on $B$ are bounded, so up to scaling they are tracial states. 
    Since $B$ is unital stably finite and exact, $T(B)$ is non-empty by \cite{BH82} and \cite{Haagerup}.
    As $B$ is unital, $T(B)$ is compact.
    
    Second, suppose $A \otimes \K$ contains no non-zero projections (i.e.\ $A$ is stably projectionless). By \cite[Lemma 3.3]{CE:PAMS}, which is also a hereditary subalgebra argument using Brown's theorem, there exists a $C^*$-algebra $B$ such that $B$ is stably isomorphic to $A$ such that all densely-defined lower-semicontinuous traces on $B$ are bounded and the trace space $T(B)$ is both non-empty and compact. 
\end{proof}

Next, we isolate the application of the structure theorem for factorial amenable tracially complete $C^*$-algebras that we will need. This result was stated earlier as Theorem~\ref{Thm:McDuff1} in the discussion of methods section. We restate it here for the benefit of the reader.

\begin{theorem}[Theorem~\ref{Thm:McDuff1}]\label{thm:McDuff1again}
    Let $A$ be a separable and nuclear $C^*$-algebra such that $T(A)$ is compact and non-empty.  Suppose that $A$ has strict comparison with respect to $T(A)$ and an almost divisible strictly positive element.  Then its tracial completion $(\completion{A}{T(A)},T(A))$ is McDuff.
\end{theorem}
\begin{proof}
Let $(\M,X) = (\completion{A}{T(A)},T(A))$ be the tracial completion of $A$. Then $\M$ is factorial by \cite[Proposition 3.23(iv)]{carrion2023tracially}. Moreover, $\M$ is amenable by \cite[Corollary 4.10]{carrion2023tracially}.\footnote{Note that this is really Connes' theorem, as the content is that each $\pi_\tau(A)''$ is an amenable von Neumann algebra.}  Applying Proposition~\ref{prop:sc_divisible_tracial_ultra} and Proposition \ref{prop:tracial-ultrapower-is-TC}, we see that $\M^\omega$ has strict comparison with respect to $X^\omega$ and has a divisible unit. By Corollary \ref{TC:Structure:MainBody} (iii)$\Rightarrow$(ii), $\M$ is McDuff.
\end{proof}

We now prove Theorem~\ref{Intro:pure-implies-Zstable} from the introduction. After handling the purely infinite case, we first use Lemma~\ref{lem:reduction-a-la-brown} to reduce to a setting where Theorem~\ref{Thm:McDuff1} is applicable to get that $A$ is uniformly McDuff. We then use a suitable non-unital version of Matui--Sato's $\Z$-stability theorem to finish the job. 

\begin{theorem}[Theorem~\ref{Intro:pure-implies-Zstable}]\label{Intro-body:ThmA}
    Let $A$ be a simple separable nuclear pure $C^*$-algebra.  Then $A$ is $\mathcal Z$-stable.
\end{theorem}
\begin{proof}
    Let $A$ be a simple separable nuclear pure $C^*$-algebra. Then $A$ is exact, and so all lower-semicontinuous densely-defined quasitraces on $A \otimes \K$ are additive by Haagerup's theorem; see \cite{Haagerup} for the unital case and \cite[Remark 2.29(i)]{BlanchardKirchberg:JFA} for the non-unital version. If $A$ is traceless, then $A$ is a Kirchberg algebra,\footnote{This is well known, as when $T(A)=\emptyset$ we have $a\precsim b$ for all non-zero $a,b\in A$, and hence $A$ is simple and purely infinite} so is $\mathcal O_\infty$-stable by Kirchberg's $\mathcal O_\infty$-absorption theorem (\cite{K:ICM,KP:Crelle}), and finally $\Z$-stable as $\mathcal O_\infty$ is $\Z$-stable (by the Kirchberg--Phillips theorem, for example).
    
    Hence, we may assume that $A$ is stably finite. 
    By Lemma~\ref{lem:reduction-a-la-brown}, we may further that $T(A)$ is compact and $A$ has strict comparison with respect to $T(A)$ (without affecting the stable isomorphism class of $A$). Since $A$ is separable, it has a strictly positive element $h \in A_+$. As $A$ is pure, $h$ is almost divisible.
    Then by Theorem~\ref{thm:McDuff1again},  $A$ has the uniform McDuff property. In the unital case, $\Z$-stability of $A$ now follows by Matui--Sato's $\Z$-stability theorem (see for example \cite[Theorem 2.6]{TWW:IMRN}). In the non-unital case, we need to use an appropriate non-unital version of Matui and Sato's $\Z$-stability theorem. A clean form of this is \cite[Corollary 3.6]{Szabo:MJM} (taking the group there to be trivial, so that equivariant $\Z$-stability is just $\Z$-stability), which builds on \cite[Theorem 5.4]{SW:APDE}. 
\end{proof}
    
Combining Theorem~\ref{Intro:pure-implies-Zstable} with the recent result of Thiel on realising ranks (\cite{Thiel:arXiv}), the proof of the Toms--Winter conjecture is now complete.

\begin{corollary}[Theorem~\ref{Intro:Structure}]\label{Intro-body:Toms-Winter-Theorem}
    Let $A$ be a simple separable nuclear non-elementary $C^*$-algebra. The following are equivalent:
    \begin{enumerate}[(i)]
        \item $A$ has finite nuclear dimension,
        \item $A$ is $\Z$-stable,
        \item $A$ has strict comparison (i.e.\ $\Cu(A)$ is almost unperforated).
    \end{enumerate}
\end{corollary}

\begin{proof}
The equivalence between finite nuclear dimension and $\Z$-stability was established in \cite{winter2012nuclear,Tikuisis:MathAnn,CETWW,CE:APDE}. Moreover, $\Z$-stable $C^*$-algebra have strict comparison by \cite{Rordam:IJM}.
Let $A$ be a simple separable nuclear non-elementary $C^*$-algebra with strict comparison. Then, by \cite{Thiel:arXiv}, $A$ is pure.\footnote{This is \cite[Corollary B]{Thiel:arXiv} in the unital case, and \cite[Theorem E]{Thiel:arXiv} in the non-unital case (as simple $C^*$-algebras have the global Glimm property).} Hence, $A$ is $\Z$-stable by Theorem~\ref{Intro:pure-implies-Zstable}.  
\end{proof}

\section{Regular tracially complete \cstar-algebras}\label{S:RegTCAlg}

This section represents Steps~\ref{Step.8} and \ref{Step.9} of the outline in Section~\ref{subsec:organisation}; that is, we prove that for a factorial tracially complete $C^*$-algebra $(\mathcal M, X)$, if $\mathcal M^\omega$ is pure as a $C^*$-algebra, then firstly $\M$ has real rank zero and, using this, $\M$  is pure.

We have already obtained the classification of projections in $\mathcal M$ from Theorem~\ref{thm:classif-proj}, and applying this to matrices computes the Murray--von Neumann semigroup $V(\mathcal M) \cong \mathrm{Aff}(X)_+$.  Since this semigroup is divisible and unperforated, it will follow from Proposition~\ref{prop:ETpure} that $\mathcal M$ is pure as soon as $\mathcal M$ has real rank zero.

The bulk of this section is hence devoted to showing that $\mathcal M$ has real rank zero.  It was shown in \cite[Theorem~2.5]{ET:RR0} that factorial tracially complete $C^*$-algebras with CPoU have real rank zero.  Our proof uses a similar argument, via a weakened form of CPoU.  Specifically, we show that for a unital embedding $\mathcal R \subseteq \mathcal M^\omega$, the relative commutant $\mathcal M^\omega \cap \mathcal R'$ has CPoU with respect to the set of traces $X^\omega$ (Proposition~\ref{prop:commutant-cpou}), and we will apply this to embeddings constructed by the next lemma.  

\begin{lemma}\label{lem:central-divide-interval}
    Let $(\M,X)$ be a type II$_1$ factorial tracially complete $C^*$-algebra such that $\M^\omega$ has strict comparison with respect to $X^\omega$ and $\mathcal M^\omega$ has divisible unit.  If $a \in \mathcal M^\omega$ is self-adjoint, then there is a unital embedding $\iota \colon \mathcal R \rightarrow \mathcal M^\omega \cap \{a\}'$.
\end{lemma}

\begin{proof}
    By a reindexing argument, it suffices to show that for all $\epsilon > 0$ and $n \in \mathbb N$, there is a unital embedding $\iota_n \colon M_n \rightarrow \mathcal M^\omega$ such that $\|a x - x a\|_{2, X} \leq 2\epsilon\|x\|$ for all $x \in M_n$.  Recall that $\mathcal M^\omega$ has real rank zero by \ref{S4.2:Prop}\ref{it:rr0-omega}, so there is a self-adjoint $b \in \mathcal M^\omega$ with finite spectrum such that $\|a - b\| < \epsilon$.  Then $\mathcal M^\omega \cap \{b\}'$ is a direct sum of corners of $\mathcal M$ (as in Lemma~\ref{Lem:RelCom}).  Using Proposition \ref{Class:FD} (on each corner), there is a unital embedding $\iota \colon M_n \rightarrow \mathcal M^\omega \cap \{b\}'$.  Then
    \begin{equation}
        \|[a, x]\|_{2, X} \leq \|[a, x]\| = \|[a-b,x]\| \leq 2 \|a - b\| \|x\| \leq 2 \epsilon \|x\|
    \end{equation}
    for all $x \in M_n$.
\end{proof}

Traces on relative commutants can be difficult to compute in general.  It is a standard fact, however, that for a II$_1$ factor $\mathcal M$ and an inclusion $\mathcal R \subseteq \mathcal M^\omega$, the relative commutant $\mathcal M^\omega \cap R'$ has unique trace.  The following gives a tracially complete analogue of this statement.

\begin{lemma}\label{lem:factorial-commutant}
    Let $(\M,X)$ be a  factorial tracially complete $C^*$-algebra such that $\M^\omega$ has strict comparison with respect to $X^\omega$. Given a unital embedding $\iota \colon \mathcal R \subseteq  \mathcal M$, we have that $T(\M^\omega \cap \iota(\mathcal R)') = T(\M^\omega)=X^\omega$ (identified via the restriction map), and so $(\M^\omega \cap \iota(\mathcal R)',X^\omega)$ is a factorial tracially complete $C^*$-algebra.
\end{lemma}

\begin{proof}
    Set $Y \coloneqq \{\tau|_{\mathcal M^\omega \cap \iota(\mathcal R)'}: \tau \in X^\omega \}$ and note that $Y$ is a compact convex subset of $T(\mathcal M^\omega \cap \iota(\mathcal R)')$.  We will show $Y = T(\mathcal M^\omega \cap \iota(\mathcal R)'$ by showing that $\M\cap \iota(\mathcal R)'$ has strict comparison with respect to $Y$ and using Proposition~\ref{prop:computing-traces}.
    To this end, fix $n \in \mathbb N$ and $a, b \in M_n(\mathcal M^\omega \cap \iota(\mathcal R)')$ such that $d_\tau(a) < d_\tau(b)$ for all $\tau \in Y_n$. Without loss of generality we may assume $n=1$.

    Fix $\epsilon > 0$.  For each $\tau \in Y$,
    \begin{equation}
        d_\tau((a - \epsilon)_+) \leq d_\tau(a) < d_\tau(b) = \sup_{\delta > 0} d_\tau((b - \delta)_+).
    \end{equation}
    By Dini's theorem and the compactness of $Y$, there $\delta > 0$ such that 
    \begin{equation}
        d_\tau((a - \epsilon)_+) < d_\tau(b - \delta)_+), \quad \tau \in Y.
    \end{equation}
    Define a continuous function $f_\delta \colon \mathcal [0, \infty) \rightarrow [0, \infty)$ by
    \begin{equation}
        f_\delta(t) \coloneqq
        \begin{cases}
            t/ \delta, \quad &t \in [0, \delta); \\
            1, \quad &t \in [\delta, \infty).
        \end{cases}
    \end{equation}
    We'll show that for each matrix subalgebra $F \subseteq \mathcal R$, there is $v \in \mathcal M^\omega \cap \iota(F)'$ such that
    \begin{equation}\label{eq:central-traces}
        v^*v = (a - 2\epsilon)_+ \quad \text{and} \quad f_\delta(b) v = v
    \end{equation}
    Assuming this has been done, note that $\|v\|^2 \leq \|a\|$, so a reindexing argument produces $v \in \mathcal M^\omega \cap \iota(\mathcal R)'$ such that \eqref{eq:central-traces} holds.  This will imply $(a - 2 \epsilon)_+ \precsim b$ in $\mathcal M^\omega \cap \iota(\mathcal R)'$, and since $\epsilon > 0$ was arbitrary, this will imply $a \precsim b$ in $\mathcal M^\omega \cap \iota(\mathcal R)'$.
    
    Given $F \subseteq \mathcal R$, let $p \in F$ be a rank one projection.  Then there is an isomorphism
    \begin{equation}
        \phi \colon \mathcal M^\omega \cap \iota(F)' \rightarrow \iota(p) \mathcal M^\omega \iota(p) \colon a \mapsto \iota(p)a,
    \end{equation}
    which satisfies 
    \begin{equation}
        \tau(\phi(c))/\tau(\iota(p)) = \tau(c), \qquad \tau \in X^\omega,\ c \in \mathcal M^\omega \cap \iota(F)'
    \end{equation}
    Given $\tau \in X^\omega$, $\tau|_{\mathcal M^\omega \cap \iota(\mathcal R)'} \in Y$, so $d_\tau((a - \epsilon)_+) < d_\tau((b - \delta)_+)$, and therefore, $d_\tau(\phi((a - \epsilon)_+)) < d_\tau(\phi((b - \delta)_+))$.  As $\mathcal M^\omega$ has strict comparison with respect to tracial states, $\phi((a- \epsilon)_+) \precsim \phi((b- \delta)_+)$ in $\mathcal M^\omega$ and hence also in the corner $\iota(p)\mathcal M^\omega \iota(p)$.  So $(a - \epsilon)_+ \precsim (b -\delta)_+$ in $\mathcal M^\omega \cap \iota(F)'$.
    Hence there is $v \in \mathcal \mathcal M^\omega \cap \iota(F)'$
    such that
    \begin{equation}
        v^*v = (a - 2\epsilon)_+ \quad \text{and} \quad vv^* \in \overline{(b - \delta)_+(\mathcal M^\omega \cap \iota(F)') (b - \delta)_+}.        
    \end{equation}
    Then $f_\delta(b)vv^* = vv^*$ and the $C^*$-identity implies $f_\delta(b) v = v$, as required.
\end{proof}

The proposition below will give us access to a restricted form of CPoU.

\begin{proposition}\label{prop:commutant-cpou}
    Let $(\M,X)$ be a factorial tracially complete $C^*$-algebra such that $\M^\omega$ has strict comparison with respect to $X^\omega$. Given a unital embedding $\iota \colon \mathcal R \rightarrow \mathcal M$, 
    the tracially complete $C^*$-algebra $(\mathcal M^\omega \cap \iota(R)', X^\omega)$ is factorial and McDuff. In particular, it has CPoU (as defined in \cite[Definition~3.1]{carrion2023tracially}).
\end{proposition}

\begin{proof}
    The tracially complete $C^*$-algebra $(\M \cap \iota(\mathcal R)',X^\omega)$ is factorial by Lemma~\ref{lem:factorial-commutant}. Since $\mathcal R$ is McDuff, for any separable set $S \subseteq \mathcal M^\omega \cap \iota(\mathcal R)'$, we can find a sequence of unital embeddings $j_n\colon \mathcal R\to \iota(\mathcal R)\subseteq \mathcal M^\omega$ which are approximately central. Kirchberg's epsilon test then gives an embedding $j\colon\mathcal R\to \mathcal M^\omega\cap \iota(\mathcal R)'\cap S'$. Hence $(\mathcal M^\omega \cap \iota(\mathcal R)', X^\omega)$ is McDuff.
    The final statement follows from \cite[Proposition~5.22 and Theorem~1.4]{carrion2023tracially}.
\end{proof}

The following result is a generalisation (and a consequence) of \cite[Proposition~2.2]{ET:RR0}, where they obtain the same conclusion when $(\mathcal M, X)$ is a factorial tracially complete $C^*$-algebra with CPoU.

\begin{lemma}[{cf.\ \cite[Proposition~2.2]{ET:RR0}}]\label{lem:ETRR02.2}
    Let $(\M,X)$ be a II$_1$ factorial tracially complete $C^*$-algebra such that $\mathcal M^\omega$ has strict comparison with respect to $X^\omega$ and $\M^\omega$ has a divisible unit.  Let $\epsilon > 0$ and $a, p, c \in \mathcal M$ be positive contractions with $cp = p$ and $pa = a$.  Suppose $\|p^2 - p\|_{2, X} \leq \delta$ for some $\delta \in (0, 1/4]$.  Then there is a positive contraction $q \in \mathcal M$ such that 
    \begin{equation}\label{eq:rr0-lemma}
        \|cq - q\|_{2, X} < \epsilon,\ \|q a - a\|_{2, X} < \epsilon,\ \|q^2 - q\|_{2, X} < \epsilon,\ \text{and}\ \|p - q\|_{2, X} < \sqrt{5\delta}. 
    \end{equation}
\end{lemma}

\begin{proof}
    Let $b = a + p + c \in \mathcal M^\omega$.  Define $f, g, h \in C_0((0, 3])$ by
    \begin{equation}
        f(t) \coloneqq \min \{t, 1\}, \quad g(t) \coloneqq \min\{(t-1)_+, 1\}, \quad \text{and} \quad h(t)\coloneqq (t - 2)_+
    \end{equation} 
    for $t \in (0, 3]$.  Then we have $f(b) = a$, $g(b) = p$, and $h(c) = p$ since the $C^*$-algebra generated by $\{a,p,c\}$ is commutative and this holds for complex numbers $a$, $p$, and $c$.  So $a, p, c \in C^*(b)$.

    By Lemma~\ref{lem:central-divide-interval}, there is a unital embedding
    \begin{equation}
        \iota \colon \mathcal R \rightarrow \mathcal M^\omega \cap \{b\}' =\mathcal M^\omega \cap \{a, p, c \}'.
    \end{equation}  Proposition~\ref{prop:commutant-cpou}  then
    implies that $(\mathcal M^\omega \cap \iota(\mathcal R)', X^\omega)$ is a factorial tracially complete $C^*$-algebra with CPoU.
    Applying \cite[Proposition~2.2]{ET:RR0} to this tracially complete $C^*$-algebra (which contains $a,p$ and $c$) gives a positive contraction $q \in \mathcal M^\omega \cap \iota(\mathcal R)'$ satisfying \eqref{eq:rr0-lemma} with $Y$ in place of $X$.  Let $(q_n)_{n=1}^\infty \subseteq \mathcal M$ be a sequence of positive contractions lifting $q$.  Then \eqref{eq:rr0-lemma} holds with $q_n$ in place of $q$ for some $n \in \mathbb N$.
\end{proof}

It is shown in \cite[Theorem~2.5]{ET:RR0} that if $(\mathcal M, X)$ is a factorial tracially complete $C^*$-algebra with CPoU, then $\mathcal M$ has real rank zero; however, the proof (going through the proofs of \cite[Proposition~2.3 and Theorems~2.4 and 2.5]{ET:RR0}) only uses CPoU in order to establish the conclusion of the previous lemma (as \cite[Proposition~2.2]{ET:RR0}).
Hence copying these proofs, with Lemma~\ref{lem:ETRR02.2} in place of \cite[Proposition~2.2]{ET:RR0} yields the following result.
Together with Theorem~\ref{thm:classif-proj}, this proves \ref{it:Pureness.3}$\Rightarrow$\ref{it:Pureness.4} of Theorem~\ref{PurenessThm}.

\begin{theorem}\label{thm:rr0}
    Let $(\M,X)$ be a factorial tracially complete $C^*$-algebra such that $\M^\omega$ has strict comparison with respect to $X^\omega$ and $\M^\omega$ has divisible unit.  Then $\mathcal M$ has real rank zero.
\end{theorem}

To go from real rank zero and classification of projections of $\M$ to pureness, we want to apply Proposition~\ref{prop:ETpure}.
There is a small subtle point: classification of projections tells us that projections in $\M$ modulo Murray--von Neumann equivalence looks like $\Aff(X,[0,1])$, whereas to apply Proposition~\ref{prop:ETpure} we need to extend this to projections in all matrix algebras over $\M$.
The following enables this.

\begin{lemma}
\label{lem:RR0StableClassificationOfProjections}
    Let $(\M,X)$ be a factorial tracially complete $C^*$-algebra with classification of projections.
    Suppose $\M$ has real rank zero.
    Then $(M_n(\M), X_n)$ has classification of projections for every $n \in N$.
\end{lemma}

\begin{proof}
    Fix $n \in \N$. 
    Note that $X_n \cong X$ as topologically affine spaces. We make therefore identify $\Aff(X)$ and $\Aff(X_n)$.
    Let $f \in \Aff(X_n)$ with $0 \leq f \leq 1$. 
    Let $p \in \M$ be a projection with $\tau(p) = f(\tau)$ for all $\tau \in X$.
    Let $p_n = p \oplus \cdots \oplus p \in M_n(\M)$ be the $n$-fold direct sum of $p$.
    Then $(\tr_n \otimes \tau)(p_n) = \tfrac{1}{n}\cdot n \cdot \tau(p) = f(\tau)$ for all $\tau \in X$.
    Hence, by the definition of $X_n$, we see that $(M_n(\M), X_n)$ has existence of projections.

    Since $\M$ has real rank zero, it follows from \cite[Theorem 2.10]{BrownPedersen:RR0} that $M_n(\M)$ also has real rank zero. By \cite[Theorem 1.1]{Zhang:Riesz-Docomposition}, the partially ordered local semigroup $D(M_n(\M))$ of Murray--von Neumann equivalence classes of projections in $M_n(\M)$ has the Riesz decomposition property.\footnote{Zhang's result is given for $C^*$-algebras with the property that the self-adjoint elements of finite spectrum are dense in the self-adjoint elements, which is an equivalent formulation of real rank zero (\cite[Theorem 2.6]{BrownPedersen:RR0}).}
     
    Let $p,q \in M_n(\M)$ be projections with $\tau(p)\leq\tau(q)$ for all $\tau\in X_n$. Then $[p] \leq n[e_{11} \otimes 1_M]$ in $D(M_n(\mathcal M))$. By the Riesz decomposition property, there exist projections $p_1,\ldots,p_n,\in M_n(\M)$ such that  $[p] = \sum_{i=1}^n [p_i]$ and $[p_i] \leq [e_{11} \otimes 1_M]$. Since each $p_i \precsim e_{11} \otimes 1_M$ in $M_n(\M)$, we can assume without loss of generality that that each $p_i \in e_{11} \otimes \M$, which we identify with $\M$. Likewise we can write $[q]=\sum_{j=1}^n[q_j]$ for projections $q_j\in\M$. 

    Set $f_i(\tau)=\tau(f_i)$ and $g_j=\tau(q_j)$ for $\tau \in X$. As $(\M,X)$ is factorial, $X$ is a Choquet simplex. Therefore, we can apply Riesz decomposition for $\Aff(X)$ in the form of Remark \ref{Rem:MoreRieszDecomp} to find $f_{i,j}\in \Aff(X)$ with $0\leq f_{i,j}$ and $f_i=\sum_{j=1}^nf_{i,j}$ and $\sum_{i=1}^nf_{i,j}\leq g_j$. Since $(\M,X)$ has classification of projections we can find projections $p_{i,j}\in \M$ with $\sum_{j=1}^np_{i,j}=p_i$ and $\tau(p_{i,j})=f_{i,j}(\tau)$ for all $\tau\in X$, as well as projections $q_{i,j}\in\M$ with $\sum_{i=1}^nq_{i,j}\leq q_j$ and $\tau(q_{i,j})=f_{i,j}(\tau)$ for all $\tau\in X$.  Then $p_{i,j}\sim q_{i,j}$ and hence, 
    \begin{equation}
[p]=\sum_{i=1}^n\sum_{j=1}^n[p_{i,j}]\leq \sum_{j=1}^n\sum_{i=1}^n[q_{i,j}]\leq [q],
    \end{equation}
    and $M_n(\M)$ has classification of projections.
\end{proof}

Note that the following will prove part of the final sentence of Theorem~\ref{PurenessThm} once its show \ref{it:Pureness.1}--\ref{it:Pureness.4} are equivalent.

\begin{proposition}\label{prop:MPure}
Let $(\M,X)$ be a II$_1$ factorial tracially complete $C^*$-algebra which has classification of projections and real rank zero.  Then $\M$ is pure and has stable rank one as a $C^*$-algebra.
\end{proposition}

\begin{proof}
    From Lemma~\ref{lem:RR0StableClassificationOfProjections}, we see that $V(\M)\cong \Aff T(X)_+$ as ordered semigroups, so pureness follows from Proposition~\ref{prop:ETpure}.  Also the computation of $V(\mathcal M)$ implies $\mathcal M$ has cancellation of projections, so it has stable rank one by Proposition~\ref{prop:RR0CancellationSR1}.
\end{proof}

\section{Uniform 2-norm comparison}\label{Sect8}

In this final section, we complete the proof of the regularity theorem for tracially complete $C^*$-algebras (Theorem~\ref{PurenessThm}) by proving that a II$_1$ factorial tracially complete $C^*$-algebra $(\mathcal M, X)$ has strict $(2, X)$-comparison and approximately constant rank functions if and only if $\mathcal M$ has strict comparison with respect to $X$ and has an almost divisible unit (\ref{it:Pureness.1}$\Leftrightarrow$\ref{it:Pureness.3} of Theorem~\ref{PurenessThm}). 
For the readers convenience, we shall reproduce Theorem~\ref{PurenessThm} at the end of the section, and provide a full proof, which mostly consists of references to the relevant results in the paper.  

We begin by recalling the uniform 2-norm variants of Cuntz subequivalence and strict comparison. 

\begin{definition}[cf.\ Definition~\ref{def:traciallyPure}]\label{def:tracial-subequiv}
    Let $A$ be a $C^*$-algebra and let $X \subseteq T(A)$ be a compact convex set.  For $n \in \mathbb N$ and $a, b \in M_n(A)_{+}$, write $a \precsim_X b$ if there is a sequence $(v_m)_{m=1}^\infty \subseteq M_n(A)$ such that 
    \begin{equation} 
        \lim_{m \rightarrow \infty} \|v_m^* b v_m - a \|_{2, X} = 0.
    \end{equation}
    We say $A$ has \emph{strict $(2, X)$-comparison} if for all $n \in \mathbb N$ and $a, b \in M_n(A)_+$ with $d_\tau(a) < d_\tau(b)$ for all $\tau \in X$, we have $a \precsim_X b$.
\end{definition} 

The following is an easy consequence of Kaplansky's density theorem for tracially complete $C^*$-algebras (\cite[Proposition~3.28]{carrion2023tracially}), which allows us to restrict our attention to tracially complete $C^*$-algebras by replacing a $C^*$-algebra with its tracial completion. 

\begin{proposition}\label{prop:dense-subalgebra}
	Let $A$ be a $C^*$-algebra, $X \subseteq T(A)$ be a compact convex subset, and $\iota \colon A \rightarrow \completion{A}{X}$ denote the canonical inclusion. Let $a, b \in A_+$. Then $a \precsim_X b$ in $A$ if and only if $\iota(a) \precsim_X \iota(b)$ in $\completion{A}{X}$.
\end{proposition}

\begin{proof} 
    Fix $a, b \in A_+$. 
    Since $\iota$ is a $\|\cdot\|_{2,X}$-preserving $^*$-homomorphism, it is clear that $a \precsim_X b$ in $A$ implies $\iota(a) \precsim_X \iota(b)$.
    
    Suppose $\iota(a) \precsim_X \iota(b)$ and $\epsilon > 0$.  Let $\bar v \in \completion{A}{X}$ be such that $\| \bar v^* \iota(b) \bar v - a \|_{2, X} < \epsilon$.
	Using the tracially complete analogue of Kaplansky's density theorem (\cite[Proposition~3.28]{carrion2023tracially}),  there is $v \in A$ with $\|v\| \leq \|\bar v\|$ such that $\|\iota(v) - \bar v\|_{2, X} < \epsilon/ (\|b\| \|v\| + 1)$.  Then $\|v^* b v - a \|_{2, X} < 3 \epsilon$, so $a \precsim_X b$.
\end{proof}

As with the usual Cuntz subequivalence relation, the sequence $(v_m)_{m=1}^\infty$ in Definition~\ref{def:tracial-subequiv} need not be bounded. R{\o}rdam's lemma (\cite[Proposition 2.4]{Rordam92}) is typically used to work around this for the usual Cuntz subequivalence. Here, we develop an analogous tool for the uniform 2-norm setting. We begin with the following lemma.

\begin{lemma}\label{lem:controlled-subequiv1}
	Let $A$ be a $C^*$-algebra and let $T(A)$ be a compact convex subset.  For all $a, b \in A_+$ and $w \in A$, there is $v \in A$ such that $\|v\| \leq \|a\|^{1/2}$, $vv^* \in \overline{bAb}$, and 
	\begin{equation}
		\quad \|v^*v - a\|_{2, \tau} \leq  \|wbw^* - a \|_{2, \tau}^{1/2}, \quad \tau \in T(A). 
	\end{equation}
\end{lemma}

\begin{proof}
	Define $g \colon \mathbb R \rightarrow \mathbb R$ by
	\begin{equation}
		g(t) \coloneqq
		\begin{cases} 
			-\|a\|^{1/2}, & t \in \left(-\infty, \|a\|^{1/2}\right), \\ 
			t, & t \in \left[-\|a\|^{1/2}, \|a\|^{1/2}\right], \\
			\|a\|^{1/2}, & t \in \left(\|a\|^{1/2}, \infty\right), 
		\end{cases}
	\end{equation}
	and note that $g$ is continuous and odd.  By approximating $g(t)$ by linear combinations of monomials of odd degree, one can see that $g\left( \begin{pmatrix} 0 & w b^{1/2} \\ b^{1/2}w^* & 0 \end{pmatrix} \right)$ has the form 
	\begin{equation} 
		g\left( \begin{pmatrix} 0 & w b^{1/2} \\ b^{1/2}w^* & 0 \end{pmatrix} \right) = \begin{pmatrix} 0 & v^* \\ v & 0 \end{pmatrix} 
	\end{equation}
    for some $v \in A$.
	Note also that\footnote{When $g$ is a polynomial, write $g(t) = th(t^2)$ for a polynomial $h$.  Then using $g(t)^2=t^2h(t^2)^2$, we have
    \begin{equation}
    \begin{pmatrix} v^*v & \\ & vv^* \end{pmatrix} =
    \begin{pmatrix} wbw^* & \\ & b^{1/2}w^*wb^{1/2} \end{pmatrix} h\left(\begin{pmatrix} wbw^* & \\ & b^{1/2}w^*wb^{1/2} \end{pmatrix} \right)^2 \end{equation}
    and so $v^*v = wbw^* h(wbw^*)^2 = g((wbw^*)^{1/2})^2$ and likewise for the other one.}
	\begin{equation}
		v^*v = g((wbw^*)^{1/2})^2 \quad \text{and} \quad vv^* = g((b^{1/2} w^*w b^{1/2})^{1/2})^2.
	\end{equation}
	Then $\|v\| \leq \|a\|^{1/2}$ and $vv^* \in \overline{bAb}$.  To see the final estimate, note that for all $\tau \in T(A)$,
	\begin{equation}
		\begin{split}
			\|v^*v -a\|_{2, \tau}
			&= \|g((wbw^*)^{1/2})^2 - g(a^{1/2})^2 \|^2_{2, \tau} \\
			&\leq \| w b w^* - a\|_{2, \tau},
		\end{split}
	\end{equation}
	using Connes' joint distribution trick (\cite[Proposition~1.1]{Connes:Annals}) and noting that the function $G:[0, \infty) \rightarrow \mathbb R$ given by $t \mapsto g(t^{1/2})^2$ satisfies
	\begin{equation}
		G(t) = \begin{cases} 
			t, & t \in \left[0, \|a\|\right], \\ 
			\|a\|, & t \in \left(\|a\|, \infty\right),
		\end{cases}
	\end{equation}
    and so if $H,K \in L^2(\Omega,\mu)_+$ for a Radon measure space $(\Omega,\mu)$ then it follows that $\|G(H)-G(K)\|_{2,\mu} \leq \|H-K\|_{2,\mu}$. 
\end{proof}

  In the sequel, we make use of the continuous functions $f_\gamma \colon [0, \infty)\rightarrow [0,1]$ given by 
  \begin{equation}
        f_\gamma(t)\coloneqq \begin{cases}
            \frac{1}{\gamma}t, & t \in [0, \gamma],\\
            1, & t \in (\gamma, \infty).
        \end{cases}
    \end{equation}
    Note that $\lim_{\gamma \to 0} f_\gamma(t)t = t$ for all $t \in [0,\infty)$.
    
    The following result establishes equivalent characterisations of $(2, X)$-Cuntz subequivalence. This is closely related to R{\o}rdam's lemma in the operator norm setting; see \cite[Theorem 2.7]{GP_Cuntz_semi} and \cite[Proposition 2.4]{Rordam92}.
\begin{proposition}\label{prop:controlled-subequiv2}
	Let $(\mathcal M, X)$ is a tracially complete $C^*$-algebra. Let $a, b \in \mathcal M_+$. Then the following are equivalent:
	\begin{enumerate}
		\item \label{it:controlled-subequiv2-1} The subequivalence $a \precsim_X b$ holds.
		\item \label{it:controlled-subequiv2-2} For all $\epsilon > 0$, there is $v \in \mathcal M$ such that 
		\begin{equation} 
			\|v^*v - a \|_{2, X} < \epsilon \quad \text{and} \quad vv^* \in \overline{b\mathcal Mb}.
		\end{equation}
		\item \label{it:controlled-subequiv2-3} For all $\epsilon > 0$, there is $v \in \mathcal M$ such that 
		\begin{equation} 
			\|v\| \leq \|a\|^{1/2}, \quad  \|v^*v - a \|_{2, X} < \epsilon, \quad \text{and} \quad vv^* \in \overline{b\mathcal Mb}.
		\end{equation}
		\item\label{it:controlled-subequiv2-4} For all $\epsilon > 0$, there is $\delta > 0$ and $v \in \mathcal M$ such that
		\begin{equation} \label{it:controlled-eq-4}
			\|v\| \leq \|a\|^{1/2}, \quad  \|v^*v - a \|_{2, X} < \epsilon, \quad \text{and} \quad vv^* \in \overline{(b-\delta)_+\mathcal M(b - \delta)_+}.
		\end{equation}
	\end{enumerate}
\end{proposition}

\begin{proof}
	The implications \ref{it:controlled-subequiv2-4}$\Rightarrow$\ref{it:controlled-subequiv2-3}$\Rightarrow$\ref{it:controlled-subequiv2-2} are immediate.  We'll show \ref{it:controlled-subequiv2-2}$\Rightarrow$\ref{it:controlled-subequiv2-1}$\Rightarrow$\ref{it:controlled-subequiv2-4}.  
	
	Assume \ref{it:controlled-subequiv2-2} holds and fix $\epsilon > 0$. Choose $v \in \mathcal M$ as in \ref{it:controlled-subequiv2-2}.
    Since the net $\{f_\gamma(b)\}_{\gamma>0}$ forms an approximate unit for the hereditary subalgebra $\overline{b\mathcal M b}$, there exists $\delta > 0$, such that $\|f_\delta(b) v - v\| <  \epsilon/(\|v\|+1)$. Since $f_\delta(b) \precsim b$, there exists $w \in \mathcal M$ with $\|w^* b w - f_\delta(b)\| < \epsilon/(\|v\|^2+1)$.  Then $wv \in \mathcal M$ and satisfies
	\begin{equation}
		\begin{split}
			\|(wv)^* b (wv) - a\|_{2, X}
			&\leq \|v\|^2 \|w^*b w - f_\delta(b)\|_{2, X} + \|v\| \|f_\delta(b) v - v\| \\
			&\quad + \|v^*v - a\|_{2, X} \\
			& < 3\epsilon.
		\end{split}
	\end{equation}
	Therefore, $a \precsim_X b$.
	
	Now assume \ref{it:controlled-subequiv2-1} holds.  Given $\epsilon > 0$, there is $w \in \mathcal M$ such that $\|w^*b w - a\|_{2, X} < \epsilon$.  Then there is $\delta > 0$ such that 
	\begin{equation}
		\|w^*(b - \delta)_+ w - a \|_{2, X} < \epsilon.
	\end{equation}
	Apply Lemma~\ref{lem:controlled-subequiv1} with $a, (b - \delta)_+\in  \mathcal M_+$ and $w^*\in \mathcal M$, to obtain $v \in \mathcal M$ satisfying \eqref{it:controlled-eq-4}.
\end{proof}

Although we shall not need it in this section, it is likely to be of independent interest that the relation $\precsim_X$ is transitive.
\begin{proposition}\label{prop:subequiv-transitive}
	If $(\mathcal M, X)$ is a tracially complete $C^*$-algebra and $a, b, c \in \mathcal M_+$ with $a \precsim_X b$ and $b \precsim_X c$, then $a \precsim_X c$.
\end{proposition}

\begin{proof}
	Since $a\precsim_X b$, using Proposition~\ref{prop:controlled-subequiv2}\ref{it:controlled-subequiv2-4},  for $\epsilon > 0$,  there exists $\delta > 0$ and $v \in \mathcal M$, such that 
    \begin{equation} \label{110_trans_eq_1}
			\|v\| \leq \|a\|^{1/2}, \quad  \|v^*v - a \|_{2, X} < \epsilon, \quad \text{and} \quad vv^* \in \overline{(b-\delta)_+\mathcal M(b - \delta)_+}.
		\end{equation}
    Moreover, as $b\precsim_X c$,  let $w \in \mathcal M$ be such that 
	\begin{equation}\label{110_trans_eq_2}
		\|w^* c w - b \|_{2, X} < \delta \epsilon. 
	\end{equation}
	Define the continuous function $h \colon [0, \infty) \rightarrow [0, \infty)$ by
	\begin{equation}
		h(t) \coloneqq \begin{cases} \delta^{-1/2}, & t \in [0, \delta), \\t^{-1/2}, & t \in (\delta, \infty), \end{cases}
	\end{equation}
	so that $h(t)^2 t = f_\delta(t)$, for all $t \in [0, \infty)$. Since $f_\delta(b)(b-\delta)_+= (b-\delta)_+$, we see that  $f_\delta(b)v=v$. As $\|v\|\leq \|a\|^{1/2} $, $w h(b) v \in \mathcal M$ satisfies
	\begin{equation}
		\begin{array}{rcl}
			\| (wh(b)v)^* c (wh(b)v) - a \|_{2, X}
			&\leq& \|v^* h(b) (w^*cw - b) h(b) v \|_{2, X}\\
            & &\quad\quad + \| v^*h(b)^2 b v - a \|_{2, X} \\
            &\leq&  \|v\|^2 \|h\|^2 \|w^*c w - b \|_{2, X} \\
            & &\quad \quad +   \| v^* f_\delta(b)v - a \|_{2, X}\\
            & \stackrel{\eqref{110_trans_eq_2}}\leq& \|a\|\delta^{-1}(\delta\epsilon)  + \| v^*v - a \|_{2, X} \\
			&\stackrel{\eqref{110_trans_eq_1}}<& (\|a\| + 1)\epsilon,
		\end{array}
	\end{equation}
	which shows that $a \precsim_X c$.
\end{proof}

It is immediate from Proposition~\ref{prop:dense-subalgebra} that is a $C^*$-algebra $A$ has strict $(2, X)$-comparison for a compact convex set of traces $X \subseteq T(A)$, then the tracial completion $\completion{A}{X}$ has strict $(2, X)$-comparison.  The following result gives a converse. 

\begin{proposition}
	Let $A$ be a $C^*$-algebra, let $X \subseteq T(A)$ be a compact convex set of traces, and let $\mathcal M = \completion{A}{X}$, so that $(\mathcal M, X)$ is a tracially complete $C^*$-algebra.  Then $A$ has strict $(2, X)$-comparison if and only if $(\mathcal M, X)$ has strict $(2, X)$-comparison.
\end{proposition}

\begin{proof}
	Since the canonical map $A \rightarrow \mathcal M$ is isometric with respect to the uniform 2-norms, we may replace $A$ with its image under this map and identify $A$ with a $\|\cdot\|_{2, X}$-dense $C^*$-subalgebra of $\mathcal M$.
	
	First, suppose that $(\mathcal M, X)$ has strict $(2, X)$-comparison. Fix $n \in \mathbb N$ and let $a, b \in M_n(A)_+$ be such that $d_\tau(a) < d_\tau(b)$ for all $\tau \in X$.  Then $a \precsim_{X} b$ in $M_n(\mathcal M)$, and hence  $a \precsim_{X} b$ in $M_n(A)$ by Proposition~\ref{prop:dense-subalgebra}.
	
	Conversely, suppose that $A$ has strict $(2, X)$-comparison. Let $n \in \mathbb N$ and $a, b \in M_n(\mathcal M)_+$ be such that $d_\tau(a) < d_\tau(b)$.  Fix $\epsilon > 0$.  Then for all $\tau \in X$ and $\epsilon > 0$, there is $\delta > 0$ such that
	\begin{equation}
		\tau(f_\epsilon(a)) < \tau(f_\delta((b-\delta)_+)),\quad\tau\in X.
	\end{equation}
	Let $w \in M_n(\mathcal M)$ be such that
	\begin{equation}\label{0110_2normEq_00}
		\|w^* b w - f_\delta(b)\| < \epsilon.
	\end{equation}
	Let $a_0, b_0 \in M_n(A)$ be such that 
	\begin{equation}\label{0110_2normEq_0}
		\|a_0\| \leq \|a\|, \quad \|a - a_0\|_{2, X} < \epsilon, \quad \|f_\delta(b) - f_\delta(b_0)\|_{2, X} < \epsilon,
	\end{equation}
	and
	\begin{equation}
		\tau(f_\epsilon(a_0)) < \tau(f_\delta((b_0 - \delta)_+)), \quad \tau \in X.
	\end{equation}
	Then for each $\tau \in X$, we have
	\begin{equation}
		d_\tau((a_0 - \epsilon)_+) \leq \tau(f_\epsilon(a_0)) < \tau(f_\delta((b_0-\delta)_+)) < d_\tau((b_0 - \delta)_+),
	\end{equation}
	so $(a_0 - \epsilon)_+ \precsim_{X} (b_0 - \delta)_+$ in $M_n(A)$.  By Proposition~\ref{lem:controlled-subequiv1}, we can choose $v \in M_n(A)$ with $\|v\|^2 \leq \|(a_0 - \epsilon)_+\|$ such that  
	\begin{equation}\label{0110_2normEq_1}
		 \|v^* v - (a_0 - \epsilon)_+ \|_{2, X} < \epsilon \quad \text{and}\quad vv^* \in \overline{(b_0 - \delta)_+ M_n(A) (b_0 - \delta)_+}.
	\end{equation}
	In particular, we have
	\begin{equation}\label{0110_2normEq_2}
		\|v\|^2 \leq \|a\|, \quad \|v^*v - a\|_{2, X} < 2 \epsilon, \quad \text{and} \quad f_\delta(b_0) v = v.
	\end{equation}
	Now, $wv \in M_n(\mathcal M)$ satisfies
	\begin{equation}
		\begin{array}{rcl} 
			\|(wv)^* b (wv) - a \|_{2, X} 
			&\leq& \|v^* w^* b w v - v^* f_\delta(b) v \|_{2, X} \\ 
			&&\quad  + \|v^* f_\delta(b)v - v^* f_\delta(b_0) v\|_{2, X} \\
            &&\quad + \|v^*v - a \|_{2, X} \\
			&\stackrel{\eqref{0110_2normEq_2}}<& \|v\|^2 \|w^*b w - f_\delta(b)\| \\
            &&\quad + \|v\|^2 \|f_\delta(b) - f_\delta(b_0)\|_{2, X} + 2 \epsilon \\
			&\stackrel{\eqref{0110_2normEq_00}, \eqref{0110_2normEq_0}}<& (2 \|a\| + 2) \epsilon. 
		\end{array} 
	\end{equation} 
	Therefore, $a \precsim_{X} b$ in $M_n(\mathcal M)$. 
\end{proof} 

The following is the main technical result of this section, giving a strong connection between strict comparison in operator norm and uniform 2-norm; in particular, they are equivalent for tracial ultrapowers.  

\begin{theorem}\label{Thm:StrictComparisonEquivalences}
	For a tracially complete $C^*$-algebra $(\mathcal M, X)$, the following are equivalent.
	\begin{enumerate}
		\item \label{0110_TCC_eq_1} $(\mathcal M, X)$ has strict $(2, X)$-comparison.
		\item \label{0110_TCC_eq_2} $(\mathcal M^\omega, X^\omega)$ has strict $(2, X^\omega)$-comparison.
		\item\label{0110_TCC_eq_3} $\mathcal M^\omega$ has strict comparison with respect to $X^\omega$.
	\end{enumerate}
\end{theorem}

\begin{proof}
	The implication \ref{0110_TCC_eq_3}$\Rightarrow$\ref{0110_TCC_eq_2} is trivial since for all $n \in \mathbb N$ and $a, b \in M_n(\mathcal M^\omega)_+$, $a \precsim b$ implies $a \precsim_{X^\omega} b$.  
	
	For the implication \ref{0110_TCC_eq_2}$\Rightarrow$\ref{0110_TCC_eq_1}, suppose $n \in \mathbb N$ and $a, b \in M_n(\mathcal M)_+$ with $d_\tau(a) < d_\tau(b)$ for all $\tau \in X$.  Then $d_\tau(a) < d_\tau(b)$ for all $\tau \in X^\omega$, so $a \precsim_{X^\omega} b$ by \ref{0110_TCC_eq_2}.  Hence for  $\epsilon > 0$, there is $v \in M_n(\mathcal M^\omega)$ with $\|v^* b v - a \|_{2, X^\omega} < \epsilon$.  Let $(v_k)_{k=1}^\infty \in M_n(\mathcal M)$ be a sequence representing $v$ and note that
	\begin{equation}
		\lim_{k \rightarrow \omega} \|v_k^*b v_k - a \|_{2, X} = \|v^* b v - a \|_{2, X^\omega} < \epsilon,
	\end{equation}
	so for some $k \in \mathbb N$, we have $v_k \in M_n(\mathcal M)$ and $\|v_k^*b v_k - a \|_{2, X} < \epsilon$.  This shows $a \precsim_{X} b$, so \ref{0110_TCC_eq_1} holds.
	
	For the final implication, \ref{0110_TCC_eq_1}$\Rightarrow$\ref{0110_TCC_eq_3}, suppose $(\mathcal M, X)$ strict $(2, X)$-comparison.  Suppose further $n\in \mathbb N$ and $a, b \in M_n(\mathcal M^\omega)_+$ such that $d_\tau(a) < d_\tau(b)$ for all $\tau \in X^\omega$.  Fix $\epsilon > 0$ and let $\delta > 0$ be such that 
	\begin{equation}
		\tau(f_\epsilon(a)) < \tau(f_\delta((b-\delta)_+)), \quad \tau \in X^\omega.
	\end{equation}
	 Let $(a_k)_{k=1}^\infty, (b_k)_{k=1}^\infty \subseteq M_n(\mathcal M)$  be  sequences representing $a$ and $b$, respectively. Then there is a set $K \in \omega$ such that 
	\begin{equation}
		\tau(f_\epsilon(a_k)) < \tau(f_\delta((b_k-\delta)_+)), \quad \tau \in X,\ k \in K.
	\end{equation}	
	For each $k \in K$, we have
	\begin{equation}
		d_\tau((a_k - \epsilon)_+) \leq \tau(f_\epsilon(a_k)) < \tau(f_\delta((b_k-\delta)_+)) \leq d_\tau((b_k - \delta)_+), \quad \tau \in X, 
	\end{equation}	
	and hence by strict $(2, X)$-comparison of $\mathcal M$, there is $v_k \in M_n(\mathcal M)$ such that
	\begin{equation}
		\|v_k\| \leq \|a\|^{1/2}, \quad \|v_k^*v_k - (a_k - \epsilon)_+ \|_{2, X_n} < \frac1k, \quad \text{and} \quad f_\delta(b_k) v_k = v_k.
	\end{equation}
	For $k \in \omega \setminus K$, set $v_k = 0$ and note that $(v_k)_{k=1}^\infty \in M_n(\mathcal M)$ is a bounded sequence.  If $v \in M_n(\mathcal M^\omega)$ denotes the induced element, then $v^*v = (a - \epsilon)_+$ and $f_\delta(b)v = v$.  This shows $(a - \epsilon)_+ \precsim f_\delta(b) \precsim b$.  As $\epsilon > 0$ was arbitrary, $a \precsim b$, which proves \ref{0110_TCC_eq_3}.
\end{proof}

Theorem~\ref{Thm:StrictComparisonEquivalences} was the final ingredient needed to prove the regularity theorem for tracially complete $C^*$-algebras, which we do now.  For the reader's convenience, we recall that statement here.

\begin{theorem}[Theorem~\ref{PurenessThm}]\label{PurenessThm-body} Let $(\M,X)$ be a II$_1$ factorial tracially complete $C^*$-algebra.  The following are equivalent:
\begin{enumerate}
    \item \label{it:Pureness.1-body} $(\M,X)$ has strict $(2, X)$-comparison and has approximately constant rank functions;
    \item \label{it:Pureness.3-body}  $\mathcal M$ has strict comparison with respect to $X$ and $\mathcal M$ has divisible unit;
    \item \label{it:Pureness.4-body}  $(\M,X)$ has classification of projections and $\mathcal M$ has real rank zero.
\end{enumerate}
In this case, $T(\mathcal M) = X$ and $\mathcal M$ is pure and has stable rank one.
\end{theorem}

\begin{proof}
    It is clear that \ref{it:Pureness.3-body} implies \ref{it:Pureness.1-body}.
    
    Suppose \ref{it:Pureness.1-body} holds.  Then by Theorem~\ref{Thm:StrictComparisonEquivalences}, $\mathcal M^\omega$ has strict comparison with respect to $X^\omega$.   Fix $k \in \mathbb N$.  Since approximately constant rank functions occur in $\mathcal M$, there are $n \in \mathbb N$ and $a \in \mathcal M_n(A)_+$ be such that
    \begin{equation}
        \frac1{k+1}< d_\tau(a) < \frac1k, \quad \tau \in X,  
    \end{equation}
    and hence also for $\tau \in X^\omega$, so strict comparison implies $ka \leq [1_{\mathcal M^\omega}] \leq (k +1) [a]$, and hence $[1_{\mathcal M^\omega}]$ is approximately divisible in $\mathcal M^\omega$.  Applying Proposition~\ref{prop:sc_divisible_tracial_ultra} and using Kirchberg's epsilon test, we get exact divisibility of $1_{\M^\omega}$.  Then $(\mathcal M, X)$ has classification of projections by Theorem~\ref{thm:classif-proj} and 
    $\mathcal M$ has real rank zero by Theorem~\ref{thm:rr0}.  So \ref{it:Pureness.4-body} holds.  

    Assuming \ref{it:Pureness.4-body}, Theorem~\ref{Class:FD} implies $\mathcal M$ has classification of finite-dimensional subalgebras and hence the unit is divisible, and Proposition~\ref{prop:rr0-comparison} implies $\mathcal M$ has strict comparison with respect to $X$. This shows \ref{it:Pureness.3-body}.

    Now assume $(\mathcal M, X)$ satisfies \ref{it:Pureness.1-body}, \ref{it:Pureness.3-body} and \ref{it:Pureness.4-body}.  It follows from \ref{it:Pureness.4-body} and Proposition~\ref{prop:MPure} that $\mathcal M$ is pure and has stable rank one.  By \ref{it:Pureness.3-body} and Proposition~\ref{prop:computing-traces}, $T(\mathcal M) = X$.
\end{proof}

Recall that a II$_1$ factorial tracially complete $C^*$-algebra is called \emph{regular} if the equivalent conditions of Theorem~\ref{PurenessThm-body} hold.  The following relates regularity, ultrapowers and tracial completions.

\begin{corollary}\label{cor:regular}
    Let $A$ be a $C^*$-algebra such that $T(A)$ is non-empty and compact.  Suppose $A$ has strict comparison with respect to $T(A)$ and the class of a strictly positive element in $A$ is almost divisible.  Then both $(\completion{A}{T(A)}, T(A))$ and $(A^\omega, \overline {T_\omega(A)})$ are regular.
    Further, a II$_1$ factorial tracially complete $C^*$-algebra $(\mathcal M, X)$ is regular if and only if $(\mathcal M^\omega, X^\omega)$ is regular.
\end{corollary}

\begin{proof}
    The tracial completion $(\completion{A}{T(A)}, T(A))$ is a factorial tracially complete $C^*$-algebra by \cite[Proposition~3.23(iv)]{carrion2023tracially}, and the divisibility condition on $A$ readily implies it is type II$_1$.  Since the tracial ultrapower of $(\completion{A}{T(A)}, T(A))$ is $(A^\omega, \overline{T_\omega(A)})$, it suffices to prove the second statement.
    
    Suppose $(\mathcal M, X)$ is regular.  Then $(\mathcal M^\omega, X^\omega)$ has approximately constant rank functions (realised by the same positive elements from matrices over $\mathcal M$), and by Theorem~\ref{Thm:StrictComparisonEquivalences}, $(\mathcal M^\omega, X^\omega)$ has strict $(2, X^\omega)$-comparison.  So $(\mathcal M^\omega, X^\omega)$ is regular.
    
    Conversely, suppose $(\mathcal M^\omega, X^\omega)$ is regular. Then, by Theorem~\ref{Thm:StrictComparisonEquivalences}, $(\mathcal M, X)$ has strict $(2, X)$-comparison, and by lifting positive matrices over $\mathcal M^\omega$ to sequences of positive matrices over $\mathcal M$, we see that $\mathcal M$ has approximately constant rank functions since $\mathcal M^\omega$ does.
\end{proof}

\end{document}